\documentclass[12pt,reqno]{amsproc}

\usepackage[sc]{mathpazo}\renewcommand{\mathbf}{\mathbold}
\normalfont
\usepackage[T1]{fontenc}
\usepackage{longtable}

\usepackage[margin=1in]{geometry}
\usepackage{comment}
\usepackage[small,bf]{caption}
\usepackage{scalerel,nicefrac}
\usepackage[all,cmtip]{xy}
\usepackage{tikz-cd}
\usepackage{array, makecell, multirow}

\usepackage{rotating, setspace}

\numberwithin{equation}{section}

\makeatletter
\def\l@subsection{\@tocline{2}{0pt}{2.5pc}{5pc}{}}
\def\l@subsubsection{\@tocline{2}{0pt}{5pc}{7.5pc}{}}

\makeatother

\usepackage{bbm}
\usepackage{listings}
\usepackage{geometry}
\usepackage{amsmath,amsthm,amssymb}
\usepackage{hyperref}
\hypersetup
{
	colorlinks,	%
	citecolor=blue,%
	filecolor=black,%
	linkcolor=red,%
	urlcolor=blue
}
\usepackage{cleveref}
\usepackage{color}
\usepackage{mathtools}
\usepackage{tikz-cd}
\usepackage{subfigure}
\usepackage{shuffle}
\usepackage{mathrsfs}
\usepackage{multicol}
\usepackage[utf8]{inputenc}
\usepackage{todonotes}
\usepackage{enumitem}

\creflabelformat{equation}{#2\textup{#1}#3}

\makeatletter
\DeclareRobustCommand{\cev}[1]{%
  {\mathpalette\do@cev{#1}}%
}
\newcommand{\do@cev}[2]{%
  \vbox{\offinterlineskip
    \sbox\z@{$\m@th#1 x$}%
    \ialign{##\cr
      \hidewidth\reflectbox{$\m@th#1\vec{}\mkern4mu$}\hidewidth\cr
      \noalign{\kern-\ht\z@}
      $\m@th#1#2$\cr
    }%
  }%
}
\makeatother

\newcommand{\tight}{\hspace{-.5em}}

\newcommand{\ip}[1]{\left\langle#1\right\rangle}
\newcommand{\set}[1]{\left\{#1\right\}}

\DeclareFontFamily{U}{matha}{\hyphenchar\font45}
\DeclareFontShape{U}{matha}{m}{n}{
	<5> <6> <7> <8> <9> <10> gen * matha
	<10.95> matha10 <12> <14.4> <17.28> <20.74> <24.88> matha12
}{}
\DeclareSymbolFont{matha}{U}{matha}{m}{n}
\DeclareMathSymbol{\swedge}         {2}{matha}{"5E}
\DeclareMathSymbol{\svee}           {2}{matha}{"5F}

\newcommand{\ps}[1]{\mkern-.25mu\mathbin{\left(\mkern-3.5mu\left({#1}\right)\mkern-3.5mu\right)}}

\newcommand{\Z}{\mathbb{Z}}
\newcommand{\N}{\mathbb{N}}

\newcommand{\R}{\mathbb{R}}

\newcommand{\E}{\mathbb{E}}
\newcommand{\D}{\mathbb{D}}

\newcommand{\fT}{\mathfrak{T}}

\newcommand{\cP}{{\mathbf{P}}}

\newcommand{\cF}{\mathcal{F}}

\newcommand{\cH}{\mathcal{H}}

\newcommand{\cG}{\mathcal{G}}
\newcommand{\cL}{\Jlipmap}
\newcommand{\cX}{\mathcal{X}}
\newcommand{\cO}{\mathcal{O}}

\newcommand{\cI}{\mathcal{I}}

\newcommand{\Lip}{\mathrm{Lip}}

\newcommand{\Sh}{\mathsf{Sh}} 

\newcommand{\bT}{\mathbf{T}}

\newcommand{\bx}{\mathbf{x}}
\newcommand{\by}{\mathbf{y}}
\newcommand{\bz}{\mathbf{z}}
\newcommand{\bs}{\mathbf{s}}
\newcommand{\bt}{\mathbf{t}}
\newcommand{\br}{\mathbf{r}}

\newcommand{\bc}{\mathbf{c}}
\newcommand{\ba}{\mathbf{a}}
\newcommand{\bb}{\mathbf{b}}
\newcommand{\bv}{\mathbf{v}}

\newcommand{\bg}{\mathbf{g}}
\newcommand{\bH}{\mathbf{H}}
\newcommand{\bone}{\mathbf{1}}
\newcommand{\bzero}{\mathbf{0}}
\newcommand{\bA}{\mathbf{A}}
\newcommand{\bI}{\mathbf{I}}

\newcommand{\bomega}{{\boldsymbol{\omega}}}
\newcommand{\bphi}{{\boldsymbol{\phi}}}
\newcommand{\bsigma}{{\boldsymbol{\sigma}}}
\newcommand{\btau}{{\boldsymbol{\tau}}}

\newcommand{\ev}{\mathrm{ev}}

\newcommand{\id}{\mathrm{id}}
\newcommand{\sgn}{\mathrm{sgn}}

\newcommand{\bid}{\text{{\bf id}}}

\newcommand{\domdim}{d}
\newcommand{\coddim}{n}

\newcommand{\obx}{\overline{\bx}}
\newcommand{\oby}{\overline{\by}}

\newcommand{\oV}{\overline{V}}
\newcommand{\oI}{\overline{I}}
\newcommand{\oJ}{\overline{J}}
\newcommand{\oPhi}{\overline{\Phi}}

\newcommand{\lippath}{\Lip(\square^1, V)}

\newcommand{\lipmap}{\Lip(\square^\domdim, V)}

\newcommand{\smmapman}{C^\infty(\square^\domdim, X)}
\newcommand{\smmapKman}{C^\infty(|K|, X)}
\newcommand{\smmapR}{C^\infty(\square^\domdim, \R^\coddim)}

\newcommand{\Jac}{\text{\rm J}}
\newcommand{\Jlippath}{{\Lip}_\Jac(\square^1, V)}
\newcommand{\tllippath}{{\Lip}_{\mathrm{tl}}(\square^1, V)}
\newcommand{\Jlipmap}{{\Lip}_\Jac(\square^\domdim, V)}
\newcommand{\olipmap}{{\Lip}(\square^\domdim, \oV_d)}

\newcommand{\JlipmapW}{{\Lip}_\Jac(\square^\domdim, W)}

\newcommand{\Linfmap}{L^\infty(\square^\domdim, \Lambda^\domdim V)}

\newcommand{\idpspvspace}[3]{\bT_{\id, #1}^{#2}\ps{#3}} 
\newcommand{\pspvspace}[3]{\bT_{#1}^{#2}\ps{#3}}
\newcommand{\psphspace}[3]{\bH_{#1}^{#2}\ps{#3}}
\newcommand{\pvspace}[3]{\bT_{#1}^{#2}(#3)}
\newcommand{\phspace}[3]{\bH_{#1}^{#2}(#3)}

\newcommand{\intdom}[3]{D^{#1}_{#3}} 
\newcommand{\cintdom}[4]{\D^{#1, #2}_{#3, #4}} 
\newcommand{\ocintdom}[3]{\overline{\D}^{#1, #2}_{#3}} 

\newcommand{\ordsubset}[2]{\cI^{#1}_{#2}} 
 
\newcommand{\met}{\mu}

\newcommand{\mapact}[3]{#1_{#2}^{#3}}

\newcommand{\gnorm}{N} 
\newcommand{\dgnorm}{\nu} 

\newcommand{\GLV}{\mathrm{GL}(V)}
\newcommand{\GLHV}{\mathrm{GL}(\psphspace{\domdim}{}{V})}

\newcommand{\bpi}{\mathbf{\pi}}

\newtheorem{lemma}{Lemma}

\newtheorem{proposition}{Proposition}
\newtheorem{corollary}{Corollary}
\newtheorem{theorem}{Theorem}
\theoremstyle{definition}
\newtheorem{definition}{Definition}
\newtheorem{example}{Example}
\newtheorem{remark}{Remark}

\newtheorem{innercustomthm}{Theorem}
\newenvironment{customthm}[1]
  {\renewcommand\theinnercustomthm{#1}\innercustomthm}
  {\endinnercustomthm}

\usepackage[foot]{amsaddr}

\makeatletter
\renewcommand{\email}[2][]{%
  \ifx\emails\@empty\relax\else{\g@addto@macro\emails{,\space}}\fi%
  \@ifnotempty{#1}{\g@addto@macro\emails{\textrm{(#1)}\space}}%
  \g@addto@macro\emails{#2}%
}
\makeatother

\title{A Topological Approach to Mapping Space Signatures}
\author{Chad Giusti}
\address[CG]{Department of Mathematical Sciences, University of Delaware, Newark, DE 19716}
\email{cgiusti@udel.edu}

\author{Darrick Lee}
\address[DL]{\'Ecole Polytechnique F\'ed\'erale de Lausanne (EPFL) , Laboratory for Topology and Neuroscience, CH-1015 Lausanne, Switzerland}
\email{darrick.lee@epfl.ch}

\author{Vidit Nanda}
\email{nanda@maths.ox.ac.uk}

\author{Harald Oberhauser}
\address[VN, HO]{Mathematical Institute, University of Oxford, Andrew Wiles Building, Radcliffe Observatory Quarter,
Woodstock Rd, Oxford OX2 6GG}
\email{oberhauser@maths.ox.ac.uk}

\begin{document}

\begin{abstract}
 A common approach for describing classes of functions and probability measures on a topological space $\mathcal{X}$ is to construct a suitable map $\Phi$ from $\mathcal{X}$ into a vector space, where linear methods can be applied to address both problems. The case where $\mathcal{X}$ is a space of paths $[0,1] \to \mathbb{R}^n$ and $\Phi$ is the path signature map has received much attention in stochastic analysis and related fields. In this article we develop a generalized $\Phi$ for the case where $\mathcal{X}$ is a space of maps $[0,1]^d \to \mathbb{R}^n$ for any $d \in \mathbb{N}$, and show that the map $\Phi$ generalizes many of the desirable algebraic and analytic properties of the path signature to $d \ge 2$. The key ingredient to our approach is topological; in particular, our starting point is a generalisation of K-T Chen's path space cochain construction to the setting of cubical mapping spaces. 
\end{abstract}

\maketitle
\vspace{-30pt}

\tableofcontents

\vspace{-24pt}
\section{Introduction}
  It is classical, when studying non-linear real-valued functions defined on a topological space $\cX$, to first obtain a convenient representation $\Phi: \cX \to \cH$ of $\cX$ in a graded, possibly infinite-dimensional, vector space $\cH$. One then hopes to approximate any $f:\cX \to \R$ lying in a large class of functions by linear functionals on $\cH$ restricted to the image $\Phi(\cX)$. In a precise dual sense, graded descriptions of probability measures on $\cX$ are given by the expected values of their pushforwards along $\Phi$. These two properties --- density of linear functionals in a given function class and the expected value characterization of measures --- are also known as \emph{universality} and \emph{characteristicness} of $\Phi$ respectively. In addition to these two analytic properties, it is often desirable to also capture the structure of $\cX$, if any, in its image $\Phi(\cX)$.

In this article we develop a representation $\Phi$ for when \[\cX \subset C(\square^\domdim, V)\] is a (sufficiently regular) set of continuous maps from the $d$-dimensional unit cube $\square^\domdim \coloneqq [0,1]^\domdim$ into a real, finite-dimensional vector space $V$. Restricted to $\domdim=0$ or $\domdim=1$, our map $\Phi$ reduces to well-known embeddings which we recall below; our main contribution is to generalize these ideas to the case where $\domdim \geq 2$.

\subsection{The Moment Map \texorpdfstring{$(d=0)$}{d=0}} \label{ssec:moment_map}
We identify $\cX \subset C(\square^0,V)\equiv V$, where $\square^0 \cong \{*\}$.
The {\em moment map}
\[
\Phi:V \to \pspvspace{0}{}{V},\quad\Phi(\bx) := \left(1,\bx,\ldots, \frac{\bx^{\otimes m}}{m!},\ldots\right)_{m \geq 0}
	\]
  embeds $V$ into its power series tensor algebra $\pspvspace{0}{}{V} := \prod_{m \geq 0}(V^{\otimes m})$.
  This embedding satisfies a host of desirable algebraic and analytic properties.
	On the algebraic side, we mention:
	\begin{enumerate}
		\item [{\bf (AL1)}] The map $\Phi$ constitutes a {group homomorphism} from the abelian group $(V,+)$ to the group-like elements in the tensor algebra in the sense that the identity $\Phi(\bx+\by) = \Phi(\bx) \otimes \Phi(\by)$ holds for all $\bx,\by \in V$. 
		\item [{\bf (AL2)}] Finite linear functionals $\ell:\pvspace{0}{}{V} \to \R$, where $\pvspace{0}{}{V} := \bigoplus_{m \geq 0}(V^{\otimes m})$ is the tensor algebra, form a {point-separating algebra} for the image of $\Phi$; in particular, for each $\bx \in V$ we have
		\[
		\ip{\ell,\Phi(\bx)} \cdot \ip{\ell',\Phi(\bx)} = \ip{\ell'',\Phi(\bx)},
		\]
		where $\ell''$ is the shuffle product between $\ell$ and $\ell'$.
		\item [{\bf (AL3)}] The moment map is {equivariant} with respect to the natural actions of the general linear group $\GLV$ on its domain and codomain: in other words,
		\[
		\Phi(A\bx) = \bA\Phi(\bx)
		\]
		for all invertible linear maps $A:V \to V$ and vectors $\bx \in V$.\footnote{Here the action of $\bA$ on $\pspvspace{0}{}{V}$ is defined as follows: $\bA(\bv_1 \otimes \cdots \otimes \bv_m)$ equals $A\bv_1 \otimes \cdots \otimes A\bv_m$ for elementary tensors; then we extend to the entire tensor algebra by linearity.} 
	\end{enumerate}
	In addition, $\Phi$ also satisfies three useful analytic properties where $\cX \subset V$ is assumed to be a compact subset for {\bf (AN2)} and {\bf (AN3)}:
	\begin{enumerate}
		\item [\bf (AN1)] The map $\bx \mapsto \Phi(\bx)$ is continuous and the sequence of tensors $\Phi(\bx)$ {decays factorially}, meaning that the component $\Phi_m(\bx) \in V^{\otimes m}$ is bounded as
		\[
		\|\Phi_m(\bx)\| \leq C^m/m!,
		\] 
        for some constant $C = C(\bx)$ that is independent of $m$ .
		\item [\bf (AN2)] Every continuous map $f:\cX \to \R$ can be uniformly approximated by linear functionals defined on $\Phi(\cX)$. Thus, $\Phi$ is {universal} --- for each such $f$ and real number $\epsilon > 0$, there exists a finite linear functional $\ell \in \pvspace{0}{}{V}$ with
		\[
		\sup_{x \in \cX} \big| f(\bx) - \ip{\ell,\Phi(\bx)} \big| < \epsilon.
		\]
		\item [\bf (AN3)] Finally, $\Phi$ is {characteristic} in the sense that the law $\mu$ of any bounded random variable $X$ taking values in $\cX$ is completely determined by its $\Phi$-moments. Namely, we have an injective map
		\[
		\mu \mapsto \E_{X \sim \mu}[\Phi(X)] := \left(1,\E[X], \ldots, \frac{\E[X^{\otimes m}]}{m!},\ldots\right)_{m \geq 0}.
		\]
	\end{enumerate}
  These three analytic properties are elementary consequences of the algebraic properties (via the Stone--Weierstrass theorem) or even just the definition of $\Phi$. Nevertheless, they form a central part in the toolbox of any analyst or probabilist: {\bf (AN1)} and {\bf (AN2)} generalize the classic Taylor expansion from smooth to continuous functions and {\bf (AN3)} expresses that the moments of a compactly supported measure are sufficient to characterize it (equivalently, the moment-generating function characterizes the law of any bounded vector-valued random variable). 
   \noindent
\subsection{The Path Signature \texorpdfstring{$(d=1)$}{d=1}} \label{ssec:path_signature}
When $\cX = \lippath \subset C(\square^1,V) $ is the set of Lipschitz continuous paths $\bx:\square^1 \to V$, we define the {\emph{path signature}}:
	\[
	\Phi:\lippath \to \pspvspace{1}{}{V},\quad \Phi(\bx) \coloneqq \left(1,\int_{\Delta^1_\pi} d \bx,  \ldots, \int_{\Delta^m_\pi} d \bx^{\otimes m},\ldots \right)_{\substack{m \geq 0\\\pi \in \Sigma_m}},
	\]
where $d \bx^{\otimes m} \coloneqq d \bx(t_1) \otimes \cdots \otimes d \bx(t_m)$, and the integral is the Riemann-Stieltjes integral over the permuted $m$-simplex 
\begin{align*}
\Delta^m_\pi \coloneqq \{0 \leq t_{\pi(1)} < \ldots < t_{\pi(m)} \leq 1 \} \subset \square^m,
\end{align*}
where $\pi \in \Sigma_m$, the symmetric group on $m$ elements. Here, $\pspvspace{1}{}{V} \coloneqq \prod_{m\geq 0} \prod_{\Sigma_m} V^{\otimes m}$ is the \emph{power series permutation tensor space}. This formulation is a variation of the classical definition of the path signature
\begin{align*}
    \Phi^{\id}:\lippath \to \pspvspace{\id, 1}{}{V}, \quad \Phi^{\id} \coloneqq \left(1,\int_{\Delta^1} d \bx,  \ldots, \int_{\Delta^m} d \bx^{\otimes m},\ldots \right)_{m \geq 0},
\end{align*}
which only uses the identity permutations, where $\Delta^m \coloneqq \Delta^m_{\id}$ is the standard $m$-simplex and $\pspvspace{\id, 1}{}{V} \coloneqq \prod_{m\geq 0} V^{\otimes m}$ is once again the ordinary power series tensor algebra. These two formulations of the path signature, $\Phi$ and $\Phi^\id$, provide the same information, but the permutation structure will be crucial for the generalization to $d \ge 2$. 

The path signature satisfies suitably modified versions of the aforementioned properties of the moment map, along with a new invariance property {\bf (AL0)}, which allows us to treat the path signature as an embedding on a space of equivalence classes of Lipschitz paths. In particular, the modified algebraic properties are:\footnote{While all of these properties can be equivalently stated for the permutation-augmented signature $\Phi$, we state the properties for the more familiar classical path signature $\Phi^\id$.}
\begin{itemize}
    \item [\bf (AL0)] The equality $\Phi^\id(\bx) = \Phi^\id(\by)$ holds for two paths $\bx, \by \in \lippath$ if and only if they differ by a specific type of reparametrisation~\cite{hambly_uniqueness_2010} (called {\em tree-like equivalence}, denoted $\bx \sim_\text{tl} \by$).
    \item [\bf (AL1)] The map $\Phi^\id$ constitutes a homomorphism between the group of tree-like equivalence classes of Lipschitz paths (equipped with the concatenation product) $(\tllippath,*)$ and the group-like elements in the tensor algebra, in other words, $\Phi^\id(\bx * \by) = \Phi^\id(\bx) \otimes \Phi^\id(\by)$ holds for all $\bx, \by \in \lippath$.
    \item [\bf (AL2)] Finite linear functionals $\ell \in \pvspace{\id, 1}{}{V}$, where $\pvspace{\id, 1}{}{V} \coloneqq \bigoplus_{m \geq 0} V^{\otimes m}$ is the tensor algebra, form a subalgebra of functions; for each $\bx \in \lippath$, we have
		\[
        \langle\ell,\Phi^\id(\bx)\rangle \cdot \langle\ell',\Phi^\id(\bx)\rangle = \langle\ell'',\Phi^\id(\bx)\rangle,
		\]
	where $\ell''$ is the shuffle product of $\ell$ and $\ell'$.
	\item [\bf(AL3)] The path signature is equivariant with respect to the natural action of $\GLV$ on $\lippath$: in other words,
	\[
	\Phi^\id(A\bx) = \bA \Phi^\id(\bx)
	\]
	for all $A \in \GLV$ and $\bx \in \lippath$. Furthermore, there is a natural $\Z_2$ action on $\lippath$ given by time reversal, and the path signature is also equivariant with respect to this action; in particular, we have
	\[
	\Phi^\id(\bx^\tau) = \Phi^\id(\bx)^{-1},
	\]
	where $\bx^\tau$ denotes time reversal, and the right side is given by the tensor inverse.
\end{itemize}

Crucially, the analytic properties {\bf (AN1)-(AN3)} are also satisfied as written by $\Phi$ provided that we restrict to compact subsets $\cX$ of the quotient $\tllippath$ by tree-like equivalence for {\bf (AN2)} and {\bf (AN3)}.
Compactness of $\cX$ is typically too strong an assumption in this context, since $\lippath$ is not even locally compact; however, a generic normalization can be used to derive a robust version of $\Phi$, which in turn allows one to drop the compactness hypothesis \cite{chevyrev_signature_2018}.

  \subsection{The Mapping Space Signature \texorpdfstring{$(d \ge 2)$}{d ge 2}}
  The main contribution of this paper is to provide and study a natural extension of the moment map and the path signature to the realm of maps $\cX \subset C(\square^d,V)$ from higher-dimensional cubes ($d\ge 2$) into $V$.
  Concretely, we introduce the {\bf mapping space signature}
  \begin{align}
  \label{eq:intro_ms_signature}
    \Phi: \lipmap \rightarrow \pspvspace{\domdim}{}{V}, \quad \Phi(\bx) =\left(1, \int_{\intdom{1}{\domdim}{\bpi}} \hat{d} \bx(\bt_1),  \ldots, \int_{\intdom{m}{\domdim}{\bpi}} \hat{d}\bx^{\otimes m}, \ldots    \right)_{\substack{m\ge 0\\ \bpi \in \Sigma_m^\domdim}}, 
  \end{align}
  where $\hat{d} x^{\otimes m} \coloneqq \hat{d} \bx(\bt_1) \otimes \cdots \otimes \hat{d} \bx(\bt_m)$, with $\hat{d}$ being the \emph{Jacobian minor operator}, which reduces to the ordinary differential in the case of $\domdim = 1$. 
  The codomain of $\Phi$ is a $\domdim$-dependent graded vector space $\pspvspace{\domdim}{}{V}$, which is a permutation-augmented variant of the tensor algebra $\pspvspace{0}{}{V}$. In particular, $\pspvspace{\domdim}{}{V} \coloneqq \prod_{m \geq 0} \prod_{\Sigma_m^d} \left(\Lambda^\domdim V\right)^{\otimes m}$, where its subfactors are indexed by $\domdim$-tuples $\bpi = (\pi_1,\ldots,\pi_\domdim)$ of elements of the symmetric group $\Sigma_m$.

In particular, for a Lipschitz continuous map $\bx \in \lipmap$, the $\bpi$ component of $\Phi(\bx)$ is a generalised iterated integral
	\[
	\Phi_m^\bpi(\bx) = \int_{\intdom{m}{\domdim}{\bpi}} \hat{d} \bx(\bt_1) \otimes \hat{d} \bx(\bt_2) \otimes \cdots \otimes \hat{d} \bx(\bt_m) \in \left(\Lambda^\domdim V\right)^{\otimes m}.
	\]
	Here $\hat{d}\bx$ is a top-dimensional differential form on $\square^\domdim$ valued in $\Lambda^\domdim V$. The domain of integration is defined as the product of permuted $m$-simplices: $\intdom{m}{\domdim}{\bpi} \coloneqq \Delta^{m}_{\pi_1} \times \cdots \times \Delta^m_{\pi_\domdim}$.

  The domain $\cX$ of $\Phi$ is the space $\lipmap$ of all Lipschitz-continuous maps from the standard $\domdim$-cube to a finite dimensional vector space $V$.
  In fact, we immediately show that $\Phi$ is invariant with respect to \emph{Jacobian equivalence}, $\sim_\Jac$, of maps (generalizing translation invariance in the $\domdim=1$ case), and we can instead view the domain of $\Phi$ to be
  \begin{align*}
      \Jlipmap \coloneqq \lipmap/\sim_\Jac,
  \end{align*}
  the space of Jacobian equivalence classes of Lipschitz maps.
  Naive generalizations of the $1$-variation and Lipschitz metrics from the path case, $\domdim=1$, induce a topology on $\cX$ for the general higher-dimensional case, $\domdim \ge 2$, that is too coarse to obtain even continuity of $\Phi$.  
  However, we introduce Jacobian variants of these metrics which measures absolute volume increments of map $\bx:\square^d \to V$. In particular, 
  \begin{itemize}
      \item the \emph{Jacobian $1$-variation} measures the total volume of a map analogous to how classical variation norms pick up the total length of a path;
      \item the \emph{Jacobian Lipschitz constant} measures the maximum infinitesimal volume of a map analogous to how the Lipschitz constant recovers the maximum infinitesimal length of a path.
  \end{itemize}
  The codomain of $\Phi$ is shown to be the Hilbert space $\psphspace{\domdim}{}{V} \subset \pspvspace{\domdim}{}{V}$ of finite norm elements.
  With these metrics on the domain and the Hilbert space norm on the codomain, we can show that the mapping space signature $\Phi$ satisfies suitably modified versions of {\bf (AL0) - (AL3)} and {\bf (AN1) - (AN3)}, outlined in the following section. 
  
\subsection{Outline}

The definition of the mapping space signature in~\Cref{eq:intro_ms_signature} is motivated by Chen's foundational work on iterated integral cochain models for path spaces and loop spaces~\cite{chen_iterated_1977}. We begin in~\Cref{sec:chen_construction} with a brief high level discussion of a cubical reformulation of Chen's construction for mapping spaces~\cite{patras_cochain_2003, ginot_chen_2010} which directly leads to our definition. We also provide a more detailed exposition in~\Cref{apx:cubical_chen_construction}.

Next, in~\Cref{sec:lipschitz}, we introduce Lipschitz mapping spaces, which is used as the domain for $\Phi$, as well as Jacobian operators, equivalence classes, and metrics which are heavily used throughout the article. \medskip

In~\Cref{sec: mapping space signature} to~\Cref{sec:ms_shuffle}, we formally define the mapping space signature and prove some structural analytic and algebraic properties, which are summarized in our first main theorem, which is stated formally in~\Cref{thm:formal_ms_properties}. Here, {\bf (AL0')} and {\bf (AL1')} denote modified versions of these properties from the path signature.

\begin{customthm}{A}[Informal]
\label{thm:ms_properties}
    The mapping space signature is/has:
    \begin{enumerate}
        \item [\bf (AL0')] coordinate-wise reparametrization invariant;
        \item [\bf (AL1')] a modified Chen's identity (for a restricted variant of the signature);
        \item [\bf (AL2)] a generalized shuffle product structure;
        \item [\bf (AL3)] equivariant with respect to both the hyperoctahedral group, $B_\domdim$, action on the domain and $\GLV$ action on the codomain.
    \end{enumerate}
\end{customthm}

We define the mapping space signature in~\Cref{sec: mapping space signature} along with continuity properties.

We address the higher dimensional analogue of {\bf (AL0')} in~\Cref{sec:ms_invariance}. In particular, we prove the above coordinate-wise reparametrization invariance property in~\Cref{prop:ms_reparametrization_invariance}, and discuss a class of maps, built from tree-like paths, which have trivial mapping space signature in~\Cref{ssec:trivial_signature}. While we do not obtain the full invariance structure of $\Phi$, we hope to address this in future work. 

\Cref{sec:equivariance} addresses {\bf (AL3)}, which can be completely generalized from the $\domdim=1$ setting. In particular, the $\Z_2$ time reversal action for paths generalizes to the hyperoctahedral group, $B_\domdim$, action on $\square^\domdim$, and equivariance is shown in~\Cref{prop:Bd_equivariance}. Furthermore, linear maps $A \in L(V,W)$ to another finite-dimensional vector space $W$ induce maps $A: \Jlipmap\rightarrow\JlipmapW$ and a map $\bA: \psphspace{\domdim}{}{V} \rightarrow \psphspace{\domdim}{}{W}$ between the corresponding domains of the signature. We show that the signature is compatible with these induced maps in~\Cref{prop:cod_equivariance}.

~\Cref{sec:ms_shuffle} develops the generalized shuffle product structure of {\bf (AL2)} in~\Cref{thm:ms_shuffle}, showing that finite linear functionals on $\psphspace{\domdim}{}{V}$ form a subalgebra of continuous functions on $\Jlipmap$.

Finally, \Cref{sec:ms_composition} develops a modified version of Chen's identity {\bf (AL1')}, and discusses the obstructions to a full generalization. For maps $\bx, \by \in \lipmap$, there are $\domdim$ natural concatenation products. We show in~\Cref{thm:composition} that for a restricted variant of the mapping space signature, called the \emph{identity mapping space signature} (introduced in~\Cref{ssec:identity_signature}), we can derive a modified Chen's identity for each concatenation product. This completes the proof of~\Cref{thm:ms_properties}. \medskip

In~\Cref{sec:ms_parametrized_injectivity} and~\Cref{sec:univ_char}, we focus on the analytic properties of universality {\bf (AN2)} and characteristicness {\bf (AN3)}. Because we do not have a full characterization of the invariances of $\Phi$\footnote{Suppose we can characterize an equivalence relation $\sim$ on $\lipmap$ such that for $\bx, \by \in \lipmap$, we have $\bx \sim \by$ if and only if $\Phi(\bx) = \Phi(\by)$. In this case, the signature $\Phi$ is injective on $\lipmap/\sim$, and we can directly show universality and characteristicness without appending the parametrization, as is done in~\cite{chevyrev_signature_2018}.}, we append the parametrization of maps,
\begin{align*}
    \iota: \lipmap \rightarrow \Lip(\square^\domdim, \oV_\domdim), \quad \iota(\bx)(\bs) \coloneqq (\bs, \bx),
\end{align*}
where $\oV_\domdim \coloneqq \R^\domdim \times V$, in order to obtain injectivity of the \emph{parametrized signature} $\oPhi \coloneqq \Phi \circ \iota$. These results are summarized in our second main theorem, stated formally in~\Cref{thm:parametrized_univ_char}.

\begin{customthm}{B}[Informal]
\label{thm:intro_univ_char}
    The normalized parametrized mapping space signature is:
    \begin{enumerate}
        \item [\bf (AN1)] continuous, injective and has factorial decay;
        \item [\bf (AN2)] universal to the space of continuous bounded functions on $\Jlipmap$;
        \item [\bf (AN3)] characteristic to the space of finite regular Borel measures on $\Jlipmap$.
    \end{enumerate}
\end{customthm}


We introduce and study the parametrized signature in~\Cref{sec:ms_parametrized_injectivity} and address the {\bf (AN1)} properties. We prove that it is injective on $\Jlipmap$ in~\Cref{thm:parametrized_injectivity_first}, and show that it is continuous in~\Cref{prop:parametrized_signature_continuity}.

Finally in~\Cref{sec:univ_char}, we discuss the universal {\bf (AN2)} and characteristic {\bf (AN3)} properties for the parametrized signature without compactness assumptions. In particular, we recall the method from~\cite{chevyrev_signature_2018} involving normalization and the specific choice of the \emph{strict topology} for continuous bounded functions, concluding with a proof of~\Cref{thm:intro_univ_char}.

 \subsection{Related Work} 
We do not even attempt to summarize the use of classic monomials ($d=0$) in analysis but we give some pointers to the case of iterated path integrals ($d=1$) below. 

The study of such iterated integrals goes back at least to work of Volterra in the nineteenth century. Certain aspects of the path signature map $\bx \mapsto \Phi(\bx)$, which associates a path with a sequence of tensors given as iterated integrals, are studied under different names among different communities; these include but are not limited to algebraic topology, physics, number theory, control theory, and stochastic analysis.
Certain aspects and variations of this series of iterated integrals are studied under the name of Dyson series, time ordered integral, Magnus expansion, chronological exponential, Chen-Fliess series, Wiener-Ito chaos expansion, and possibly many others that we are not aware of. We refer to the signature of a path by $\Phi$ throughout this subsection.

Our general construction of $\Phi$ for general $\domdim \geq 2$ is heavily influenced by the use of the path signature $d=1$ by two communities in particular, namely topology and stochastic analysis/rough paths: 
\begin{enumerate}
\item
    K-T Chen~\cite{chen1973iterated} introduced and developed a cochain algebra model based on iterated integrals to determine the de Rham cohomology of the loop space $\Omega M$ of simply connected manifolds $M$ via variants of the Hochschild complex.  The fundamental construction here is an assignment (for all $m \geq 1$) of the type
    \begin{align*}\label{eq:Chen Lift}
	\left[\begin{matrix}
	\text{differential forms $(\omega_1,\ldots,\omega_m)$} \\
	\text{on $M$ with $\deg \omega_i = q_i$}
	\end{matrix}\right] \stackrel{\int}{\longrightarrow}
	\left[\begin{matrix}
	\text{a differential form $\int \omega_1 \cdots \omega_m$} \\
	\text{on $\Omega M$ of degree $\sum_i q_i - m$}
	\end{matrix}\right].
    \end{align*}
    However, our motivation is not to understand the de Rham cohomology of the loop space $\Omega M$. From this topological perspective, the path signature can be viewed as specific $0$-cochains ($0$-forms) of this model where $M = V$ is a finite-dimensional vector space. These $0$-forms correspond to the iterated integration of the standard $1$-forms on $V$. Rather, our motivation is to exploit this connection between Chen's construction and the path signature to introduce a higher dimensional analogues of the signature, and study its algebraic and analytic properties.
    
    Towards this end, our study begins with a generalization of Chen's construction to mapping spaces~\cite{patras_cochain_2003, ginot_chen_2010}, which provides an iterated integral cochain algebra model for mapping spaces into manifolds. This construction is facilitated by the higher Hochschild complex~\cite{pirashvili_hodge_2000, ginot_chen_2010}, and determines the de Rham cohomology of mapping spaces provided the manifold in the codomain is sufficiently connected. In particular, we develop a cubical variant of the simplicial constructions developed in~\cite{patras_cochain_2003, ginot_chen_2010}, and restrict our focus to finite-dimensional vector spaces $M = V$. By extracting certain $0$-cochains of this model corresponding to the iterated integration of standard $\domdim$-forms, we arrive at our notion of the mapping space signature.

\item
  Lyons~\cite{lyons_differential_1998} initiated the use of signatures in stochastic analysis in order to understand stochastic processes.
  This in turn builds upon many insights that were made by control theorists in the 1980's \cite{fliess1981, sussmann1986}, results from non-commutative algebra about free Lie algebras \cite{reutenauer93}, and expansions of stochastic differential equations \cite{benarous89}. 
  For example, the above two mentioned properties of universality and characteristicness of signatures play a central theme since they give a structured description of functions and measures on a genuinely infinite-dimensional space, namely a space of paths.  
  Consequently, many classical applications for measures on finite-dimensional spaces can be extended to path spaces; for instance, path signatures naturally extend the classical method of moments for parameter estimation from vector to path-valued random variables (i.e.~stochastic processes) \cite{papavasiliou2011}; they allow us to construct cubature formulas for probability measures on path \cite{kusuoka2003, cubature2004}; and they induce computable metrics for laws of stochastic processes~\cite{chevyrev_signature_2018}. 
The underlying algebraic properties, such as ({\bf{AL0})-(\bf{AL3}}) are essential to derive these results.
\end{enumerate}

Several works have studied extensions of rough path theory from $d=1$ to $d\ge 2$; in particular the case $d=2$ has received attention.
For example, \cite{chouk2014rough,chouk2015skorohod} extend the notion of controlled rough paths to the case $d=2$ to derive change of variable formula for Brownian sheets.
Similar in spirit, however with a more differential geometric focus are the works \cite{alberti2019integration,stepanov_towards_2020}, which develops an integration theory of rough differential forms.

Our motivation is very different from all these works: instead of directly trying to define a rough integration theory, we focus on the ``universal embedding'' that the signature $\Phi$ provides for smooth paths (resp.~smooth maps in the general $d\ge 1$ case). Lyons' original approach to rough path theory can be outlined in two steps:
\begin{enumerate}
    \item Establish fundamental analytic and algebraic properties of the signature.
    \item Take closures of the space of smooth paths, equipped with appropriate topologies, to handle non-smooth paths.
\end{enumerate}
For the case of $\domdim \ge 2$, the first step is already nontrivial, and forms the focus of this article. 
The difficulty is due to the fact that one can no longer rely on classical results of control theory and non-commutative algebra; even the definition of the codomain and the definition of an appropriate extension of variation metrics do not seem to follow from classical results.
Nevertheless, having a well-defined universal and characteristic embedding opens the door for taking closures in appropriate variation topologies, which would extend $\Phi$ to genuinely rough maps.
Although we do not pursue this second step here, we draw attention to the multi-dimensional extension of Young integration \cite{zust_integration_2011} which may become very useful for this task, analogous to the role that classical Young integration plays in the definition of classical rough path integrals \cite[Chapter 1]{lyons_differential_2007}. 

\addtocontents{toc}{\protect\setcounter{tocdepth}{1}}
{\footnotesize
	\subsection*{Acknowledgements}	
	We thank Ilya Chevyrev for helpful discussions. CG is supported by NSF-1854683 and AFOSR FA9550-21-1-0266. DL was funded by the United States Office of the Assistant Secretary of Defence Research and Engineering, ONR N00014-16-1-2010 and NSERC PGS-D3 scholarship. VN's work is supported by EPSRC grant EP/R018472/1. HO is supported by the Oxford-Man Institute and DataSig, EPSRC grant EP/S026347/1.
}

\clearpage
\addtocontents{toc}{\protect\setcounter{tocdepth}{2}}
\subsection{Notation}

Throughout this article, we will use boldface symbols to denote a vectors of objects such as $\bs = (s_1, \ldots, s_\domdim) \in \square^\domdim$. We fix $V$ to be a finite-dimensional real vector space, where $\dim(V) = \coddim$. Our main definitions and results are stated coordinate-free but for some arguments it is convenient to work in coordinates. We will assume that $V$ is equipped with an inner product structure, and if a basis is needed, we will let $\{e_1, \ldots, e_{\coddim}\}$ denote an orthonormal basis for $V$.

\renewcommand{\arraystretch}{1.2}
\begin{center}
\begin{longtable}{  m{0.13\textwidth}  m{0.73\textwidth} m{0.07\textwidth}  } 
  \Xhline{1pt}
  Symbol & Meaning & Page \\ 
  \Xhline{1pt}
  \multicolumn{3}{c}{General}\\
  \hline
    $\square^{\domdim}$ & the unit $d$-cube, $\square^{\domdim} \coloneqq [0,1]^d$ &  \\
    $\Delta^m$ & the standard $m$-simplex & \\
    $\Delta^m_\pi$ & the $\pi$-permuted $m$-simplex for $\pi \in \Sigma_m$ & \pageref{eq:permuted_m_simplex}\\
    $V$ & finite dimensional real vector (Hilbert) space $V \cong \R^n$ & \\
    $\oV_\domdim$ & augmented vector (Hilbert) space $\oV_\domdim \cong \R^\domdim \times V$ & \pageref{eq:ov_def}\\
    $\Lambda^\domdim V$ & degree $\domdim$ exterior algebra of $V$ & \pageref{eq:exterior_algebra}\\ 
    $[n]$ & the finite set $[n] \coloneqq \{1, \ldots, n\} $ &  \\ 
    $\cO_{\domdim, \coddim}$ & set of order-preserving injections $P: [\domdim] \rightarrow [\coddim]$ & \pageref{eq:order_preserving_injections}\\
    $\Sigma_m$ & the symmetric group on $m$ elements & \\
    $\Sh(p,q)$ & the set of $(p,q)$-shuffles; $\Sh(p,q) \subset \Sigma_{p +q}$ & \pageref{eq:shuffle_def}\\
    $\cev{\pi}$ & the reversal of a permutation $\pi \in \Sigma_m$ & \pageref{eq:reversal_def}\\
    $B_\domdim$ & hyperoctahedral group of $\domdim$-cube & \pageref{eq:hyperoctahedral_def}\\
    $\delta_\nu$ & the graded scaling by $\nu \in \R$; a linear map $\delta_\nu: \cH \rightarrow \cH$ & \pageref{eq:graded_scaling_def}\\
    $N$ & graded normalization & \pageref{eq:graded_normalization_def} \\
    
  \hline
  \multicolumn{3}{c}{Function Spaces}\\
  \hline
  $C(X,Y)$ & continuous functions &\\
  $C^\infty(X,Y)$ & smooth functions & \\
  $C_b(X,Y)$ & continuous bounded functions & \\
  $\Lip(X,Y)$ & Lipschitz continuous functions & \pageref{eq:lipschitz_mapping_def}\\
  $\Lip_\Jac(X,V)$ & Jacobian equivalence classes of Lipschitz functions & \pageref{eq:Jlipmap_def}\\
  \hline
  \multicolumn{3}{c}{Differentials and Jacobians}\\
  \hline
  $d \bx$ & differential of $\bx \in \lipmap$ & \pageref{eq:differential_def}\\
  $\hat{d} \bx $ & Jacobian minor operator of $\bx \in \lipmap$ & \pageref{eq:jacobian_minor_operator_d}\\
  $\hat{d}^P \bx $ & projected Jacobian minor operator of $\bx \in \lipmap$ for $P \in \cO_{\domdim,\coddim}$ & \pageref{eq:projected_jac_minor_d_def}\\
  $J[\bx_P]$ & Jacobian minor of $\bx \in \lipmap$ for $P \in \cO_{\domdim,\coddim}$ & \pageref{eq:jacobian_minor_def}\\
  $J[\bx]$ & Jacobian minor operator for Lipschitz maps & \pageref{eq:jacobian_minor_operator_J}\\
  $\mu_1$ & $1$-Jacobian variation metric & \pageref{def:jacobian_metrics} \\
  $\mu_\infty$ & Jacobian Lipschitz metric & \pageref{def:jacobian_metrics}\\
  \hline
  \multicolumn{3}{c}{Mapping Space Signature}\\
  \hline
  $\Phi$ & mapping space signature & \pageref{eq:mapping_signature_def}\\
  $\Phi_W$ & mapping space signature restricted to $W \subset \cO_{\domdim, \coddim}$ & \pageref{eq:signatureW_def}\\
  $\oPhi$ & parametrized mapping space signature & \pageref{eq:parametrized_signature_def}\\
  $S$ & path signature ($\Phi^\id$ for $\domdim=1$) & \pageref{eq:path_signature_def}\\
  $\intdom{m}{\domdim}{\bpi}$ & product of permuted $m$-simplices (integration domain of signature) & \pageref{eq:Dmd_domain}\\
  $\pvspace{\domdim}{(m)}{V}$ & degree $m$ permutation tensor space & \pageref{eq:degm_perm_tens_space_def}\\
  $\pspvspace{\domdim}{}{V}$ & power series permutation tensor space & \pageref{eq:ps_perm_tens_space_def}\\
  $\phspace{\domdim}{}{V}$ & permutation Hilbert space & \pageref{eq:perm_hilb_space_def}\\
  $\psphspace{\domdim}{}{V}$ & power series permutation Hilbert space & \pageref{eq:ps_perm_hilb_space_def} \\
  $\idpspvspace{d}{}{V}$ & identity tensor space & \pageref{eq:identity_ms_sig} \\
  $\bx^\btau_\sigma$ & action of $(\btau, \sigma) \in B_\domdim$ on $\bx \in \square^\domdim$ & \pageref{eq:Bd_lip_action}\\
  $\bpi^\btau_\sigma$ & action of $(\btau, \sigma) \in B_\domdim$ on $\bpi \in \Sigma_m^\domdim$ & \pageref{eq:Bd_perm_action} \\
  $\bx *_j \by$ & the $j$-composition of $\bx, \by \in \lipmap$ (or $\Jlipmap$) & \pageref{eq:ms_composition_formula}\\
  \hline
  \multicolumn{3}{c}{Topology (only used in~\Cref{sec:chen_construction} and~\Cref{apx:cubical_chen_construction})}\\
  \hline
  $\Omega^\bullet(X)$ & de Rham complex of $X$ & \pageref{eq:deRham_def}\\
  $\ordsubset{\domdim}{p}$ & $\domdim$-ordered subsets of $[p]$ & \pageref{def:d_ordered_subset}\\
  $d_i^\epsilon$ & cubical face map & \pageref{eq:face} \\
  $s_i$ & cubical degeneracy map & \pageref{eq:degeneracy}\\
  $g_i$ & cubical connection map & \pageref{eq:connection}\\
  $\delta_i^\epsilon$ & cubical coface map & \pageref{eq:coface}\\
  $\sigma_i$ & cubical codegeneracy map & \pageref{eq:codegeneracy}\\
  $\gamma_i$ & cubical coconnection map & \pageref{eq:coconnection}\\
  $|K|$ & geometric realization of a cubical set $K_\bullet$ & \pageref{eq:geometric_realization}\\
  $\eta_a$ & the evaluation map with respect to $a \in K_p$, $\eta_a : \square^p \rightarrow |K|$ & \pageref{eq:evaluation_map_a} \\
  $\eta_\bI$ & the evaluation map with respect to $\bI \in \ordsubset{\domdim}{p}$, $\eta_\bI : \square^p \rightarrow \square^\domdim$ & \pageref{eq:evaluation_map_I} \\
  $Z^\domdim_\bullet$ & standard cubical model for $\square^\domdim$ & \pageref{eq:Zbullet_def}\\
  $CH_\bullet^{K_\bullet}(A)$ & cubical Hochschild complex of $A$ over a cubical set $K_\bullet$ & \pageref{eq:cubical_hochschild_def}\\
  $\ev_{K_\bullet, p}$ & degree $p$ evaluation map with respect to $K_\bullet$ & \pageref{eq:ev_Kbullet_def}\\
  $\ev_{\domdim,p}$ & restricted degree $p$ evaluation map with respect to $Z^\domdim_\bullet$ & \pageref{eq:evaluation_dp_def}\\
  
  \Xhline{1pt}
\end{longtable}
\end{center}

\section{From Chen's Construction to Mapping Space Signatures}\label{sec:chen_construction}

The mapping space signature studied in this paper is motivated by returning to the topological origins of the path signature in Chen's iterated integral cochain construction~\cite{chen_iterated_1977}. Indeed, the path signature arises as certain $0$-cochains in Chen's construction for path spaces~\cite{giusti_iterated_2020}, which suggests the study of the $0$-cochains in a generalization of Chen's construction for mapping spaces~\cite{patras_cochain_2003, ginot_chen_2010} to develop an extension of the path signature for maps from higher dimensional domains. In this section, we provide a brief, high-level overview of a cubical reformulation of Chen's construction for mapping spaces and defer precise details to~\Cref{apx:cubical_chen_construction}. \medskip

Chen's iterated integral cochain construction provides a method to turn a collection of differential forms $\omega_1, \ldots, \omega_m \in \Omega^\bullet(X)$ on the manifold $X$ into a differential form $\int \omega_1 \ldots \omega_m \in \Omega^\bullet(\smmapman)$ on the smooth mapping space\footnote{Chen's construction is formally defined only for smooth mapping spaces, but we return to the lower regularity setting of Lipschitz maps in the remainder of the paper once the definition is established. Furthermore, the smooth structure on $\square^\domdim$ is given by the piecewise smooth structure with respect to the simplicial decomposition of the cube; see~\Cref{def:simplicial_smooth_structure} for details.} $\smmapman$ where $\Omega^\bullet(X)\label{eq:deRham_def}$ denotes the de Rham complex of $X$. This procedure uses a combinatorial model $Z^\domdim_\bullet$ of the domain $\square^\domdim$ in the form of a \emph{cubical set} (\Cref{def:cubical_set}), which provides a hierarchical description of $\square^\domdim$. Certain elements of this combinatorial model (the degeneracies of the top dimensional cube) are indexed by \emph{$\domdim$-ordered subsets of $[p]$} (\Cref{def:d_ordered_subset}, \Cref{lem:degeneracy_ordered_subsets}), denoted by $\bI = (I^{(1)}, \ldots, I^{(\domdim)})$, and consists of subsets $I^{(j)} \subset [p]$ for $j \in [d]$. These $\domdim$-ordered subsets are then associated to canonical \emph{evaluation maps with respect to $\bI$}, denoted by $\eta_\bI: \square^p \rightarrow \square^\domdim$, and defined by
\begin{align}
    \eta_\bI(s_1, \ldots, s_p) = \left( \max_{i \in I^{(1)}}\{s_i\}, \ldots, \max_{i \in I^{(\domdim)}}\{s_i\}\right).
\end{align}

Note that the unit $p$-cube, $\square^p$, can be decomposed into permuted simplices (\Cref{def:permuted_simplex}) as
\begin{align}
    \square^p = \coprod_{\pi \in \Sigma_p} \Delta^p_\pi.
\end{align}
Restricted to each permuted $p$-subsimplex, $\Delta^p_\pi$, the evaluation map $\eta_\bI$ with respect to $\bI$ is a fixed projection map due to the fixed ordering of the parameters in $\Delta^p_\pi$. 

For each $p \geq \domdim$, these evaluation maps for degenerate $p$-cubes allows us to stitch together differential forms on $X$ with total degree $q$ into a differential form on $\smmapman$ of degree $q-p$. The construction for obtaining $0$-forms on the mapping space is given by the following composition and is summarized in three steps below.
\begin{alignat}{3}
\label{eq:restricted_cubical_chen_construction}
    &\left(\Omega^\bullet(X)^{\otimes \#\ordsubset{\domdim}{p}}\right)_p &&\xrightarrow{\ev_{\domdim, p}^*} \Omega^\bullet(\square^p \times \smmapman) &&\xrightarrow{\int_{\square^p}} \Omega^{0}(\smmapman) \\
    &\hspace{32pt}\text{\small (\ref{enum:one})} && \hspace{10pt}\text{\small (\ref{enum:two})} && \hspace{8pt}\text{\small (\ref{enum:three})} \nonumber
\end{alignat}
Here $\ordsubset{\domdim}{p}$ denotes the set of $\domdim$-ordered subsets of $[p]$, $\#\ordsubset{\domdim}{p}$ denotes its cardinality and the subscript $p$ indicates the total degree of the tensor product of differential forms.

\begin{enumerate}
    \item \label{enum:one} For each $\bI$ in $\ordsubset{\domdim}{p}$, select a differential form $\omega_\bI \in \Omega^{q_\bI}(X)$ and set $\bomega = \bigotimes_{\bI} \omega_\bI \in \left(\Omega^\bullet(X)^{\otimes \#\ordsubset{\domdim}{p}}\right)_p$, where $\sum_{\bI} q_\bI = p$. This tensor product of forms contains two types of information --- a choice of nontrivial forms with total degree $p$, and a choice of $\domdim$-ordered subsets of $[p]$.
    
    \item \label{enum:two} Define the \emph{evaluation map}
    \begin{align}
    \label{eq:evaluation_dp_def}
        \ev_{\domdim,p}: \square^p \times \smmapman &\rightarrow X^{\#\ordsubset{\domdim}{p}} \\
        (\bs, \bx) &\mapsto \left(\bx \circ \eta_{\bI}(\bs)\right)_{\bI \in \ordsubset{\domdim}{p}},\nonumber
    \end{align}
    and compute the pullback of $\bomega$ along this map to get
    \begin{align*}
        \ev_{\domdim, p}^*\bomega(\bs, \bx) = \bigwedge_{\bI \in \ordsubset{\domdim}{p}} (\bx \circ \eta_\bI)^*\omega_\bI(\bs).
    \end{align*}
    
    \item \label{enum:three} Integrate the resulting differential form along $\square^p$ to obtain the following $0$-form on $\smmapman$,
    \begin{align}
    \label{eq:mapping_space_0_cochain}
        \int_{\square^p}\bigwedge_{\bI \in \ordsubset{\domdim}{p}} (\bx \circ \eta_\bI)^*\omega_\bI(\bs).
    \end{align}
\end{enumerate}

An explicit example of this construction is given in~\Cref{ex:0_cochain}. Our definition of the mapping space signature for the finite dimensional vector space $X = V \cong \R^\coddim$, where we use $\bv = (v_1, \ldots, v_n)$ for the coordinates of $V$, is obtained via this construction by specific choices of $\bomega \in \left(\Omega^\bullet(X)^{\otimes |\ordsubset{\domdim}{p}|}\right)_p$. In particular, we make the following restrictions.
\begin{enumerate}
    \item All nontrivial forms in $\bomega$ are standard $\domdim$-forms in $\R^\coddim$ so that the total degree is a multiple of $\domdim$; in particular, nontrivial forms are given by 
    \begin{align*}
        \omega_\bI = dv_{P(1)} \swedge \ldots \swedge dv_{P(\domdim)},
    \end{align*}
    where $P: [\domdim] \rightarrow [\coddim]$ is an order-preserving injection. Furthermore, this implies that the total degree of the forms is $p = m\domdim$, where $m \in \N$. Thus, given a map $\bx \in \smmapman$, the pullback in~\Cref{eq:mapping_space_0_cochain} is computed to be the Jacobian minor of $\bx$ with respect to $P$ (\Cref{def:jacobian_minor}), $\bx^*\omega_\bI(\bs) = J[\bx_P](\bs)$.
    \item These nontrivial forms are associated to specific $\domdim$-ordered subsets $\bI \in \ordsubset{\domdim}{p}$ such that the support of the pullback along $\ev_{d,p}$ is restricted to a product of permuted simplices. Explicit formulas for the $\bI$ is given in~\Cref{eq:restricted_d_ord_subsets}.
\end{enumerate}

These restrictions result in $0$-forms given by an iterated integral of Jacobian minors,
\begin{align}
\label{eq:chen_0_form_restricted}
    \int_{\intdom{m}{\domdim}{\bpi}} \prod_{i=1}^m J[\bx_{P_i}](\bt_i) d\bt,
\end{align}
where $\bpi = (\pi_1, \ldots, \pi_\domdim) \in \Sigma_m^\domdim$, the integration domain is $\intdom{m}{\domdim}{\bpi} \coloneqq \Delta^m_{\pi_1} \times \ldots \times \Delta^m_{\pi_\domdim}$, and the parametrization $\bt$ and $\bt_i$ are given in~\Cref{def:integration_domain}.  \medskip

These $0$-forms provide the basis for the definition of the mapping space signature, which is introduced in~\Cref{sec: mapping space signature} and studied for the remainder of the paper. In this section, we have used smooth mapping spaces in order to introduce and discuss Chen's cochain construction, which is naturally defined in the smooth setting. However, the specific $0$-forms obtained in~\Cref{eq:chen_0_form_restricted} can be easily defined in the setting of Lipschitz maps. Thus, we begin by introducing Lipschitz mapping spaces and corresponding metrics in the next section, and focus on this lower regularity setting for the remainder of the paper.

\section{Lipschitz Mapping Spaces and Jacobian Equivalence}\label{sec:lipschitz}
In this section, we introduce \emph{Lipschitz mapping spaces}, which will be the domain of the mapping space signature $\Phi$ defined in the next section. Among the desirable properties of a mapping space signature $\Phi$ that we have mentioned in the introduction are
\begin{enumerate}
\item continuity of the map $\bx \mapsto \Phi(\bx)$, and 
\item invariance properties.
\end{enumerate}
To address the first item we need to specify a topology on the domain and codomain of $\Phi$.
The codomain is a Hilbert space which induces a natural topology, but the topology on the domain $\lipmap \subset C(\square^\domdim,V)$ of $\Phi$ needs some care and is the topic of this section. 
Furthermore, we address the second item by studying a certain invariance that generalize the translation invariance of the classical $d=1$ case of paths in which
\begin{align}
 \Phi(T_\bv \bx) = \Phi(\bx),
\end{align}
where $T_\bv$ denotes translation of path by a vector $\bv \in \R^\coddim$, $T_\bv (\bx(t))_{t \in \square^1} = (\bv+\bx(t))_{t \in \square^1}$.
In this section, we introduce a useful generalization to $\domdim\ge 1$ of this invariance that we call \emph{Jacobian equivalence}.
The idea is as follows: in dimension $\domdim=1$ the (infinitesimal) increments of the path $x$ and $T_\bv \bx$ are the same, $d\bx =dT_\bv\bx$, and in dimension $\domdim \ge 2$ it is thus natural to ask for equivalence between mappings if they have the same (infinitesimal) ``volume increments'' which is exactly the definition of Jacobian equivalence.

\subsection{Lipschitz Mappings}
As we have stated in the previous section, Lipschitz maps are a lower regularity alternative to smooth maps in which the integrals of the desired $0$-forms (\Cref{eq:chen_0_form_restricted}) can easily be interpreted.
\begin{definition}
  Let $V$ be an inner product space. The \emph{Lipschitz mapping space} is
  \begin{align}
  \label{eq:lipschitz_mapping_def}
    \lipmap \coloneqq \left\{\bx \in C(\square^\domdim, V): \sup_{\bs,\bt \in \square^\domdim,\bs \neq \bt}\frac{\|\bx(\bt)-\bx(\bs)\|}{\|\bt-\bs\|}<\infty\right\}.
  \end{align}
\end{definition}
While the Lipschitz constant provides a natural semi-norm structure on $\lipmap$, the resulting topology is too coarse for our purposes. Below, we will introduce a metric structure on equivalence classes of Lipschitz maps, leading to a topology that is compatible with the mapping space signature. Before doing so, we introduce various formulations of the Jacobian minor operator extracts (projected) volume increments from Lipschitz maps.

\subsection{The Jacobian Minor Operator}
Given $\bx \in \lipmap$, we can consider the $V$-valued $1$-form $d \bx$.
By Rademacher's theorem, Lipschitz functions are differentiable almost
everywhere on the interior of $\square^{\domdim}$, and since the derivatives are bounded if they exist, we identify the differential as an $L^\infty$ function,
\begin{align}
\label{eq:differential_def}
  d\bx \in L^\infty( \square^{\domdim}, L(\R^\domdim, V)).
\end{align}
We introduce the \emph{Jacobian minor operator} which plays an important role throughout this article. We use $\Lambda^\domdim V \label{eq:exterior_algebra}$ to denote the degree $\domdim$ part of the exterior algebra of $V$.
\begin{definition}
  We call
  \begin{align}
    \label{eq:jacobian_minor_operator_d}
    \hat{d} : \lipmap &\rightarrow L^\infty \left( \square^\domdim, L\left(\Lambda^\domdim \R^\domdim, \Lambda^\domdim V\right)\right)\\
    \bx(\bs) & \mapsto \wedge^d d\bx(\bs), \nonumber
  \end{align}
  the \emph{Jacobian minor operator}.
\end{definition}
The Jacobian minor operator can be seen as a $\Lambda^\domdim V$-valued $\domdim$-form on $\square^\domdim$.
In order to simplify notation, we also make the identification $L(\Lambda^\domdim \R^\domdim, \Lambda^\domdim V) \cong \Lambda^\domdim V$, and simply view the Jacobian minor operator as a map\footnote{Explicitly, at each $\bs$ in $\square^d$, the map $\hat{d}\bx(\bs):\Lambda^\domdim\R^\domdim \to \Lambda^\domdim V$ sends the volume form $\alpha_1 \swedge \cdots \swedge \alpha_\domdim$ of $\R^\domdim$ to the $\domdim$-form
		$d\bx(\bs)(\alpha_1) \swedge \cdots \swedge d\bx(\bs)(\alpha_\domdim)$ on $V$.}
\begin{align*}
    \hat{d}: \lipmap \rightarrow L^\infty \left( \square^\domdim,  \Lambda^\domdim V\right).
\end{align*}
The operator $\hat d$ does not depend on a choice of basis of $V$ and throughout this article we state our main definitions and results in a basis independent way.
However, for some concrete computations it is useful to work in coordinates with respect to a basis of $V$ and in the next two sections we introduce the Jacobian minor $J$ which is simply a basis dependent version of the same map $\hat d$.

\subsection{The Projected Jacobian Minor Operator}
Let $\{e_1, \ldots, e_\coddim\}$ be a basis of $V$.
Then we can index a basis for $\Lambda^\domdim V$ using the collection of increasing sequences \[1 \leq P(1) < \ldots < P(\domdim) \leq \coddim,\] which we view as an order-preserving injection $P:[\domdim] \rightarrow [\coddim]$.
\begin{definition}
  Given a basis $\{e_1, \ldots, e_\coddim\}$ of $V$ we define 
\begin{align}
\label{eq:order_preserving_injections}
    \cO_{\domdim,\coddim} \coloneqq \{ P: [\domdim] \rightarrow [\coddim] \, : \, P \text{ is an order-preserving injection} \}.
\end{align}
In particular, given $P \in \cO_{\domdim,\coddim}$, denote the corresponding basis vector in $\Lambda^\domdim  V$ as
\begin{align}
    e^P \coloneqq e_{P(1)} \swedge \ldots \swedge e_{P(\domdim)}.
\end{align}
\end{definition}
\begin{definition}
  Given a basis $\{e_1, \ldots, e_\coddim\}$ of $V$, and a map $\bx \in \lipmap$, we use the coordinate-wise notation $\bx= (x_1, \ldots, x_\coddim)$.
For $P \in \cO_{\domdim,\coddim}$, we define
\begin{align*}
    \bx_P \coloneqq (x_{P(1)}, \ldots, x_{P(\domdim)}) : \square^{\domdim} \rightarrow \R^\domdim
\end{align*}
to be the projection onto the coordinates specified by $P$.
We denote the $e^P$-component of $\hat{d}\bx$ as 
\begin{align}
\label{eq:projected_jac_minor_d_def}
    \hat{d}^P \bx = dx_{P(1)} \swedge \ldots \swedge dx_{P(\domdim)}.
\end{align}
\end{definition}

\subsection{The Projected Jacobian Minor}
In the setting of Lipschitz maps, we explicitly formulate the Jacobian minor operator using determinants of the Jacobian matrix.

\begin{definition}
\label{def:jacobian_minor}
  Let $\bx \in \lipmap$ and $P \in \cO_{\domdim,\coddim}$. The \emph{Jacobian minor of $\bx$ with respect to $P$} is defined by
\begin{align}
\label{eq:jacobian_minor_def}
	J[\bx_P](\bs) \coloneqq
	\begin{vmatrix}
	\frac{\partial x_{P(1)}}{\partial s_1}(\bs) & \cdots & \frac{\partial x_{P(1)}}{\partial s_{\domdim}}(\bs) \\
	\vdots & \ddots & \vdots \\
	\frac{\partial x_{P(1)}}{\partial s_1}(\bs)   & \cdots & \frac{\partial x_{P(\domdim)}}{\partial s_{\domdim}}(\bs)
	\end{vmatrix}.
\end{align}
The $e^P$-component of $\hat{d}\bx$ is given by the Jacobian minor,
\begin{align}
    \hat{d}^P(\bx)(\bs) = J[\bx_P](\bs) \hat{d}\bs,
\end{align}
where $\hat{d}\bs \coloneqq ds_1 \swedge \ldots \swedge ds_{\domdim}$ and thus $\hat{d}^P(\bx)$ is the collection of Jacobian minors of $\bx$ with respect to the chosen basis on $V$. Thus, the Lipschitz formulation of the Jacobian minor operator is
\begin{align}
\label{eq:jacobian_minor_operator_J}
    J: \lipmap \rightarrow L^\infty(\square^{\domdim}, \Lambda^\domdim  V) \\
    \bx(\bs) \mapsto \left(J[\bx_P](\bs)\right)_{P \in \cO_{\domdim,\coddim}},\nonumber
\end{align}
where
\begin{align}
\label{eq:lipschitz_jacobian_minor}
    \hat{d}x(\bs) = J[x(\bs)] \hat{d}\bs.
\end{align}
\end{definition}
\begin{remark}
  In the Lipschitz setting, the two formulations of the Jacobian minor operator, $\hat{d}$ given in~\Cref{eq:jacobian_minor_operator_d} and $J$ given in~\Cref{eq:jacobian_minor_operator_J} are simply basis-independent and basis-dependent versions of the same map. However, the basis-independent formulation $\hat{d}$ suggests possible generalizations to lower regularity maps using generalizations of the Young integral to differential forms~\cite{zust_integration_2011, stepanov_towards_2020}. Thus, we will state the main definitions using the $\hat{d}$ operator, but we will primarily work with the basis-dependent formulation $J$, which allows us to perform explicit computations for Lipschitz maps.
\end{remark}

\subsection{Equivalence Classes of Lipschitz Maps}
The Jacobian minor operator can then be used to define an equivalence class of functions in $\lipmap$.
\begin{definition}
\label{def:jacobian_equivalence}
    Given two maps $\bx, \by \in \Lip(\square^{\domdim}, V)$, we say that $\bx$ and $\by$ are \emph{Jacobian equivalent}, denoted $\bx \sim_\Jac \by$, if $\hat{d}(\bx) = \hat{d}(\by)$ as elements of $L^\infty(\square^{\domdim}, V)$. We denote Jacobian equivalence classes of Lipschitz functions by
    \begin{align}
    \label{eq:Jlipmap_def}
        \Jlipmap \coloneqq \Lip(\square^{\domdim}, V)/\sim_\Jac.
    \end{align}
\end{definition}
By definition, the Jacobian minor operator is well-defined on Jacobian equivalence classes, and will be viewed as an injective map
\begin{align*}
    \hat{d}: \Jlipmap \rightarrow L^\infty(\square^{\domdim}, \Lambda^\domdim  V).
\end{align*}
\begin{example}[$d=1$]
  For a path $\bx \in \lippath$, the Jacobian minor operator reduces to the differential $\hat{d}(\bx) = d\bx$.
  Thus, in this case Jacobian equivalence of Lipschitz paths corresponds to translation equivalence. 
\end{example}

\begin{example}[$d \ge 2$]
\label{ex:jacobian_equivalence}
For higher-dimensional domains, Jacobian equivalence is much more complex than translation equivalence.
For example, there is a simple 1-parameter family of maps which are all Jacobian equivalent but not related by translations. Consider any nonzero $a \in \R$, and let $\bx^a: \square^2 \rightarrow \R^3$ be given by
	\begin{align*}
		\bx^a(s_1, s_2) = \left( as_1, \, \frac{1}{a} s_2,\,  -as_1 + \frac{1}{a}s_2 \right).
	\end{align*}
	The Jacobian of this map is
	\begin{align*}
		d\bx^a(\bs) = \begin{pmatrix} a & 0 \\
	0 & \frac{1}{a} \\
	-a & \frac{1}{a} \end{pmatrix}.
	\end{align*}
	Then, the determinants of the three Jacobian minors are
	\begin{align*}
		J[\bx^a_{1,2}](\bs) = J[\bx^a_{2,3}](\bs) = J[\bx^a_{1,3}](\bs) = 1
	\end{align*}
	for all $\bs \in \square^2$. Thus, $\bx^a \sim_\Jac \bx^b$ for any $a, b \in \R- \{0\}$.
\end{example}

\subsection{Metrics}
We use Jacobian minors to define higher dimensional notions of the $1$-variation and Lipschitz metric of a path.
\begin{definition}
\label{def:jacobian_metrics}
Let $\bx, \by \in \Jlipmap$.
We define the \emph{$1$-Jacobian variation metric} and \emph{Jacobian Lipschitz metric} to be
\begin{align}
  \met_1(\bx, \by)  \coloneqq \|\hat{d}\bx - \hat{d}\by\|_1  \text{ and }
    \met_\infty(\bx, \by)  \coloneqq \|\hat{d}\bx - \hat{d}\by \|_\infty
\end{align}
\end{definition}
Both of these are metrics since $\|\cdot\|_1$ and $\|\cdot\|_\infty$ are norms on $\Linfmap$.
When $d = 1$, the $1$-Jacobian variation and Jacobian Lipschitz metrics reduce to the classical $1$-variation and Lipschitz metrics for paths. Furthermore, they also arise as norms on $\Jlippath$ since the differential is linear. In particular, given $\bx, \by \in \Jlippath$, we have 
    \begin{align*}
        \met_1(\bx, \by) = \|d\bx - d\by\|_1 \quad \text{and}\quad \met_\infty(\bx, \by) = \|d\bx - d\by\|_{\infty}.
    \end{align*}
    However, these equalities do not hold in general for $d > 1$ because in this case the Jacobian minor operator $\hat{d}: \Jlipmap \rightarrow L^\infty(\square^\domdim, \Lambda^\domdim V)$ is not linear.

\begin{remark} The metrics of Definition \ref{def:jacobian_metrics} generalize the $1$-variation and Lipschitz metrics for paths by using \emph{volume} increments rather than \emph{length} increments. Given a map $\bx \in \Jlipmap$, its Jacobian minors $\hat{d}\bx(\bs)$ at a point $\bs \in \square^\domdim$ correspond to the components of the infinitesimal volume increment of $\bx$ at $\bs$ (where components refer to projections onto $\domdim$-planes). By the Binet-Cauchy formula \cite[Chapter 1.2.4]{mattheo}, the pointwise Euclidean norm $|\hat{d}\bx(\bs)|$ is the unsigned infinitesimal volume increment of $\bx$ at $\bs$. Thus $\|\hat{d}\bx\|_\infty$ is the maximal infinitesimal volume increment of $\bx$. Furthermore, the quantity $\|\hat{d}\bx\|_1$ equals the {\em Hausdorff area} of $\bx$ (see \cite[Theorem 3.2.3]{federer1969geometric} for details). Given another $\by \in \Jlipmap$, the metrics $\mu_\bullet$ are the $L^1$ and $L^\infty$ norms of $\hat{d}\bx - \hat{d}\by$, which in turn compute the component-wise (projected) differences in the unsigned volume increments between $\bx$ and $\by$.
\end{remark}

In the following section, we introduce the mapping space signature of a map $\bx \in \lipmap$. This signature will be defined purely in terms of the Jacobian minor $\hat{d}\bx$, and as such, the $\mu_\bullet$ metrics defined above will play an important role when establishing its continuity.

\section{The Mapping Space Signature}\label{sec: mapping space signature}
In this section, we introduce the mapping space signature
\begin{align}
  \Phi(\bx) = \left( 1, \int_{\intdom{m}{\domdim}{\bpi}} \hat d x(\bt_1), \ldots, \int_{\intdom{m}{\domdim}{\bpi}} \hat d x(\bt_1) \otimes \ldots \otimes \hat{d} x (\bt_m), \ldots\right)_{m \geq 1, \bpi \in \Sigma_m^\domdim}
\end{align}
which is derived from certain $0$-forms (\Cref{eq:chen_0_form_restricted}) in Chen's mapping space construction described in~\Cref{sec:chen_construction}.

 \subsection{Integration Domain}
 The mapping space signature is graded and each component is determined by multi-iterated integrals, where the domain of integration ${\intdom{m}{\domdim}{\bpi}}$ is a product of $m$-simplices.
 However, unlike the standard construction of the path-case, the higher dimensional setting uses permutations within each $m$-simplex which we now introduce. 

 \begin{definition}
 \label{def:permuted_simplex}
   Let $m \in \N$, $a, b \in [0,1]$ such that $a < b$, and $\pi \in \Sigma_m$.
   The \emph{permuted $m$-simplex on $(a,b)$} is
   \begin{align}
   \label{eq:permuted_m_simplex}
     \Delta^m_\pi(a,b) \coloneqq \{ a \leq t_{\pi(1)} < \ldots < t_{\pi(m)} \leq b\}.
   \end{align}
   If $(a,b) = (0,1)$, then we will simply define $\Delta^m_\pi \coloneqq \Delta^m_\pi(0,1)$, and if $\pi = \id \in \Sigma_m$ is the identity permutation, we omit the permutation and simply write $\Delta^m(a,b) \coloneqq \Delta^m_\id(a,b)$ or $\Delta^m \coloneqq \Delta^m_\id(0,1)$.
 \end{definition}
 \begin{definition}\label{def:subcube}
Suppose $\ba = (a_1, \ldots, a_{\domdim}), \bb = (b_1, \ldots, b_{\domdim}) \in \square^{\domdim}$ such that $a_j < b_j$ for all $j \in [\domdim]$. In other words, $(\ba, \bb) \in (\Delta^2)^\domdim$, which denotes a subcube of $\square^{\domdim}$, defined by
\begin{align}
    \square^{\domdim}(\ba, \bb) \coloneqq \prod_{j=1}^{\domdim} [a_j, b_j].
\end{align}
\end{definition}

\begin{definition}
\label{def:integration_domain}
Let $\bpi = (\pi_1, \ldots, \pi_{\domdim}) \in \Sigma^{\domdim}_m$ and $(\ba, \bb) \in (\Delta^2)^\domdim$.
Let
\begin{align}
\label{eq:Dmd_domain}
    \intdom{m}{\domdim}{\bpi}(\ba, \bb) &= \Delta^m_{\pi_1}(a_1, b_1) \times \ldots \times \Delta^m_{\pi_{\domdim}}(a_{\domdim}, b_{\domdim})\\
                                 &=\{\bt = (t_{i,j})_{i \in [m], j \in [\domdim]}:  a_j \leq t_{\pi_j(1), j} < t_{\pi_j(2), j} < \ldots < t_{\pi_j(m), j} \leq b_j,  \,\forall j \in [\domdim] \}. \nonumber
\end{align}
Following the convention from the permuted $m$-simplex, we let \[\intdom{m}{\domdim}{\bpi} \coloneqq \intdom{m}{\domdim}{\bpi}(\bzero, \bone)\] and $\intdom{m}{\domdim}{} \coloneqq \intdom{m}{\domdim}{\bid}(\bzero, \bone)$, where $\bid \in \Sigma_m^{\domdim}$ is the identity. We denote 
\begin{align}
\bt_i \coloneqq (t_{i,1}, t_{i,2}, \ldots, t_{i,\domdim}),
\end{align}
which has one element from each of the simplices in $\intdom{m}{\domdim}{\bpi}(\ba, \bb)$. Note that we reserve the indices $i$ and $j$ to index elements in $[m]$ and $[\domdim]$ respectively.
\end{definition}

\subsection{The Mapping Space Signature}
The mapping space signature consists of all of the $0$-forms in Chen's mapping space construction in the form given in~\Cref{eq:chen_0_form_restricted}. The codomain of this signature is a vector space which encapsulates all of this information, extending the classical power series tensor algebra for the path signature.
\begin{definition}
    Let $V$ be a vector space and $d \in \N$. For $m \in \N$, define the \emph{level $m$ permutation tensor space} to be the vector spaces
    \begin{align}
    \label{eq:degm_perm_tens_space_def}
        \pvspace{d}{(0)}{V} \coloneqq \R, \quad \text{and} \quad \pvspace{d}{(m)}{V} \coloneqq \prod_{\bpi \in \Sigma_m^{\domdim}} \left( \Lambda^\domdim V \right)^{\otimes m}
    \end{align}
    for $m > 0$. Given a basis $(e_i)_{i=1}^\coddim$ of $V$, the set $(e^P)_{P \in \cO_{\domdim,\coddim}}$ is a basis for $\Lambda^\domdim V$. Thus, the set
    \begin{align}
        \left(e^{\cP, \bpi}\right)_{(\cP, \bpi) \in \cO_{\domdim,\coddim}^m \times \Sigma_m^\domdim}
    \end{align}
    is a basis for $\pvspace{d}{(m)}{V}$. If $V$ is a Hilbert space with orthonormal basis $(e_i)_{i=1}^\coddim$, this provides an orthonormal basis of $\pvspace{d}{(m)}{V}$. Thus, we say that $(\cP, \bpi) \in \cO_{\domdim,\coddim}^m \times \Sigma_m^\domdim$ is a \emph{level $m$ index}. In particular, we call $\cP\label{eq:forms_index_def}$ the \emph{forms index} and $\bpi\label{eq:permutation_index_def}$ the \emph{permutation index}.
    Finally, define the \emph{power series permutation tensor space} to be the graded vector space
    \begin{align}
    \label{eq:ps_perm_tens_space_def}
        \pspvspace{d}{}{V} \coloneqq \prod_{m \geq 0} \pvspace{d}{(m)}{V}.
    \end{align}
\end{definition}

If $\br_m \in \pvspace{\domdim}{(m)}{V}$, then we will use $r^{\cP, \bpi}_m$ to denote the $e^{\cP, \bpi}$ component and emphasize that the component is degree $m$. Furthermore, we use $\br^\bpi_m \in \Lambda^\domdim V $ to denote the element in the $\bpi$ coordinate of $\pvspace{\domdim}{(m)}{V}$, and use $\br_m \in \pvspace{\domdim}{(m)}{V}$ to denote the full degree $m$ component. Furthermore, we note that the dimension of $\pvspace{\domdim}{(m)}{V}$ is
\begin{align}
    \dim(\pvspace{\domdim}{(m)}{V}) =  \binom{\coddim}{\domdim}^m (m!)^{\domdim}.
\end{align}

\begin{definition}
  The \emph{mapping space signature of} $\bx \in \lipmap$ is defined as
  \begin{align}
  \label{eq:mapping_signature_def}
    \Phi(\bx) \coloneqq \left(\Phi_m(\bx)\right)_{m=0}^\infty  \in \pspvspace{d}{}{V},
  \end{align}
  where by convention $\Phi_0(\bx) \coloneqq 1 \in \R$ and 
  \begin{align}
    \Phi_m(\bx) \coloneqq \left(\Phi^\bpi_m(\bx)\right)_{\bpi \in \Sigma_m^{\domdim}} \in \pvspace{d}{(m)}{V},
  \end{align}
with 
\begin{align}
  \Phi^\bpi_m(\bx) \coloneqq \int_{\intdom{m}{\domdim}{\bpi}} \hat{d}\bx(\bt_1) \otimes \cdots \otimes \hat{d}\bx(\bt_m) \in \left(\Lambda^\domdim V\right)^{\otimes m}
\end{align}
We refer to $\Phi_m(\bx)$ as \emph{level $m$ mapping space signature of $x$} and to $\Phi_m^\bpi(\bx)$ as the \emph{level $m$ mapping space signature of $x$ with respect to $\bpi$}. 

   \end{definition}
   
Because the mapping space signature is defined strictly in terms of the Jacobian minor operator $\hat{d}$, it is invariant under Jacobian equivalence (\Cref{def:jacobian_equivalence}), and is well defined as a map
\begin{align}
    \Phi: \Jlipmap \rightarrow \pspvspace{d}{}{V}.
\end{align}

\subsection{Mapping Space Signature Coordinates}
   As mentioned before, it can be beneficial to work in coordinates.
   Therefore we fix an orthonormal basis of $V$ which allows us to use a collection $\cP = (P_1, \ldots, P_m) \in \cO_{\domdim,\coddim}^m$ of order preserving injections to index an orthonormal basis of $\left(\Lambda^\domdim  V\right)^{\otimes m}$.
\begin{definition}
\label{def:mapping_space_monomial}
Let $m \in \N$. The \emph{level $m$ mapping space monomial of $\bx\in\lipmap$ with respect to $(\cP, \bpi) \in \cO_{\domdim,\coddim}^m \times \Sigma_m^{\domdim}$} is defined to be
    \begin{align}
      \label{eq:original_mapping_space_monomial}
      \Phi^{\cP, \bpi}_m(\bx) &\coloneqq \int_{\intdom{m}{\domdim}{\bpi}} \hat{d}^{P_1}\bx(\bt_1) \swedge \cdots \swedge \hat{d}^{P_m}\bx(\bt_m).
    \end{align}
\end{definition}

Note that in the setting of a Lipschitz map $\bx \in \lipmap$, we can use~\Cref{eq:lipschitz_jacobian_minor} to rewrite the mapping space monomial with respect to $(\cP, \bpi) \in \cO_{\domdim,\coddim}^m \times \Sigma_m^{\domdim}$ as
\begin{align}
    \Phi^{\cP, \bpi}_m(\bx) = \int_{\intdom{m}{\domdim}{\bpi}} \prod_{i=1}^m J[\bx_{P_i}](\bt_i) d\bt,
\end{align}
recovering the $0$-form from Chen's construction in~\Cref{eq:chen_0_form_restricted}. Following the arguments of the full mapping space signature above, the mapping space monomials is well defined on Jacobian equivalence classes of Lipschitz maps,
\begin{align}
    \Phi_m^{\cP, \bpi}: \Jlipmap \rightarrow \R.
\end{align}

\begin{example}
Let $\bx:\square^2 \to \R^4$ be the algebraic map
\begin{align*}
\bx(s_1,s_2) := \left(s_1, ~ s_1s_2^2,~ s_1^2s_2^3,~s_2^4\right),
\end{align*}
whose Jacobian with respect to the standard basis of $\R^4$ is
\[
\renewcommand*{\arraystretch}{1.2}
d\bx(\bs) = 	\begin{pmatrix} 
				1 & 0 \\
				s_2^2 & 2s_1s_2 \\
				2s_1s_2^3 & 3s_1^2s_2^2 \\
				0 & 4s_2^3 
		        \end{pmatrix}.
\]
At level $m=3$, consider $\cP \in \cO_{2,4}^3$ given by the strictly increasing sequences $P_i:[2] \to [4]$ with images
\[
P_1\set{1,2} = \set{1,3}, \quad P_2\set{1,2} = \set{1,4}, \quad P_3\set{1,2} = \set{2,3};
\]
the three Jacobian minors implicated by the $P_i$ are 
\[
J[\bx_{P_1}](s_1, s_2) = 3s_1^2s_2^2, \quad J[\bx_{P_2}](s_1, s_2) = 4s_2^3, \quad J[\bx_{P_3}](s_1, s_2) = -s_1^2s_2^4.
\]
Let $\bpi = (\pi_1,\pi_2) \in \Sigma_3^2$ be the permutations:
\[
\pi_1 = (1~2~3), \quad \pi_2 = (3~1~2).
\]
We integrate over $\Delta^3_{\pi_1} \times \Delta^3_{\pi_2}$ consisting of points $(t_{1,1},t_{1,2},t_{1,3}) \times (t_{2,1},t_{2,2},t_{2,3})$ whose components satisfy the inequalities
\[
0 \leq t_{1,1} \leq t_{1,2} \leq t_{1,3} \leq 1 \quad \text{and} \quad 0 \leq t_{2,3} \leq t_{2,1} \leq t_{2,2} \leq 1.
\]
Thus, the monomial $\Phi_m^{\cP,\bpi}(\bx)$ evaluates to
\[
\underbrace{\int_0^1\tight\int_0^{t_{3,1}}\tight\int_0^{t_{2,1}}}_{\Delta^3_{\pi_1}}\underbrace{\int_0^1\tight\int_0^{t_{2,2}}\tight\int_0^{t_{1,2}}}_{\Delta^3_{\pi_2}}  3t_{1,1}^2t_{1,2}^2 \cdot 4t_{2,2}^3 \cdot (-t_{3,1}^2t_{3,2}^4)\, dt_{3,2}\, dt_{1,2}\, dt_{2,2}\, dt_{1,1}\, dt_{2,1}\, dt_{3,1}. 
\]
Monomials corresponding to other choices of $\cP$ and $\bpi$ may be computed analogously.
\end{example}

\subsection{A Hilbert Space Codomain}    
As in the $d=1$ case, the codomain of the mapping space signature has more structure when the underlying vector space $V$ is a Hilbert space.
In particular, we now show that $\Phi$ takes values in a graded Hilbert space.
\begin{definition}
    Suppose $V$ is a Hilbert space. The \emph{permutation tensor space} is the graded Hilbert space
    \begin{align}
    \label{eq:perm_hilb_space_def}
        \phspace{d}{}{V} \coloneqq \bigoplus_{m=0}^\infty \pvspace{d}{(m)}{V}.
    \end{align}
    Furthermore, we define the \emph{power series permutation tensor Hilbert space} to be the graded Hilbert space
    \begin{align}
    \label{eq:ps_perm_hilb_space_def}
        \psphspace{d}{}{V} \coloneqq \left\{ (\br_m)_{m=0}^\infty \in \pspvspace{d}{}{V} \, : \, \sum_{m=0}^\infty \|\br_m\|^2 < \infty \right\},
    \end{align}
    where $\br_m \in \pvspace{d}{(m)}{V}$ and the grading of the Hilbert space is given by $m$. 
\end{definition}

Now, we show that the mapping space signature takes values in $\psphspace{d}{}{V}$.

\begin{proposition}
\label{prop:factorial_decay}
    For any $\bx \in \Lip(\square^{\domdim}, V)$, we have $\|\Phi(\bx)\| < \infty$ and thus $\Phi(\bx) \in \psphspace{\domdim}{}{V}$. 
\end{proposition}
\begin{proof}
    Consider the level $m>0$ mapping space signature $\Phi_m^{\cP, \bpi}$, where $(\cP, \bpi) \in \cO^m_{\domdim,\coddim} \times \Sigma_m^\domdim$. Suppose $L = \|\bx\|_{\Lip}$ is the Lipschitz constant of $\bx$. Then, we have
    \begin{align*}
        \left|\Phi_m^{\cP, \bpi}(\bx)\right| &\leq \int_{\intdom{m}{\domdim}{\bpi}}\prod_{i=1}^m \big|J[\bx_{P_i}](\bt_i) \big| \leq \frac{L^{\domdim m}}{(m!)^\domdim},
    \end{align*}
    since the volume of $\intdom{m}{\domdim}{\bpi}$ is $\frac{1}{(m!)^{\domdim}}$. Now, considering the full signature, we have
    \begin{align*}
        \|\Phi(\bx)\|^2 &= 1 + \sum_{m=1}^\infty \sum_{(\cP,\bpi) \in \cO^m_{\domdim,\coddim}\times\Sigma_m^\domdim} |\Phi_m^{\cP,\bpi}(\bx)|^2 \\
        & \leq 1 + \sum_{m=1}^\infty \frac{L^{2 \domdim m}}{(m!)^{2\domdim}} \binom{\coddim}{\domdim}^m (m!)^{\domdim} \\
        & \leq \sum_{m=0}^\infty \frac{L^{2\domdim m}}{(m!)^{\domdim}}\binom{\coddim}{\domdim}^m < \infty.
    \end{align*}
\end{proof}

Thus, we view the mapping space signature as a map
\begin{align}
    \Phi: \Jlipmap \rightarrow \psphspace{d}{}{V}.
\end{align}

\subsection{Continuity}
The mapping space signature is continuous when $\Jlipmap$ is equipped with either the $1$-Jacobian variation metric or the Jacobian Lipschitz metric from~\Cref{def:jacobian_metrics}.
We provide the proofs for the Jacobian Lipschitz metric, but essentially the same arguments can be used for the $1$-Jacobian variation metric.

\begin{proposition}
\label{prop:monomial_continuity}
\label{cor:signature_continuty}
The mapping space signature
\begin{align}
  \Phi: (\Jlipmap, \met_\infty) \rightarrow \psphspace{\domdim}{}{V}
\end{align}
is continuous.
Moreover, if $\bx, \by \in \Jlipmap$ and
\begin{align*}
        L &> \max\{ \|\hat{d}\bx\|_\infty, \|\hat{d}\by\|_\infty\}\\
        \epsilon &> \met_\infty(\bx, \by),
    \end{align*}
     then for every $m\ge 0$
    \begin{align*}
        |\Phi^{\cP, \bpi}_m(\bx) - \Phi^{\cP, \bpi}_m(\by)| < \frac{mL^{m-1}}{(m!)^{\domdim}} \epsilon.
    \end{align*}
\end{proposition}

\begin{proof}
We first show the second claim and use the fact that for $a_1, \ldots, a_m, b_1, \ldots, b_m \in \R$ such that $|a_i|, |b_i| < L$ and $|a_i - b_i| < \epsilon$ for all $i \in [m]$, then
    \begin{align}
    \label{eq:prod_inequality}
        \left|\prod_{i=1}^m a_i - \prod_{i=1}^m b_i \right| < m L^{m-1} \epsilon.
    \end{align}
    Then, we have
    \begin{align*}
        |\Phi^{\cP, \bpi}_m (\bx) - \Phi^{\cP, \bpi}_m(\by)| &\leq  \int_{\intdom{m}{\domdim}{\bpi}} \left|\prod_{i=1}^m J[\bx_{P_i}](\bt_i) - \prod_{i=1}^m J[\by_{P_i}](\bt_i) \right| d\bt \\
        & < \int_{\intdom{m}{\domdim}{\bpi}} mL^{m-1} \epsilon d\bt \\
        & = \frac{mL^{m-1}}{(m!)^{\domdim}} \epsilon
    \end{align*}
   
    We now use this estimate to show the continuity of $\Phi$.
    Let $\bx \in \Jlipmap$ and $\epsilon > 0$. Suppose $L > \|\hat{d}\bx\|_\infty$, then 
\begin{align*}
  L + \epsilon > \sup_{\by \in \Jlipmap} \{ |\by|_{J, Lip} \, : \, \met_\infty(\bx, \by) < \epsilon\}.
\end{align*}
Then, for every $\by \in \Jlipmap$ such that $\met_\infty(\bx, \by) < \epsilon$, we have
\begin{align*}
  \left\| \Phi(\bx) - \Phi(\by) \right\|^2 & = \sum_{m=1}^\infty \sum_{(\cP, \bpi)} |\Phi_m^{\cP,\bpi}(\bx) - \Phi_m^{\cP, \bpi}(\by)|^2 \\
                                          & < \sum_{m=1}^\infty \sum_{(\cP, \bpi)} \frac{m^2 (L+\epsilon)^{2(m-1)}}{(m!)^{2\domdim}} \epsilon^2 \\
                                          & = \left(\sum_{m=1}^\infty \binom{\coddim}{\domdim}^m  \frac{m^2 (L+\epsilon)^{2(m-1)}}{(m!)^{\domdim}}\right) \epsilon^2,
\end{align*}
where the sum converges for any fixed $L, \epsilon>0$. 
\end{proof}

\begin{remark}
    The metrics in~\Cref{def:jacobian_metrics} form pseudometrics on the space of Lipschitz maps $\lipmap$ and induce topologies which are finer than the topology induced by the naive generalizations of the 1-variation and Lipschitz metrics. Consider the Jacobian Lipschitz metric $\mu_\infty$. Suppose $\bx \in \lipmap$, $\epsilon > 0$, and consider the ball $B_\Lip(\bx, \epsilon)$ with respect to the ordinary Lipschitz metric. Given any $\by \in B_\Lip(\bx,\epsilon)$, we have $\|\by\|_\Lip < \|\bx\|_\Lip + \epsilon \eqqcolon L$. Then given $P \in \cO_{\domdim, \coddim}$, and using~\Cref{eq:prod_inequality}, we have
    \begin{align*}
        \left|J[\bx_P](\bs) - J[\by_P](\bs) \right|  \leq \sum_{\sigma \in \Sigma_\domdim}  \left|\prod_{j = 1}^\domdim \frac{\partial x_{P(j)}}{\partial s_{\sigma(j)}} - \frac{\partial y_{P(j)}}{\partial s_{\sigma(j)}} \right|  < (d!) \domdim L^{\domdim-1}\epsilon,
    \end{align*}
    which implies that
    \begin{align*}
        \mu_\infty(\bx, \by) < \binom{\coddim}{\domdim}^{1/2} (d!) \domdim L^{\domdim-1}\epsilon.
    \end{align*}
    This shows that the ordinary Lipschitz topology is included in the topology induced by $\mu_\infty$. However, consider the family of maps $\bx^a$ defined in~\Cref{ex:jacobian_equivalence}. In this case, $\mu_\infty(\bx^a, 0) = 1$, but $\|\bx^a\|_{\Lip} > |a|$ for any $a \in \R - \{0\}$. Thus, the topology induced by $\mu_\infty$ is strictly finer than the topology induced by $\|\cdot\|_\Lip$. 
\end{remark}

\subsection{Permutation Invariance}
\label{ssec:permutation_invariance}

One of the differences between the mapping space signature and the classical path signature is the additional terms involving integration along permuted simplices. However, we show that the level $m$ mapping space monomials are invariant under an action of the symmetric group $\Sigma_m$. By quotienting out this symmetric group action in the case of $d=1$, we recover the usual construction of the path signature, which is discussed in the following subsection.

\begin{proposition}
\label{prop:permutation_invariance}
	Let $(\cP, \bpi) \in \cO_{\domdim,\coddim}^m \times \Sigma_m^\domdim$ be a level $m$ index. Let $\sigma \in \Sigma_m$ and define the permutation actions
	\begin{align*}
	    \sigma\cP & = (P_{\sigma(1)}, \ldots, P_{\sigma(m)}) \\
		\sigma\bpi & = (\sigma \pi_1, \ldots, \sigma \pi_{\domdim}).
	\end{align*}
	Then for any $\bx \in \Jlipmap$, 
	\begin{align}
	\label{eq:permutation_invariance}
		\Phi_m^{ \sigma \cP, \bpi}(\bx) = \Phi_m^{\cP, \sigma \bpi}(\bx).
	\end{align}
\end{proposition}
\begin{proof}
	By definition,
	\begin{align*}
		\Phi_{m}^{\sigma \cP, \bpi}(\bx) = \int_{\intdom{m}{\domdim}{\bpi}} \prod_{i=1}^m J[\bx_{P_{\sigma(i)}}](\bt_i) d\bt.
	\end{align*}
	We perform the change of variables
	\begin{align*}
		t_{i, j} \mapsto t_{\sigma(i),j}.
	\end{align*}
	Under this transformation, the domain of integration becomes $\intdom{m}{\domdim}{\sigma \bpi}$, and we obtain the desired result.
\end{proof}

\begin{remark}
\label{rmk:permutation_action}
The expression in~\Cref{eq:permutation_invariance} yields a general action of $\Sigma_m$ on the level $m$ permutation tensor space, $\pvspace{\domdim}{(m)}{V}$. Indeed, given $\sigma \in \Sigma_m$ and $\br_m \in \pvspace{\domdim}{(m)}{V}$, we define the action to be
\begin{align}
    \sigma \cdot \br_m^{\cP, \bpi} \coloneqq \br_m^{\sigma^{-1} \cP, \sigma\bpi}.
\end{align}
Although we can quotient out this permutation action in the codomain while retaining the same information as the full signature, we primarily work with the full signature in this paper to avoid computing explicit representatives and to simplify notation in the algebraic formulas.
\end{remark}

This result implies that, in practice, we do not need to compute all of the monomials in order to obtain the mapping space signature. We can restrict ourselves to the monomials where the permutation index $\bpi = (\pi_1, \ldots, \pi_{\domdim}) \in \Sigma_m^{\domdim}$ has one component fixed to the identity; for example $\pi_1 = \id$. Indeed, we can compute the monomial for an arbitrary $\bpi = (\pi_1, \ldots, \pi_\domdim) \in \Sigma_m^\domdim$ as
\begin{align*}
	\Phi_{m}^{\cP,\bpi}(\bx) = \Phi_{m}^{\pi_1\cP, \pi_1^{-1}\bpi}(\bx).
\end{align*}

In the case of $\domdim=1$, this implies that the permutation index can always be set to the identity, and thus all of the path space monomials can be written as integrals over the standard $n$ simplex $\Delta^n$. This explains the absence of a permutation index for the standard definition of the path signature.

\subsection{The Identity Signature and Recovering the Path Signature} \label{ssec:identity_signature}
In the case of $\domdim =1$, the integration domain only contains a single permuted simplex, and thus the permutation index of the corresponding quotiented mapping space signature can always be fixed to be the identity permutation. In fact, we can define higher dimensional analogues of the \emph{identity mapping space signature}.

\begin{definition}
    The \emph{identity mapping space signature of $\bx \in \Jlipmap$} is defined as
    \begin{align}
    \label{eq:identity_ms_sig}
        \Phi^\bid(\bx)\coloneqq \left(\Phi_m^\bid(\bx)\right) \in \prod_{m=0}^\infty \left(\Lambda^\domdim V\right)^{\otimes m} \eqqcolon \idpspvspace{d}{}{V},
    \end{align}
    where $\bid =(\id, \ldots, \id) \in \Sigma_m^\domdim$ is the vector of identity permutations.
\end{definition}

\begin{remark}
    The identity mapping space signature can be defined recursively, analogous to the case of the path signature. This is discussed in~\Cref{apx:recursive_definition}.
\end{remark}

Returning to the case of $\domdim=1$, the identity signature is equivalent to the full signature $\Phi$ quotiented out by the permutation action from~\Cref{rmk:permutation_action}. Furthermore, the $1$-dimensional Jacobian minor operator is simply the differential, and therefore
\begin{align}
    \Phi_m^{\id}(\bx) = \int_{\Delta^m}d\bx(t_1) \otimes \ldots \otimes d\bx(t_m),
\end{align}
recovering the classical path signature. In order to distinguish between the path signature in the case of $\domdim=1$ and the mapping space signature for arbitrary $\domdim \geq 1$, we set
\begin{align}
\label{eq:path_signature_def}
    S_m^P \coloneqq \Phi^{P, \id}_m: \lippath \rightarrow \R, \quad \quad S \coloneqq \Phi^{\id}: \lippath \rightarrow \idpspvspace{1}{}{V},
\end{align}
where $\Phi$ is understood to be the $d=1$ mapping space signature in these expressions and $P \in \cO^m_{1,\coddim}$ can be viewed as a multi-index of length $m$ valued in $[\coddim]$. 

The permutatations in the integration domain of the mapping space signature are necessary for the generalized shuffle product structure (\Cref{sec:ms_shuffle}) as well as the universal and characteristic properties (\Cref{sec:univ_char}). However, the composition structure (\Cref{sec:ms_composition}) and the recursive definition of the signature (\Cref{apx:recursive_definition}) only hold for the identity signature.

\section{Invariance Properties}
\label{sec:ms_invariance}

In this section, we discuss invariance properties of the mapping space signature. We begin by establishing reparametrization invariance in the case where each coordinate of $\square^{\domdim}$ is reparametrized independently. Next, we discuss a special class of maps, constructed using tree-like paths, which have trivial mapping space signature.

\subsection{Reparametrization Invariance}

In this section, we will consider reparametrization invariance, one of the fundamental properties of the path signature. We find that the mapping space signature is invariant under reparametrizations in which each coordinate of $\square^{\domdim}$ is reparametrized independently.

\begin{proposition}
\label{prop:ms_reparametrization_invariance}
	Let $\phi_j: \square^1 \rightarrow \square^1$ be Lipschitz, monotone increasing bijections for $j \in [\domdim]$ and let $\bphi(\bs) = \left(\phi_1(s_1), \ldots, \phi_\domdim(s_{\domdim})\right) : \square^{\domdim} \rightarrow \square^{\domdim}$. For any level $m$ index $(\cP, \bpi) \in \cO_{\domdim,\coddim}^m \times \Sigma_m^\domdim$, and $\bx \in \lipmap$,
	\begin{align}
		\Phi_{m}^{\cP, \bpi}(\bx \circ \bphi) = \Phi_{m}^{\cP, \bpi}(\bx).
	\end{align}
\end{proposition}
\begin{proof}
	First, we note that for any order preserving bijection $P \in \cO_{\domdim, \coddim}$, the chain rule gives us
	\begin{align*}
		J[(\bx \circ \bphi)_P] (s_1, \ldots, s_{\domdim}) &= J[\bx_P]\big(\phi_1(s_1), \ldots, \phi_{\domdim}(s_{\domdim})\big) \cdot J[\bphi]\big(s_1, \ldots, s_{\domdim}\big) \\
		& = J[\bx_P]\big(\phi_1(s_1), \ldots, \phi_{\domdim}(s_{\domdim})\big)\cdot \phi'_1(s_1) \cdot \ldots \cdot\phi'_{\domdim}(s_{\domdim}).
	\end{align*}
	Then, applying this to $\Phi_{m}^{\cP, \bpi}(\bx \circ \phi)$, we get
	\begin{align*}
		\Phi_{m}^{\cP, \bpi}(\bx \circ \phi) & = \int_{\intdom{m}{\domdim}{\bpi}} \prod_{i=1}^m J[(\bx \circ \bphi)_{P_i}](t_{i,1}, \ldots, t_{i,\domdim}) d\bt \\
		& = \int_{\intdom{m}{\domdim}{\bpi}} \prod_{i=1}^m J[\bx_{P_i}](\phi_1(t_{i,1}), \ldots, \phi_{\domdim}(t_{i,\domdim})) \phi'_1(t_{i,1}) \ldots \phi'_{\domdim}(t_{i,\domdim}) d\bt.
	\end{align*}
	Next, we make the change of variables
	\begin{align*}
		\phi_j(t_{i,j}) &\mapsto t_{i,j}.
	\end{align*}
	Because each $\phi_j$ is a monotone increasing bijection, each simplex $\Delta^m_{\pi_j}$ is preserved under this transformation:
	\begin{align*}
		t_{\pi_j(1), j} < \ldots < t_{\pi_j(m), j} \quad \implies \quad \phi_j^{-1}(t_{\pi_j(1), j}) < \ldots < \phi_j^{-1}(t_{\pi_j(m), j}).
	\end{align*}
	Then, under this transformation, the above integral becomes
	\begin{align*}
		\Phi_{m}^{\cP,\bpi}(\bx \circ \bphi)  = \int_{\intdom{m}{\domdim}{\bpi}} \prod_{i=1}^m J[\bx_{P_i}](\bt_i) d\bt = \Phi_{m}^{\cP, \bpi}(\bx).
	\end{align*}
\end{proof}

\subsection{Maps with Trivial Signature} \label{ssec:trivial_signature}

Here, we consider maps $\bx \in \lipmap$ defined as a sum of $d$ paths. For such maps, we show that the mapping space monomials can be expressed as a linear combination of path signature monomials, which then leads to a class of maps which have trivial mapping space signature.\medskip

Let $\bg_1, \ldots, \bg_{\domdim} \in \lippath$ be a collection of paths on $V$, where each path is written component-wise as $\bg_j = (g_{1,j}, \ldots, g_{\coddim,j})$. Note that $j$ is fixed in the second index. Define the map $\bx = (x_1, \ldots, x_\coddim): \square^{\domdim} \rightarrow V$ componentwise by
\begin{align*}
    x_k(s_1, \ldots, s_{\domdim}) = \sum_{j=1}^d g_{k,j}(s_j).
\end{align*}
Because this is a sum of Lipschitz maps, $\bx \in \lipmap$. Note that as a matrix, the Jacobian of $\bx$ is 
\begin{align*}
    \big(d\bx(\bs)\big)_{k,j} = g'_{k,j}(s_j).
\end{align*}
Thus, given an order-preserving injection $P: [\domdim] \rightarrow [\coddim]$, the Jacobian minor is
\begin{align}
    J[\bx_P](s_1, \ldots, s_{\domdim}) = \sum_{\sigma \in \Sigma_d} (-1)^{\sgn(\sigma)} \prod_{q=1}^{\domdim} g'_{P(q),\sigma(q)}(s_{\sigma(q)}).
\end{align}
Now, we can compute the mapping space monomials for such maps as follows. Suppose $(\cP, \bpi) \in \cO_{\domdim,\coddim}^m \times \Sigma_m^\domdim$. Then,
\begin{align*}
    \Phi_{m}^{\cP, \bpi}(\bx) & = \int_{\intdom{m}{\domdim}{\bpi}} \prod_{i=1}^m J[\bx_{P_i}](t_{i,1}, \ldots, t_{i,d}) d\bt\\
    & = \int_{\intdom{m}{\domdim}{\bpi}} \prod_{i=1}^m \left(\sum_{\sigma_i \in \Sigma_{\domdim}} (-1)^{\sgn(\sigma_i)} \prod_{q=1}^{\domdim} g'_{P_i(q), \sigma_i(q)}(t_{i,\sigma_i(q)})\right) d\bt \\
    & = \int_{\intdom{m}{\domdim}{\bpi}} \left(\sum_{\bsigma  \in \Sigma_{\domdim}^m} (-1)^{\sgn(\bsigma)} \prod_{i=1}^m \prod_{q=1}^{\domdim} g'_{P_i(q), \sigma_i(q)}(t_{i,\sigma_i(q)})\right) d\bt \\
    & = \sum_{\bsigma \in \Sigma_{\domdim}^m} (-1)^{\sgn(\bsigma)} \int_{\intdom{m}{\domdim}{\bpi}}\prod_{q=1}^{\domdim} \prod_{i=1}^m g'_{P_i(q), \sigma_i(q)}(t_{i,\sigma_i(q)}) d\bt,
\end{align*}
where we use $\bsigma = (\sigma_1, \ldots, \sigma_m) \in \Sigma_\domdim^m$. At this point, we break the inner integral up into a product of $d$ integrals. This is done by using the structure of the domain of integration
\begin{align*}
    \intdom{m}{\domdim}{\bpi} = \Delta^m_{\pi_1} \times \ldots \times \Delta^m_{\pi_{\domdim}}
\end{align*}
where $\Delta^m_{\pi_j}$ is parametrized by $\bt^j = (t_{1,j}, \ldots, t_{m,j})$ (note that $j$ is the second index). Therefore, we write the double product in the integrand as a product of functions of $t_{i,j}$ for fixed $j$. Given $\bsigma \in \Sigma_{\domdim}^m$ and some $j \in [\domdim]$, for every $i$, there exists a unique $q \in [\domdim]$ such that $\sigma_i(q) = j$. Thus, we re-index the double product in terms of $j$ rather than $q$:
\begin{align*}
    \prod_{q=1}^{\domdim} \prod_{i=1}^m g'_{P_i(q), \sigma_i(q)}(t_{i,\sigma_i(q)}) = \prod_{j=1}^{\domdim} \prod_{i=1}^m g'_{P_i(\sigma_i^{-1}(j)),j}(t_{i,j}).
\end{align*}
Continuing along with the above manipulations, we get
\begin{align*}
    \Phi_{m}^{\cP,\bpi}(\bx) & = \sum_{\bsigma \in \Sigma_{\domdim}^m} (-1)^{\sgn(\bsigma)} \int_{\bt \in \intdom{m}{\domdim}{\bpi}}\prod_{j=1}^{\domdim} \prod_{i=1}^m g'_{P_i(\sigma_i^{-1}(j)), j}(t_{i,j}) d\bt \\
    & = \sum_{\bsigma \in \Sigma_{\domdim}^m} (-1)^{\sgn(\bsigma)}  \prod_{j=1}^{\domdim} \int_{\bt^j \in \Delta^m_{\pi_j}} \prod_{i=1}^m g'_{P_i(\sigma_i^{-1}(j)),j}(t_{i,j}) d\bt^j \\
    & = \sum_{\bsigma \in \Sigma_{\domdim}^m} (-1)^{\sgn(\bsigma)}  \prod_{j=1}^{\domdim} S_m^{Q(\cP, \bpi, \bsigma, j)}(\bg_j),
\end{align*}
where $S^P_m: \lippath \rightarrow \R$ is the classical path signature (\Cref{eq:path_signature_def}). Here, $Q(\cP, \bpi, \bsigma, j)$ is a multi-index of length $m$ valued in $[\coddim]$, defined as follows. Suppose $(\cP, \bpi) \in \cO_{\domdim,\coddim}^m \times \Sigma_m^\domdim$, $\bsigma \in \Sigma_{\domdim}^m$, and $j \in [\domdim]$. Then, 
\begin{align}
\label{eq:Qdef}
    Q(\cP, \bpi, \bsigma, j) = \pi_j^{-1} \cdot \left(P_1(\sigma_1^{-1}(j)), \ldots, P_m(\sigma_m^{-1}(j))\right),
\end{align}
where $\left(P_1(\sigma_1^{-1}(j)), \ldots, P_m(\sigma_m^{-1}(j))\right)$ is a multi-index of length $m$ valued in $[\coddim]$, and $\pi_{j}^{-1} \in \Sigma_m$ acts by permutation of elements. Thus, we have proved the following.

\begin{proposition}
    Let $\bg_1, \ldots, \bg_{\domdim}\in \lippath$. Let $\bx: \square^{\domdim} \rightarrow V$ be defined by
    \begin{align}
    \label{eq:gamma_path}
        \bx(s_1, \ldots, s_{\domdim}) \coloneqq \sum_{j=1}^\domdim \bg_j(s_j). 
    \end{align}
    Then,
    \begin{align}
    \label{eq:gamma_path_expression}
        \Phi_{m}^{\cP,\bpi}(\bx) = \sum_{\bsigma \in \Sigma_{\domdim}^m} (-1)^{\sgn(\bsigma)}  \prod_{j=1}^{\domdim} S_m^{Q(\cP, \bpi, \bsigma, j)}(\bg_j),
    \end{align}
    where $Q(\cP, \bpi, \bsigma, j)$ is given in~\Cref{eq:Qdef}.
\end{proposition}

An immediate corollary is the following.
\begin{corollary}
\label{cor:trivial_ms_signature}
    Let $\bg_1, \ldots, \bg_{\domdim} \in \lippath$. Let $\bx: \square^{\domdim} \rightarrow V$ be defined by~\Cref{eq:gamma_path}. Furthermore, suppose at least one of the paths $\bg_j$ is tree-like\footnote{A path $\bx \in \lippath$ is \emph{tree-like}~\cite{hambly_uniqueness_2010, hambly_notes_2008} if and only if there exists some $\R$-tree (i.e., a metric space in which any two points are connected by a unique arc which is isometric to a real interval) $\fT$ such that $\bx$ decomposes as
    \[
    \bx: \square^1 \xrightarrow{\phi} \fT \xrightarrow{\psi} V
    \]
    where $\phi$ and $\psi$ are continuous maps with $\phi(0) = \phi(1)$.}. Then, for any $m \in \N$ and $(\cP, \bpi) \in \cO_{\domdim,\coddim}^m \times \Sigma_m^\domdim$, we have
    \begin{align*}
        \Phi_{m}^{\cP,\bpi}(\bx) = 0.
    \end{align*}
\end{corollary}
\begin{proof}
	Without loss of generality, suppose $\bg_1$ is tree-like. Because $\bg_1$ is tree-like, its path signature is trivial~\cite{hambly_uniqueness_2010}, so for any multi-index $P$ valued in $[\coddim]$, $S^P(\bg_1) = 0$. Then, by~\Cref{eq:gamma_path_expression}, $\Phi_{m}^{\cP, \bpi}(\bx) = 0$ for any $(\cP, \bpi) \in \cO_{\domdim,\coddim}^m \times \Sigma_m^\domdim$. 
\end{proof}

\section{Equivariance}\label{sec:equivariance}

In this section, we consider equivariance properties of the mapping space signature. There are two natural group actions on the space of Jacobian equivalence classes of Lipschitz maps $\Jlipmap$:
\begin{enumerate}
    \item \textbf{pre-composition:} a right action of the $d$-dimensional hyperoctadral group $B_{\domdim}$ on the domain $\square^{\domdim}$;
    \item \textbf{post-composition:} a left action of the group of linear operators $\GLV$ on the codomain $V$.
\end{enumerate}
In fact, given another finite dimensional vector space $W$, a linear transformation $A \in L(V,W)$ induces a map $A : \Jlipmap \rightarrow \JlipmapW$. These group actions and linear transformations also induce corresponding actions and transformations on the permutation tensor spaces $\psphspace{\domdim}{}{V}$, and we will show that the mapping space signature is compatible with these induced maps.

In the case of paths, the $1$-dimensional hyperoctahedral group is simply $B_1 \cong \Z_2$, where the nontrivial action of $\tau \in B_1$ on $\bx \in \lippath$ corresponds to time reversal
\begin{align}
\label{eq:path_reversal}
    (\mapact{\bx}{}{\tau})(s) \coloneqq \bx(1-s).
\end{align}
It is a classical result~\cite{chen_iterated_1954} that the signature of a time-reversed path is obtained by the tensor inverse
\begin{align*}
    S(\mapact{\bx}{}{\tau}) = (S(\bx))^{-1}.
\end{align*}
When this is combined with the shuffle identity for paths (or through direct computation), this can be formulated at the level of monomials. In particular, given $I = (i_1, \ldots, i_m)$, let $\cev{I} = (i_m, \ldots, i_1)$. Then, the above identity is equivalent to
\begin{align*}
    S^I(\mapact{\bx}{}{\tau}) = S^{\cev{I}}(\bx).
\end{align*}

In the case of post-composition, a linear transformation $A \in L(V,W)$ induces linear transformations $A^{\otimes m} \in L(V^{\otimes m}, W^{\otimes m})$ for all $m \geq 1$. Then, by the linearity of integration, we find that the path signature transforms with respect to these induced transformations~\cite{friz_multidimensional_2010},
\begin{align*}
    S_m(A \bx) = A^{\otimes m} S_m(\bx),
\end{align*}
where $S_m(\bx) \in V^{\otimes m}$ is the full level $m$ path signature of $\bx$. 

\subsection{Action on the Domain}
The $\domdim$-dimensional hyperoctahedral group $B_{\domdim}$ is the group of symmetries of the $d$-cube, defined by the semi-direct product $B_{\domdim} \coloneqq \Z_2^{\domdim} \rtimes \Sigma_{\domdim}\label{eq:hyperoctahedral_def}$, where the action of $\Sigma_{\domdim}$ on $\Z_2^{\domdim}$ is given by permutation of components. Given $(\btau, \sigma) \in B_{\domdim}$, the element $\btau =(\tau_1, \ldots, \tau_{\domdim}) \in \Z_2^{\domdim}$ represents reflections of coordinates and $\sigma \in \Sigma_{\domdim}$ represents rotations of coordinates. This induces a map $\rho_{\btau, \sigma}: \square^{\domdim} \rightarrow \square^{\domdim}$ given by
\begin{align*}
    \rho_{\btau, \sigma}(\bs) = \left(s_{\sigma(1)}^{\tau_{\sigma(1)}}, \ldots, s_{\sigma(\domdim)}^{\tau_{\sigma(\domdim)}}\right),
\end{align*}
where $\bs \in \square^{\domdim}$ and
\begin{align*}
    s_j^0 \coloneqq s_j, \quad \quad s_j^1 \coloneqq 1-s_j
\end{align*}
for any $j \in [\domdim]$. We define the action of $(\btau, \sigma) \in B_{\domdim}$ on $\bx \in \Jlipmap$ by
\begin{align}
\label{eq:Bd_lip_action}
    \mapact{\bx}{\sigma}{\btau} \coloneqq \bx \circ \rho_{\btau, \sigma}. 
\end{align}
We note that this action is well-defined on Jacobian equivalence classes by the chain rule and the fact that the Jacobian of $\rho_{\btau, \sigma}$ is constant. We will begin studying the equivariance of the mapping space signature with respect to this action by considering the \emph{reflection} ($\Z_2^{\domdim}$) and \emph{rotation} ($\Sigma_{\domdim}$) components separately. 
\begin{definition}
Let $\pi \in \Sigma_m$, we define the \emph{reversal} of $\pi$, denoted $\cev{\pi} \in \Sigma_m$, by
\begin{align}
\label{eq:reversal_def}
    \cev{\pi}(i) \coloneqq \pi(m-i)
\end{align}
for all $i \in [m]$
\end{definition}

\begin{lemma}
    Suppose $\bx \in \Jlipmap$, $\btau = (\tau_1, \ldots, \tau_{\domdim}) \in \Z_2^d$, and $(\cP, \bpi) \in \cO_{\domdim,\coddim}^m \times \Sigma_m^\domdim$ is a level $m$ index. Define
    \begin{align*}
        \bpi^\btau = \left(\pi_1^{\tau_1}, \ldots, \pi_{\domdim}^{\tau_{\domdim}}\right),
    \end{align*}
    where
    \begin{align*}
        \pi_j^0 \coloneqq \pi_j, \quad \quad \pi_j^1 \coloneqq \cev{\pi}_j.
    \end{align*}
    Then,
    \begin{align}
        \Phi_m^{\cP, \bpi}(\mapact{\bx}{\id}{\btau}) = (-1)^{m|\btau|} \Phi_m^{\cP, \bpi^\btau}(\bx),
    \end{align}
    where $|\btau|$ is the number of nonzero entries of $\btau$. 
\end{lemma}
\begin{proof}
    We begin by considering the case where $\tau_{j'} =1$ and $\tau_j = 0$ for all $j \neq j'$. Given any $P \in \cO_{\domdim,\coddim}$, note that
    \begin{align*}
        J[(\mapact{\bx}{\id}{\btau})_P](s_1, \ldots, s_{\domdim}) = - J[\bx_P](s_1, \ldots, 1- s_{j'}, \ldots s_{\domdim}).
    \end{align*}
    Then, 
    \begin{align*}
        \Phi^{\cP, \bpi}_m(\mapact{\bx}{\id}{\btau}) = \int_{\intdom{m}{\domdim}{\bpi}} (-1)^m \prod_{i=1}^m J[\bx_{P_i}](t_{i,1}, \ldots,  1- t_{i,j'}, \ldots, t_{i,\domdim}) d\bt.
    \end{align*}
    Now, note that all of the reflected coordinates $t_{i,j'}$ lie in the same $m$-simplex in $\intdom{m}{\domdim}{\bpi}$, in particular, $(t_{1,j'}, \ldots, t_{m,j'}) \in \Delta^m_{\pi_{j'}}$. By performing the change of variables
    \begin{align*}
        1 - t_{i,j'} \mapsto t_{i,j'},
    \end{align*}
    the factor $\Delta^m_{\pi_{j'}}$ in the domain becomes $\Delta^m_{\cev{\pi}_{j'}}$, and thus we have
    \begin{align*}
        \Phi^{\cP, \bpi}_m(\mapact{\bx}{\id}{\btau}) = (-1)^m \Phi_m^{\cP, \bpi^\btau}(\bx).
    \end{align*}
    For the case of arbitrary $\btau \in \Z_2^d$, we repeat the same change of coordinates for each nonzero entry and obtain a factor of $(-1)^{m|\btau|}$.
\end{proof}

\begin{lemma}
    Suppose $\bx \in \Jlipmap, \sigma \in \Sigma_{\domdim}$, and $(\cP, \bpi) \in \cO_{\domdim,\coddim}^m \times \Sigma_m^\domdim$ is a level $m$ index. Define
    \begin{align}
        \bpi_\sigma \coloneqq \left(\pi_{\sigma(1)}, \ldots, \pi_{\sigma(d)}\right).
    \end{align}
    Then,
    \begin{align*}
        \Phi^{\cP, \bpi}_m(\mapact{\bx}{\sigma}{\bzero}) = (-1)^{m \cdot \sgn(\sigma)} \Phi_m^{\cP, \bpi_\sigma}(\bx).
    \end{align*}
\end{lemma}
\begin{proof}
    The action of $(\bzero, \sigma)$ on $\bx$ permutes the columns of the Jacobian by $\sigma$, and therefore
    \begin{align*}
        J[(\mapact{\bx}{\sigma}{\bzero})_P](s_1, \ldots, s_{\domdim}) = (-1)^{\sgn(\sigma)} J[\bx_P](s_{\sigma(1)}, \ldots, s_{\sigma(\domdim)}).
    \end{align*}
    Then,
    \begin{align*}
        \Phi_m^{\cP, \bpi}(\mapact{\bx}{\sigma}{\bzero}) = \int_{\intdom{m}{\domdim}{\bpi}} (-1)^{m \cdot\sgn(\sigma)} \prod_{i=1}^m J[\bx_{P_i}](t_{i, \sigma(1)}, \ldots, t_{i, \sigma(\domdim)}) \, d\bt.
    \end{align*}
    We apply the change of coordinates
    \begin{align*}
        t_{i, \sigma(j)} \mapsto t_{i, j},
    \end{align*}
    which transforms the domain from $\intdom{m}{\domdim}{\bpi}$ into $\intdom{m}{\domdim}{\bpi_{\sigma}}$.
\end{proof}

The transformation of the mapping space monomials with respect to the reflection and rotation actions on the domain of maps suggests the following right action of $B_{\domdim}$ on $\psphspace{\domdim}{}{V}$. Fix $m \in \N$, let $(\btau, \sigma) \in B_{\domdim}$, and let $\bpi \in \Sigma_m^{\domdim}$. We define
\begin{align}
\label{eq:Bd_perm_action}
    \bpi^\btau_\sigma \coloneqq \left(\pi_{\sigma(1)}^{\tau_{\sigma(1)}}, \ldots, \pi_{\sigma(\domdim)}^{\tau_{\sigma(\domdim)}}\right).
\end{align}
Next, let
\begin{align*}
    (\br^\bpi)_{\bpi \in \Sigma_m^{\domdim}} \in \prod_{\bpi \in \Sigma_m^{\domdim}} \left( \Lambda^\domdim V\right)^{\otimes m},
\end{align*}
where $\br^\bpi \in \Lambda^\domdim V$. Then, for any $m \in \N$, the action is given by
\begin{align}
\label{eq:Bd_H_action}
    (\br^\bpi) \cdot (\btau, \sigma) \mapsto (-1)^{m\cdot(|\btau| + \sgn(\sigma))}(\br^{\bpi_\sigma^\btau}),
\end{align}
and we extend this to $\psphspace{\domdim}{}{V}$. 

\begin{proposition}
\label{prop:Bd_equivariance}
    The mapping space signature $\Phi$ is $B_{\domdim}$-equivariant with respect to the action on $\Jlipmap$ defined in~\Cref{eq:Bd_lip_action} and the action on $\psphspace{\domdim}{}{V}$ defined in~\Cref{eq:Bd_H_action}.
\end{proposition}

\subsection{Action on the Codomain}
In this subsection, we consider the behavior of the mapping space signature under linear transformations of the codomain. In this section, we fix another finite-dimensional real vector space $W$. Given a linear transformation $A \in L(V,W)$, we obtain an induced map
\begin{align}
    A : \Jlipmap \rightarrow \JlipmapW,
\end{align}
where the induced map is well-defined on Jacobian equivalence classes via the chain rule and the fact that $A$ is a linear transformation and thus has a constant Jacobian.

Next, the linear transformation $A \in L(V,W)$ also induces a linear map
\begin{align*}
    \bA: \psphspace{\domdim}{}{V} \rightarrow \psphspace{\domdim}{}{W}.
\end{align*}
First, $A: V \rightarrow W$ extends to a linear map on the exterior product
\begin{align*}
    \Lambda^\domdim  A : \Lambda^\domdim  V \rightarrow \Lambda^\domdim  W.
\end{align*}
By taking tensor products of this transformation, and applying it component-wise in the product over permutations, we obtain the induced map on the degree $m$ component of the permutation tensor space,
\begin{align*}
    A^{(m)} \coloneqq \prod_{\bpi \in \Sigma_m^\domdim} \left(\Lambda^\domdim  A\right)^{\otimes m} : \phspace{\domdim}{(m)}{V} \rightarrow \phspace{\domdim}{(m)}{W} .
\end{align*}
Finally, we have
\begin{align}
    \bA \coloneqq (A^{(0)}, A^{(1)}, \ldots, A^{(m)}, \ldots) : \psphspace{\domdim}{}{V}  \rightarrow \psphspace{\domdim}{}{V} .
\end{align}

\begin{proposition}
\label{prop:cod_equivariance}
    Let $A \in L(V,W)$. Then, for any $\bx \in \Jlipmap$,
    \begin{align}
        \Phi(A\bx) = \bA \Phi(\bx).
    \end{align}
\end{proposition}
\begin{proof}
    It suffices to show that for any $\bpi \in \Sigma_m^\domdim$ and $m \in \N$, that
    \begin{align*}
        \Phi_m^\bpi(A\bx) = (\Lambda^\domdim  A)^{\otimes m}\Phi_m^\bpi(\bx). 
    \end{align*}
    We will use the basis-independent formulation of the signature. Applying the Jacobian minor operator $\hat{d}$ to $A\bx$, we find that
    \begin{align*}
        \hat{d}(A\bx) &= (\Lambda^\domdim  A) \hat{d}(\bx).
    \end{align*}
    Then, we have
    \begin{align*}
        \Phi_m^\bpi(A\bx) &= \int_{\intdom{m}{\domdim}{\bpi}} \hat{d}(A\bx)(\bt_1) \otimes \ldots \otimes \hat{d}(A\bx)(\bt_m) \\
        & = \int_{\intdom{m}{\domdim}{\bpi}} \left(\Lambda^\domdim  A\right)^{\otimes m} \left(\hat{d}\bx(\bt_1) \otimes \ldots \otimes \hat{d}\bx(\bt_m)\right) \\
        & = (\Lambda^\domdim  A)^{\otimes m}\Phi_m^\bpi(\bx),
    \end{align*}
    as desired.
\end{proof}

\begin{remark}
    In particular, if we consider $A \in \GLV \subset L(V,V)$, the induced linear transformation $\bA: \psphspace{\domdim}{}{V} \rightarrow \psphspace{\domdim}{}{V}$ is invertible, by taking the induced transformation of $A^{-1}$, and thus $\bA \in \GLHV$. Thus,~\Cref{prop:cod_equivariance} shows that the mapping space signature is $\GLV$ equivariant with respect to the standard $\GLV$ action on $\Jlipmap$ and the induced $\GLV$ action on $\psphspace{\domdim}{}{V}$. 
\end{remark}

\section{Shuffle Product}
\label{sec:ms_shuffle}

In this section, we consider the intrinsic algebraic structure of the mapping space monomials, generalizes shuffle product of the path signature.

\begin{definition}
\label{def:shuffle_perm}
    Let $p$ and $q$ be non-negative integers. A \textit{$(p,q)$-shuffle} is a permutation $\sigma \in \Sigma_{p+q}$ such that
    \begin{align*}
        \sigma^{-1}(1) < \sigma^{-1}(2) < \ldots < \sigma^{-1}(p)
    \end{align*}
    and
    \begin{align*}
        \sigma^{-1}(p+1) < \sigma^{-1}(p+2) < \ldots < \sigma^{-1}(p+q).
    \end{align*}
    We denote by $\Sh(p,q) \label{eq:shuffle_def}$ the set of $(p,q)$-shuffles. 
\end{definition}

\begin{theorem}
\label{thm:ms_shuffle}
	Let $m_1, m_2 \in \N$. Let $(\cP_1, \bpi_1) \in \cO_{\domdim,\coddim}^{m_1} \times \Sigma_{m_1}^\domdim$ be a level $m_1$ index and $(\cP_2, \bpi_2) \in \cO_{\domdim,\coddim}^{m_2} \times \Sigma_{m_2}^\domdim$ be a level $m_2$ index, where
	\begin{align*}
		\begin{array}{ll}
		\bpi_1 = ( \pi_{1,1}, \ldots, \pi_{1,k}) \in \Sigma_{m_1}^\domdim, & \bpi_2 = (\pi_{2,1}, \ldots, \pi_{2,k}) \in \Sigma_{m_2}^\domdim \\
		\cP_1  = (P_1, \ldots, P_{m_1}), & \cP_2  = (P_{m_1+1}, \ldots, P_{m_1+m_2}),
		\end{array}
	\end{align*}
    Then, define a level $m_1+m_2$ index $(\cP, \bpi) \in \cO_{\domdim,\coddim}^{m_1+m_2} \times \Sigma_{m_1+m_2}^\domdim$ by
	\begin{align*}
	    \cP &= (P_1, \ldots, P_{m_1+m_2}) \\
		\bpi &= (\pi_1, \ldots, \pi_{\domdim}) \in \Sigma_{m_1+m_2}^\domdim,
	\end{align*}
	where $\pi_j \in \Sigma_{m_1+m_2}$ is defined by
	\begin{align*}
	    \pi_j(i) = \left\{
	    \begin{array}{ll}
	    \pi_{1,j}(i) & : 1 \leq i \leq m_1 \\
	    \pi_{2,j}(i - m_1) &: m_1 + 1 \leq i \leq m_1 + m_2
	    \end{array}\right.
	\end{align*}
	Then, for any $\bx \in \lipmap$,
	\begin{align}
		\Phi_{m}^{\cP_1,\bpi_1}(\bx) \Phi_{m}^{\cP_2,\bpi_2}(\bx)  = \sum_{\bsigma} \Phi_{m}^{\cP, \bsigma \bpi}(\bx),
	\end{align}
	where the sum is indexed over $\bsigma \Sh(m_1, m_2)^{\domdim}$ and $\bsigma \bpi$ denotes the component-wise group multiplication in $\Sigma_{m_1+m_2}^\domdim$.
	
\end{theorem}
\begin{proof}
	Writing out the product of the two mapping space monomials, we have
	\begin{align}
	\label{eq:shuffle_integral}
		\Phi_{m}^{\cP_1, \bpi_1}(\bx) \Phi_{m}^{\cP_2, \bpi_2}(\bx) = \int_{\intdom{m_1}{\domdim}{\bpi_1} \times \intdom{m_2}{\domdim}{\bpi_2}} \prod_{i=1}^{m_1+m_2} J[\bx_{P_i}](\bt_i) d\bt,
	\end{align}
	where
	\begin{align*}
	    (t_{i,j})_{\substack{i=1, \ldots, m_1 \\ j = 1, \ldots, \domdim}} \in \intdom{m_1}{\domdim}{\bpi_1}, \quad \quad (t_{i,j})_{\substack{i=m_1+1, \ldots, m_1+m_2 \\ j = 1, \ldots, \domdim}} \in \intdom{m_2}{\domdim}{\bpi_2}.
	\end{align*}
	The domain of integration is
	\begin{align*}
		\intdom{m_1}{\domdim}{\bpi_1} \times \intdom{m_2}{\domdim}{\bpi_2} = \left( \prod_{j=1}^{\domdim} \Delta^{m_1}_{\pi_{1,j}}\right) \times \left( \prod_{j=1}^{\domdim} \Delta^{m_2}_{\pi_{2,j}}\right)  = \prod_{j=1}^{\domdim} \left( \Delta^{m_1}_{\pi_{1,j}} \times \Delta^{m_2}_{\pi_{2,j}}\right).
	\end{align*}
	Then, using the shuffle decomposition of $\Delta^{m_1}_{\pi_{1,j}} \times \Delta^{m_2}_{\pi_{2,j}}$ into the disjoint union
	\begin{align*}
		\Delta^{m_1}_{\pi_{1,j}} \times \Delta^{m_2}_{\pi_{2,j}} = \coprod_{\sigma_j \in \Sh(m_1, m_2)} \Delta^{m_1 + m_2}_{\sigma_j \pi_j},
	\end{align*}
	and letting $\bsigma = (\sigma_1, \ldots, \sigma_\domdim) \in \Sh(m_1,m_2)^\domdim$, we get
	\begin{align*}
		\intdom{m_1}{\domdim}{\bpi_1} \times \intdom{m_2}{\domdim}{\bpi_2}  = \coprod_{\bsigma} \prod_{j=1}^{\domdim} \Delta^{m_1 + m_2}_{\sigma_j \pi_j} = \coprod_{\bsigma} \intdom{m_1+m_2}{\domdim}{\bsigma\bpi}.
	\end{align*}
	This decomposition of $\intdom{m_1}{\domdim}{\bpi_1} \times \intdom{m_2}{\domdim}{\bpi_2}$ allows us to rewrite the integral in~\Cref{eq:shuffle_integral} as a sum of integrals over the domain $\intdom{m_1+m_2}{\domdim}{\bsigma\bpi}$,
	\begin{align*}
		\Phi_{m}^{ \cP_1, \bpi_1}(\bx) \Phi_{m}^{\cP_2,\bpi_2}(\bx) &= \sum_{\bsigma} \int_{\intdom{m_1+m_2}{\domdim}{\bsigma\bpi}}\prod_{i=1}^{m_1+m_2} J[\bx_{P_i}](\bt_i) d\bt = \sum_{\bsigma} \Phi_{m}^{\cP, \bsigma \bpi}(\bx).
	\end{align*}
\end{proof}

This shuffle product for the mapping space signature shows that the space of linear combinations of mapping space monomials is a subalgebra of $C(\Jlipmap, \R)$.

\begin{corollary}
\label{cor:shuffle_subalgebra}
    Let $\bx \in \Jlipmap$ and $\ell, \ell' \in \phspace{d}{}{V}$. Then there exists some $\ell'' \in \phspace{d}{}{V}$ such that
    \begin{align}
        \ip{\ell, \Phi(\bx)} \cdot \ip{\ell', \Phi(\bx)} = \ip{\ell'', \Phi(\bx)}.
    \end{align}
\end{corollary}

This product structure will used to show the universal and characteristic properties of the mapping space signature in~\Cref{sec:univ_char}.

\section{Composition of Maps}
\label{sec:ms_composition}

In this section, we study a possible generalization of Chen's identity to higher dimensions. One of the primary reasons for using cubical domains (rather than simplicial domains) as the higher dimensional generalization of the interval is the existence of natural concatenation operations. 

\begin{definition}
	Let $\bx, \by \in \lipmap$ and let $j \in [\domdim]$. If
	\begin{align}
	\label{eq:k_composable}
		\bx(s_1, \ldots, s_{j-1}, 1, s_{j+1}, \ldots, s_{\domdim}) = \by(s_1, \ldots, s_{j-1}, 0, s_{j+1}, \ldots, s_{\domdim})
	\end{align}
	for all $(s_1, \ldots, s_{j-1}, s_{j+1}, \ldots, s_{\domdim}) \in \square^{\domdim-1}$, then we say that $\bx$ and $\by$ are \emph{$j$-composable}. For $j$-composable maps $\bx, \by \in \lipmap$, the \emph{$j$-composition} $\bx *_j \by \in \lipmap$ is defined as
	\begin{align}
	\label{eq:ms_composition_formula}
		\bx *_j \by (s_1, \ldots, s_{\domdim}) \coloneqq \left\{
			\begin{array}{cl}
				\bx(s_1, \ldots, 2s_j, \ldots, s_{\domdim}) & : s_j \in [0, 1/2) \\
				\by(s_1, \ldots, 2s_j -1, \ldots, s_{\domdim}) &: s_j \in [1/2, 1]
			\end{array}
			\right.
	\end{align} 
\end{definition}

The $j$-composable condition ensures that the $j$-composition is continuous, and the resulting map is still Lipschitz. Furthermore, composition is well-defined on Jacobian equivalence classes since for any $j$-composable $\bx, \by \in \lipmap$, we have
\begin{align}
    J[\bx *_j \by] = J[\bx] *_j J[\by],
\end{align}
where the composition on the right is defined for $J[\bx], J[\by] \in \Linfmap$ in the same manner as~\Cref{eq:ms_composition_formula}, where we do not require the $j$-composability condition since they are not required to be continuous.

\begin{definition}
    Let $[\bx], [\by] \in \Jlipmap$ and let $j \in [\domdim]$. The equivalence classes $[\bx]$ and $[\by]$ are \emph{$j$-composable} if there exists a representative $\bx, \by$ of each which are $j$-composable as Lipschitz maps, satisfying~\Cref{eq:k_composable}. For $j$-composable equivalence classes $[\bx], [\by] \in \Jlipmap$, the \emph{$j$-composition} is the equivalence class of
    \begin{align*}
        [\bx] *_j [\by] \coloneqq [\bx *_j \by],
    \end{align*}
    where we use the $j$-composable representatives $\bx, \by$ on the right side. 
\end{definition}

Recall that Jacobian equivalence of paths corresponds to translation equivalence in the case of paths, and thus, two Jacobian equivalence classes of paths are always composable. Furthermore, the path signature is multiplicative with respect to path composition~\cite{chen_iterated_1954}; given $\bx, \by \in \Jlippath$, the multiplicative formula
\begin{align}
\label{eq:chens_identity_paths}
    S(\bx * \by) = S(\bx) \otimes S(\by)
\end{align}
holds, and is called \emph{Chen's identity}. In terms of path space monomials, suppose $P = (p_1, \ldots, p_m)\in [\coddim]^m \cong \cO_{1,\coddim}^m$, then we can write Chen's identity as
\begin{align*}
    S_m^P(\bx * \by) = S_m^P(\bx) + \sum_{i=1}^{m-1} S_{i}^{(p_1, \ldots, p_{i})}(\bx) \cdot S_{m-i}^{(p_{i+1}, \ldots, p_m)}(\by) + S^p_m(\by).
\end{align*}
Thus, the path space monomials of a composite path can be computed using the monomials of each of its constituent parts, a key property which leads to the definition of rough paths. While there is an obstruction to a direct generalization of Chen's identity to higher dimension, there is a variant which can be formulated for the identity mapping space signature (\Cref{ssec:identity_signature}). To begin, we will need the signature restricted to a subcube (\Cref{def:subcube}).

\begin{definition}
\label{def:mapping_space_subdomain}
    Suppose $(\ba, \bb) \in (\Delta^2)^\domdim$. Let $m \in \N$, suppose $(\cP, \bpi) \in \cO_{\domdim,\coddim}^m \times \Sigma_m^\domdim$ is a level $m$ index. Given $\bx \in \lipmap$, the \emph{mapping space monomial of $\bx$ with respect to $(\cP, \bpi)$ in the region $(\ba, \bb)$} is
    \begin{align}
        \Phi_m^{\cP, \bpi}(\bx)_{\ba, \bb} \coloneqq \int_{\intdom{m}{\domdim}{\bpi}(\ba,\bb)} \hat{d}^{P_1}x(\bt_1) \swedge \cdots \swedge \hat{d}^{P_m}x(\bt_m),
    \end{align}
    where $\intdom{m}{\domdim}{\bpi}(\ba,\bb)$ is given in~\Cref{def:integration_domain}.
\end{definition}

We illustrate the obstruction and the modification of Chen's identity with an example.

\begin{example}
We will consider $d=2$ for maps from the domain unit square, and level $m=2$. Suppose $\bx, \by \in \Lip(\square^2, V)$ are $1$-composable, and let $\bz = \bx *_1 \by$ be their $1$-composition. Let $(\cP, \bid) \in \cO_{2,\coddim}^2 \times \Sigma_2^2$, where $\cP = (P_1, P_2)$ and the permutation index $\bid \in \Sigma_2^2$ is the identity. Then,
\begin{align*}
	\Phi_2^{\cP, \bid}(\bz) = \int_{\intdom{2}{2}{\bid}} J[\bz_{P_1}](\bt_1) J[\bz_{P_2}](\bt_2) d\bt.
\end{align*}
In order to express the integral in terms of $\bx$ and $\by$, we decompose the first $\Delta^2$ factor in $\intdom{m}{\domdim}{\bid}$ as
\begin{align*}
	\Delta^2 = \Delta^2\left(0, 1/2\right) \cup \left( \left[0, 1/2\right] \times \left[1/2, 1\right]\right) \cup \Delta^2\left(1/2, 1\right).
\end{align*}
Using the definition of the $1$-composition in~\Cref{eq:ms_composition_formula}, the monomial above can be expressed as
\begin{align*}
	\Phi_2^{\cP, \bid}(\bz) &= \int_{\Delta^2 \times \Delta^2}J[\bx_{P_1}](\bt_1) J[\bx_{P_2}](\bt_2) d\bt  + \int_{\square^2 \times \Delta^2} J[\bx_{P_1}](\bt_1) J[\by_{P_2}](\bt_2) d\bt  \\
	&\hspace{20pt}+ \int_{\Delta^2 \times \Delta^2} J[\by_{P_1}](\bt_1) J[\by_{P_2}](\bt_2)  d\bt
\end{align*}

Note that the parametrization of $\square^2 \times \Delta^2$ in the second integral follows the same convention as the $\intdom{2}{2}{\bid}$ domain; namely we have
\begin{align}
    (t_{1,1}, t_{2,1}) \in \square^2, \quad(t_{1,2} < t_{2,2}) \in \Delta^2,\quad \bt_1 = (t_{1,1}, t_{1,2}), \quad\bt_2 = (t_{2,1}, t_{2,2}).
\end{align}

We can see that the first and third integrals in this decomposition correspond to monomials for $\bx$ and $\by$, and thus we have
\begin{align}
\label{eq:ms_composition_integral_form}
	\Phi_{2}^{\cP,\bid}(\bz) = \Phi_{2}^{\cP,\bid}(\bx) + \int_{\square^2 \times \Delta^2} J[\bx_{P_1}](t_{1,1}, t_{1,2}) J[\by_{P_2}](t_{2,1}, t_{2,2}) d\bt + \Phi_{2}^{\cP,\bid}(\by).
\end{align}
The obstruction to a formula such as Chen's identity is the remaining integral in the above expression. Due to the nontrivial relationship $t_{1,2} < t_{2,2}$ in the domain of integration, we cannot express this integral as a product of a monomial of $\bx$ and a monomial of $\by$, as is done in the case of path space monomials. However, we can express it as the \emph{integral} of a product of monomials. By expressing the domain of integration explicitly, the integral can be written as
\begin{align}
\int_{\square^2 \times \Delta^2} &J[\bx_{P_1}](t_{1,1}, t_{1,2}) J[\by_{P_2}](t_{2,1}, t_{2,2}) d\bt \label{eq:ms_composition_obstruction}\\
 & = \int_{t_{1,2}=0}^1 \left(\int_{t_{1,1}=0}^{1} J[\bx_{P_1}](t_{1,1}, t_{1,2}) dt_{1,1}\right) \left( \int_{t_{2,1}=0}^1 \int_{t_{2,2}=t_{1,2}}^1 J[\by_{P_2}](\bt_2) d\bt_2\right) dt_{1,2}. \nonumber
\end{align}
This integral can be expressed in terms of mapping space monomials of $\bx$ and $\by$ restricted to subdomains, defined in~\Cref{def:mapping_space_subdomain}. In particular, consider the derivative
\begin{align*}
    \frac{\partial}{\partial s} \Phi_{2}^{P_1, \id}(\bx)_{\bzero, (0,s)} &= \frac{\partial}{\partial s} \int_{t_1=0}^1 \int_{t_2=0}^{s} J[\bx_{P_1}](t_1, t_2) dt_2 dt_1 = \int_0^1 J[\bx_{P_1}](t_1, s) dt_1.
\end{align*}
Then, we can express the integral in~\Cref{eq:ms_composition_obstruction} as
\begin{align*}
    \int_{\square^2 \times \Delta^2} &J[\bx_{P_1}](\bt_1) J[\by_{P_2}](\bt_2) d\bt = \int_{0}^1 \frac{\partial}{\partial s} \Phi_{1}^{P_1, \id}(\bx)_{\bzero, (1,s)} \cdot \Phi_{1}^{P_2, \id}(\by)_{(0, s), \bone} \, ds.
\end{align*}
Thus, using~\Cref{eq:ms_composition_integral_form}, the monomial of the composition $\bx *_1 \by$ is
\begin{align*}
	\Phi^{\cP, \bid}_2(\bx *_1 \by) = \Phi_2^{\cP, \bid}(\bx) + \int_{0}^1 \frac{\partial}{\partial s} \Phi_{1}^{P_1, \id}(\bx)_{\bzero, (1,s)} \cdot \Phi_{1}^{P_2, \id}(\by)_{(0, s), \bone} ds + \Phi_{2}^{\cP,\bid}(\by).
\end{align*}
\end{example}

This composition identity can be established for any mapping space monomial where the permutation index is the identity $\bid \in \Sigma_m^{\domdim}$.

\begin{theorem}
\label{thm:composition}
	Let $\bx, \by \in \lipmap$ such that $\bx$ and $\by$ are $q$-composable for some $q \in [\domdim]$. Let $(\cP, \bid) \in \cO_{\domdim,\coddim}^m \times \Sigma_m^\domdim$, where $\cP = (P_1, \ldots, P_m)$, and let
	\begin{align*}
		\overline{\cP}_r = (P_1, \ldots, P_r), \quad \underline{\cP}_r = (P_{m-r+1}, \ldots, P_m)
	\end{align*}
	denote the first $r$ and last $r$ elements of $\cP$ respectively. Then, we will use
	\begin{align*}
		\hat{\bt}_q = (t_1, \ldots, t_{q-1}, t_{q+1}, \ldots, t_\domdim) \in \square^{\domdim-1}
	\end{align*}
	to denote an element of $\square^{d-1}$ with the $t_q$ coordinate omitted. Let
	\begin{align*}
		\hat{\bt}_q^0 & = (t_1, \ldots, t_{q-1}, 0, t_{q+1}, \ldots, t_\domdim)  \in \square^{\domdim}\\
		\hat{\bt}_q^1 & = (t_1, \ldots, t_{q-1}, 1, t_{q+1}, \ldots, t_\domdim) \in \square^{\domdim},
	\end{align*}
	denote the the corresponding coordinates on the $0$ or $1$ face of $\square^{\domdim}$ in the $q$-direction. Finally, let
	\begin{align*}
		\frac{\partial}{\partial \hat{\bt}_q} \coloneqq \frac{\partial}{\partial t_1} \ldots \frac{\partial}{\partial t_{q-1}} \frac{\partial}{\partial t_{q+1}} \ldots \frac{\partial}{\partial t_\domdim}.
	\end{align*}
	Then, we have
	\begin{align}
	\label{eq:mapping_signature_composition_formula}
		\Phi_{m}^{\cP,\bid}(\bx *_q \by) &= \Phi_{m}^{\cP,\bid}(\bx) + \Phi_{m}^{\cP,\bid}(\by) + \sum_{r=1}^{m-1} \int_{\hat{\bt}_q \in \square^{\domdim-1}} \frac{\partial}{\partial \hat{\bt}_q} \Phi_{r}^{ \overline{\cP}_r, \bid}(\bx)_{\bzero, \hat{\bt}_q^1} \cdot \Phi_{m-r}^{ \underline{\cP}_{m-r}, \bid}(\by)_{\hat{\bt}_q^0, \bone} \, d\hat{\bt}_q. 
	\end{align}
\end{theorem}
\begin{proof}
	We begin with the following decomposition of the standard $m$-simplex
	\begin{align}
	\label{eq:simplex_decomp_chen}
		\Delta^m = \coprod_{r=0}^m \Delta^r(0, 1/2) \times \Delta^{m-r}(1/2, 1).
	\end{align}
	The definition of $\Phi_{m}^{\cP, \bid}(\bx *_q \by)$ is an integral over $\intdom{m}{\domdim}{} = \prod_{j=1}^{\domdim} \Delta^m$ and we apply the decomposition in~\Cref{eq:simplex_decomp_chen} to the $q^{th}$ $\Delta^m$ in $\intdom{m}{\domdim}{}$. By the definition of $q$-composition in~\Cref{eq:ms_composition_formula}, this allows us to express the Jacobians in the integrand in terms of $\bx$ and $\by$ individually, after a change of variables which takes
	\begin{align*}
		\Delta^r(0, 1/2) \times \Delta^{m-r}(1/2, 1) \mapsto \Delta^r \times \Delta^{m-r}.
	\end{align*}
	We will use
	\begin{align}
		\cintdom{m}{\domdim}{q}{r} \coloneqq \left(\Delta^m\right)^{q-1}\times (\Delta^r \times \Delta^{m-r}) \times \left(\Delta^m\right)^{d-q},
	\end{align}
	to denote the resulting domain of integration. We will continue to use $\bt = (t_{i,j})_{i \in [m], j \in [\domdim]}$ as the coordinates of $\cintdom{m}{\domdim}{q}{r}$. In particular, for $\bt \in \cintdom{m}{\domdim}{q}{r}$, we have
	\begin{align*}
	    0 \leq t_{1,q} < \ldots < t_{r,q} \leq 1, \quad \quad 0 \leq t_{r+1, q} &< \ldots < t_{m, q} \leq 1.
	\end{align*}

	Then, applying the decomposition and the change of variables to the monomial of the composition, we have
	\begin{align}
		\Phi_{m}^{\cP, \bid}(\bx *_q \by) &= \Phi_{m}^{\cP,\bid}(\bx) + \Phi_{m}^{\cP,\bid}(\by) + \sum_{r=1}^{m-1} \int_{\cintdom{m}{\domdim}{q}{r}} \prod_{i=1}^r J[\bx_{P_i}](\bt_i) \cdot \prod_{i=r+1}^m J[\by_{P_i}](\bt_i) d\bt. \label{eq:ms_composition_first_decomposition}
	\end{align}
	The integrals in the summand can be rewritten by changing the order of integration. In particular, we integrate the $t_{r,1}, \ldots, t_{r,q-1}, t_{r, q+1}, \ldots, t_{r,\domdim}$ at the end. In the following decomposition, given $(\ba, \bb) \in (\Delta^2)^\domdim$, we let
	\begin{align*}
		\ocintdom{m}{\domdim}{q}(\ba, \bb) \coloneqq \prod_{j=1}^{q-1} \Delta^{m-1}(a_j, b_j) \times \Delta^m(a_q, b_q) \times \prod_{j=q+1}^{\domdim} \Delta^{m-1}(a_j, b_j).
	\end{align*}
	In order to preserve the labelling of the coordinates, we let
	\begin{align*}
		\hat{\bt}_q = (t_{r,1}, \ldots, t_{r,q-1}, t_{r,q+1}, \ldots, t_{r,\domdim})
	\end{align*}
	in the following expressions. By changing the order of integration, the integrals in the summand in~\Cref{eq:ms_composition_first_decomposition} can be expressed as
	\begin{align*}
		\int_{\cintdom{m}{\domdim}{q}{r}} &\prod_{i=1}^r J[\bx_{P_i}](\bt_i) \cdot \prod_{i=r+1}^m J[\by_{P_i}](\bt_i) d\bt \\
		& = \int_{\hat{\bt}_q \in \square^{\domdim-1}} \left(\int_{\ocintdom{r}{\domdim}{q}(\bzero, \hat{\bt}_q^1)} \prod_{i=1}^r J[\bx_{P_i}](\bt_i) d\bt\right) \cdot \left(\int_{\intdom{m-r}{\domdim}{}(\hat{\bt}_q^0, \bone)} \prod_{i={r+1}}^m J[\by_{P_i}](\bt_i) d\bt\right) d\hat{\bt}_q.
	\end{align*}
	By definition, the second integral in the integrand is
	\begin{align*}
		\Phi_{m-r}^{\underline{\cP}_{m-r}, \bid}(\by)_{\hat{\bt}_q^0, \bone} = \int_{\intdom{m-r}{\domdim}{}(\hat{\bt}_q^0, \bone)} \prod_{i={r+1}}^m J[\by_{P_i}](\bt_i) d\bt.
	\end{align*}
	Furthermore, it can be checked by iterated application of the fundamental theorem of calculus that
	\begin{align*}
		\frac{\partial}{\partial \hat{\bt}_q} \Phi_{r}^{ \overline{\cP}_r, \bid}(\bx)_{\bzero, \hat{\bt}_q^1} & = \frac{\partial}{\partial \hat{\bt}_q} \int_{\intdom{r}{k}{}(\bzero, \hat{\bt}_q^1)} \prod_{i=1}^r J[\bx_{P_i}](\bt_i) d\bt = \int_{\ocintdom{r}{\domdim}{q}(\bzero, \hat{\bt}_q^1)}\prod_{i=1}^r J[\bx_{P_i}](\bt_i) d\bt,
	\end{align*}
	which is the first integral in the integral. Thus, we can express  
	\begin{align*}
		\int_{\cintdom{m}{\domdim}{q}{r}} &\prod_{i=1}^r J[\bx_{P_i}](\bt_i) \cdot \prod_{i=r+1}^m J[\by_{P_i}](\bt_i) d\bt = \int_{\hat{\bt}_q \in\square^{d-1}} \frac{\partial}{\partial \hat{\bt}_q} \Phi_{r}^{ \overline{\cP}_r, \bid}(\bx)_{\bzero, \hat{\bt}_q^1} \cdot \Phi_{m-r}^{ \underline{\cP}_{m-r}, \bid}(\by)_{\hat{\bt}_q^0, \bone} \, d\hat{\bt}_q.
	\end{align*}
	Finally, substituting this integral back into~\Cref{eq:ms_composition_first_decomposition}, we obtain the desired composition formula. 
\end{proof}

\section{Parametrized Mapping Space Signature}
\label{sec:ms_parametrized_injectivity}

In this section, we define a parametrized variant of the mapping space signature and show that it is injective and continuous on $\Jlipmap$. Let 
\begin{align}
\label{eq:ov_def}
    \oV_\domdim \coloneqq \R^\domdim \times V.
\end{align}
We begin with the inclusion
\begin{align}
\label{eq:parametrized_lift}
	\iota: \lipmap & \rightarrow \Lip(\square^{\domdim}, \oV_\domdim) \\
	\bx(\bs) & \mapsto (\bs, \bx(\bs)). \nonumber
\end{align}
To simplify notation, we will denote $\obx \coloneqq \iota(\bx)$ for $\bx \in \lipmap$. Note that this inclusion map is not well-defined for Jacobian equivalence classes, which is further discussed in~\Cref{ssec:parametrized_continuity}. For the injectivity theorem on $\Jlipmap$, we will use a restriction of the mapping space signature for $\olipmap$, computed on $\obx$. In particular, we restrict the possible elements of the forms index by using a subset $W \subset \cO_{\domdim,\domdim+\coddim}$.

In order to describe the subset $W$, suppose $(e_{u,1}, \ldots, e_{u,\domdim})$ is an orthonormal basis for $\R^\domdim$ and $(e_{v,1}, \ldots, e_{v,\coddim})$ is an orthonormal basis for $V$. We will view $\cO_{\domdim, \coddim} \subset \cO_{\domdim, \domdim+\coddim}$ by letting $P \in \cO_{\domdim,\coddim}$ denote the orthonormal basis element
\begin{align}
    e^P \coloneqq e_{v, P(1)} \swedge \ldots \swedge e_{v, P(\domdim)} \in \Lambda^\domdim \oV_d.
\end{align}
Furthermore, we let $U \in \cO_{\domdim, \domdim+\coddim}$ denote the basis element
\begin{align}
\label{eq:U_basis}
    e^U \coloneqq e_{u,1} \swedge \ldots \swedge e_{u,\domdim}.
\end{align}
Finally, we let
\begin{align*}
    W = \{U\} \cup \cO_{\domdim,\coddim} \subset \cO_{\domdim, \domdim+\coddim}
\end{align*}
The Jacobian minor of $\obx$ with respect to $U$ only has information about the parametrization, and the Jacobian minor of $\obx$ with respect to some $P \in \cO_{\domdim,\coddim}$ remains the same as for $\bx \in \lipmap$. In other words, we have
\begin{align}
\label{eq:jacobiansW}
    J[\obx_U](\bs) = 1, \quad \quad J[\obx_P](\bs) = J[\bx_P](\bs).
\end{align}

\begin{definition}
The \emph{mapping space signature restricted to $W$} is defined by
\begin{align}
\label{eq:signatureW_def}
    \Phi_W &\coloneqq \left(\Phi_{W,m}\right)_{m\geq 0}: \olipmap \rightarrow \psphspace{\domdim}{}{\oV_d}\\
    \Phi_{W,m} &\coloneqq \left(\Phi_{m}^{\cP, \bpi}\right)_{\cP \in W^m, \bpi \in \Sigma_m^\domdim}: \olipmap \rightarrow \pvspace{\domdim}{(m)}{\oV_d}.
\end{align}
\end{definition}

In this definition, all the coordinates in $\psphspace{\domdim}{}{\oV_d}$ for basis elements $e_m^{\cP, \bpi}$ with $\cP \in \cO^m_{\domdim,\domdim+\coddim} - W$ are set to $0$

\begin{definition}
The \emph{parametrized mapping space signature} is defined to be the composition of $\Phi_W$ with the inclusion in~\Cref{eq:parametrized_lift},
\begin{align}
\label{eq:parametrized_signature_def}
    \oPhi \coloneqq \Phi_W \circ \iota : \Jlipmap \rightarrow \psphspace{\domdim}{}{\oV_d}.
\end{align}
\end{definition}

While the inclusion map $\iota$ is not well-defined on Jacobian equivalence classes, the parametrized signature depends only on the Jacobian minors of $\bx \in \lipmap$, based on the choice of elements in $W$, and the corresponding Jacobians in~\Cref{eq:jacobiansW}. Therefore, the parametrized signature is well-defined on $\Jlipmap$.

\subsection{Injectivity} In this section, we show that the parametrized signature is injective.
The idea behind the proof is to show that for any exponent vector $\bc = (c_1, \ldots, c_\domdim) \in \N^{\domdim}$ and any $P \in \cO_{\domdim,\coddim}$, there exists a linear combination of the parametrized monomials $\oPhi_m^{\cP, \bpi}$ which is equal to
\begin{align*}
	\int_{\square^{\domdim}} \bs^{\bc} J[\bx_P](\bs) d\bs
\end{align*}
for all $\bx \in \lipmap$, where $\bs^{\bc} \coloneqq s_1^{c_1} \ldots s_{\domdim}^{c_\domdim}$. In other words, the following lemma holds.

\begin{lemma}
\label{lem:injectivity_1}
	For any $\bc \in \N^{\domdim}$, and any $P \in \cO_{\domdim,\coddim}$, there exists a linear combination of mapping space monomials, denoted using $\ell_{\bc, P} \in \phspace{\domdim}{}{\oV_\domdim}$ and defined by
	\begin{align*}
	    \ell_{\bc,P} = \sum_{i=1}^r \lambda_i \, e^{\cP_i, \bpi_i},
	\end{align*}
	where $\lambda_i \in \R$ and $(\cP_i, \bpi_i) \in W^m \times \Sigma_m^\domdim$, such that
	\begin{align}
		\langle \ell_{\bc, P}, \oPhi(\bx) \rangle = \int_{\square^{\domdim}} \bs^{\bc} J[\bx_P](\bs) d\bs.
	\end{align}
\end{lemma}

Before we prove this lemma, we will show how it is used to prove the injectivity theorem. This will require a second lemma.

\begin{lemma}
\label{lem:injectivity_2}
	Suppose $g \in L^\infty(\square^{\domdim}, V)$. Then, if
	\begin{align}
		\int_{\square^{\domdim}} \bs^\bc g(\bs) d\bs = 0
	\end{align}
	for all $\bc \in \N^d$, then $g = 0$.
\end{lemma}
\begin{proof}
	Let $\R[s_1, \ldots, s_{\domdim}]$ denote the space of polynomials in $d$ variables viewed as functions $Q: \square^{\domdim} \rightarrow \R$. By the Stone-Weierstrass theorem, $\R[s_1, \ldots, s_{\domdim}]$ is dense in $C(\square^{\domdim}, \R)$, the space of continuous functions under the uniform topology. Next, we treat $\mu_g = g(\bs) d\bs$ as a finite Borel measure on $\square^{\domdim}$, where
	\begin{align*}
		\mu_g(\bs^\bc) = \int_{\square^{\domdim}} \bs^\bc g(\bs) d\bs.
	\end{align*}
	Then, since $\R[s_1, \ldots, s_{\domdim}]$ is dense in $C(\square^{\domdim}, \R)$, and $\mu_g(\bs^\bc) = 0$ for all $ \bc \in \N^d$, we have $\mu_g(f) = 0$ for all $f \in C(\square^{\domdim}, \R)$. The space of finite Borel measures on $\square^{\domdim}$ is the dual space of $C(\square^{\domdim}, \R)$, and the fact that $\mu_g(f) = 0$ for all $f \in C(\square^{\domdim}, \R)$ implies that $\mu_g = 0$. Thus, $g = 0$.
\end{proof}

Using these two lemmas, we can prove the parametrized injectivity theorem.

\begin{theorem}
\label{thm:parametrized_injectivity_first}
	The parametrized mapping space signature
	\begin{align*}
		\oPhi: \Jlipmap \rightarrow \psphspace{\domdim}{}{\oV_\domdim}
	\end{align*}
	is injective.
\end{theorem}
\begin{proof}
	Suppose $\bx, \by \in \lipmap$ such that $\oPhi(\bx) = \oPhi(\by)$. Let $\ell_{\bc, P} \in \phspace{\domdim}{}{\oV_\domdim}$ be the element from~\Cref{lem:injectivity_1}; then by our hypothesis, we have
	\begin{align*}
		\langle \ell_{\bc, P}, \oPhi(\bx)- \oPhi(\by) \rangle = \int_{\square^{\domdim}} \bs^{\bc} \left( J[\bx_P](\bs) - J[\by_P](\bs)\right) d\bs = 0
	\end{align*}
	for all $\bc \in \N^d$, and every $P \in \cO_{\domdim,\coddim}$. By~\Cref{lem:injectivity_2}, this implies that 
	\begin{align*}
		J[\bx_P](\bs) = J[\by_P](\bs),
	\end{align*}
	for all $P \in \cO_{\domdim,\coddim}$. Therefore, $\bx$ is Jacobian equivalent to $\by$.
\end{proof}

Finally, it remains to prove~\Cref{lem:injectivity_1}.

\begin{proof}[Proof of~\Cref{lem:injectivity_1}]
	Fix $\bc \in \N^{\domdim}$ and $P \in \cO_{\domdim,\coddim}$. The goal of this proof is to find a linear combination of mapping space monomials such that
	\begin{align}
	\label{eq:lincomb_injectivity}
		\sum_{i=1}^r \lambda_i \oPhi_m^{\bpi_i, \cP}(\bx) = \int_{\square^{\domdim}} \bs^\bc J[\bx_P](\bs) d\bs,
	\end{align}
	where $\lambda_i \in \R$, $\bpi_i \in \Sigma_m^{\domdim}$, and $\cP = (U, \ldots, U, P)$ (where $U$ is defined in~\Cref{eq:U_basis}). Note that $\cP$ is fixed in this sum.

	Given this particular $\cP$, and using the identifications of the Jacobians with respect to $U$ and $P$ from~\Cref{eq:jacobiansW}, the mapping space monomial is
	\begin{align}
	\label{eq:V_monomial}
		\oPhi_m^{\cP, \bpi} = \int_{\intdom{m}{\domdim}{\bpi}} J[\bx_P](t_{m,1}, \ldots, t_{m,\domdim}) d \bt.
	\end{align}
	Each component $\pi_j$ of $\bpi$ specifies the order of $t_{1,j}, \ldots, t_{m,j}$, and since the integrand has no $t_{i,j}$ dependence for $i < m$, the integral only depends on $\pi^{-1}_j(m)$. Thus, we let
	\begin{align}
		B = (b_1, \ldots, b_d) = \big( \pi_1^{-1}(m), \ldots, \pi_{\domdim}^{-1}(m)\big),
	\end{align}
	which we treat as a multi-index in $[m]$ of length $\domdim$. For $\bt_m = (t_{m,1}, \ldots, t_{m,\domdim})$, we let
	\begin{align}
	\label{eq:edm}
		E^{m}_B(\bt_m) = \prod_{j=1}^{\domdim} \Delta^{b_j-1}(0, t_{m,j}) \times \Delta^{m - b_j}(t_{m,j},1).
	\end{align}

	We let $\hat{\bt} = (t_{i,j})_{i \in [m-1], j \in [\domdim]}$ parametrize $E^{m}_B(\bt_m)$. Then, the integral in~\Cref{eq:V_monomial} can be rewritten as
	\begin{align*}
		\oPhi_m^{\cP, \bpi}(\bx) = \int_{\bt_m \in \square^{\domdim}} \left( \int_{\hat{\bt} \in E^{m}_B(\bt_m)} d\hat{\bt}\right) J[\bx_P](\bt_m) d\bt_m.
	\end{align*}
	The inner integral is simply the volume of the region $E^{m}_B(\bt_m)$, which we can compute by using the fact that the volume of the $n$-simplex $\Delta^n(a,b)$ is
	\begin{align*}
		\int_{\bt \in \Delta^n(a,b)} d\bt = \frac{(b-a)^n}{n!}.
	\end{align*}
	Then, using the definition in~\Cref{eq:edm}, the volume of $E^{m}_B(\bt_m)$ is given by 
	\begin{align}
		\int_{\hat{\bt} \in E^{m}_B(\bt)} d\hat{\bt} = \prod_{j=1}^d \frac{t_{j}^{b_j-1}}{(b_j-1)!}\frac{(1-t_{j})^{m-b_j}}{(m-b_j)!} = C_{m,B} Q_{m,B}(\bt),
	\end{align}
	where $C_{m,B} \in \R$ is a constant and
	\begin{align}
	    Q_{m,B}(\bt) = \prod_{j=1}^{\domdim} t_j^{b_j-1} (1-t_j)^{m-b_j}.
	\end{align}
	Therefore, our original problem of finding a linear combination of basic monomials in~\Cref{eq:lincomb_injectivity} is reduced to finding some $m \in \N$ and subsets $B_i \subset [m]$ of length $\domdim$ such that a linear combination of the polynomials $Q_{m,B_i}(\bt) \in \R[t_1, \ldots t_\domdim]$ satisfy
	\begin{align}
	\label{eq:qk_lincomb}
		\sum_{i=1}^r \lambda_i Q_{m, B_i}(\bt) = \bt^{\bc},
	\end{align}
	for some $\lambda_i \in \R$. The choice of $m \in \N$ is straightforward: we will choose $m = \max_{j \in [\domdim]} \{c_j\} + 1$. Now, we must find the subsets $B_i$ and coefficients $\lambda_i$. \medskip

	We can further reduce the problem since $Q_{m,B}(\bt)$ is defined in factored form, and thus we can consider each factor individually. Given $m,b \in \N$ such that $b < m$, let
	\begin{align}
		q_{m,b}(t) \coloneqq t^{b-1}(1-t)^{m-b}.
	\end{align}
    Thus, it suffices to show that for any $m \in \N$, and $c \in \N$ such that $c \leq m-1$ there exists a linear combination
	\begin{align}
	\label{eq:qb_lincomb}
		\sum_{i=1}^r \lambda_i q_{m, b_i}(t) = t^c.
	\end{align}
	We fix $m \in \N$, and prove this by induction on $c$ (in descending order). The base case will be $c=m-1$, which is straightforward since
	\begin{align*}
		q_{m,m}(t) = t^{m-1}
	\end{align*}
	Now, we assume the linear combination in~\Cref{eq:qb_lincomb} holds down to $c$, and we will show the case of $c-1$. Consider
	\begin{align*}
		q_{m,c}(t) = t^{c-1}(1-t)^{m-c},
	\end{align*}
	which contains a $t^{c-1}$ term, as well as monomials of degree strictly greater than $c-1$. By the induction hypothesis, each of these monomials can be written as a linear combination of the $q_{m,b}(t)$. Thus,~\Cref{eq:qb_lincomb} holds for any $m,c \in \N$ such that $c \leq m-1$. 
\end{proof}

\subsection{Continuity}
\label{ssec:parametrized_continuity}
Now that the parametrized mapping space signature is shown to be injective, we return to the notion of continuity. We begin by noting that the inclusion in~\Cref{eq:parametrized_lift} appending the parametrization is not a continuous map. In fact, it is not well defined on Jacobian equivalence classes. In particular, suppose $\bx = (x_1, \ldots, x_\coddim) \in \lipmap$. There exist order-preserving injections $P \in \cO_{\domdim, \domdim+\coddim} - W$ such that\footnote{In particular, this is the $P$ that corresponds to the basis element $e_{u,1} \swedge \ldots \swedge e_{u,j-1} \swedge e_{v, k} \swedge e_{u,j+1} \swedge \ldots \swedge e_{u,\domdim}$.}
\begin{align*}
    J[\obx_P](\bs) = \frac{\partial x_k}{\partial s_{j}}(\bs)
\end{align*}
for any $k \in [\coddim]$ and $j \in [\domdim]$. Thus, two maps $\bx, \by \in \lipmap$ which are Jacobian equivalent, $\bx \sim_\Jac \by$, may have different partial derivatives and will generally be in different equivalence classes after appending the parametrization, $\obx \nsim_J \overline{\by}$. For a specific case, consider~\Cref{ex:jacobian_equivalence}.\medskip

However, as discussed in the previous section, the parametrized mapping space signature is well defined on Jacobian equivalence classes due to the specific choice of forms $W$. Furthermore, the parametrized signature is still continuous, which can be proved using the same methods as~\Cref{prop:monomial_continuity}.

\begin{proposition}
\label{prop:parametrized_signature_continuity}
    The parametrized mapping space signature
    \begin{align}
        \oPhi: (\Jlipmap, \met_\infty) \rightarrow \psphspace{\domdim}{}{\oV_\domdim}
    \end{align}
    is continuous.
    Moreover, if $\bx, \by \in \Jlipmap$ and
    \begin{align*}
        L &> \max\{ \|\hat{d}\bx\|_\infty, \|\hat{d}\by\|_\infty, 1\} \\
        \epsilon &> \met_\infty(\bx, \by),
    \end{align*}
    then for all parametrized mapping space monomials, we have
    \begin{align*}
        \left|\oPhi_m^{\cP, \bpi}(\bx) - \oPhi_m^{\cP, \bpi}(\by)\right| < \frac{mL^{m-1}}{(m!)^{\domdim}} \epsilon.
    \end{align*}
\end{proposition}
\begin{proof}
    We note that $J[\obx_P](\bs), J[\oby_P](\bs) < L$ (in particular, this holds for $P=U$ since $L > 1$) and $|J[\obx_P](\bs) - J[\oby_P](\bs)| < \epsilon$ for all allowable $P \in W$. Thus, the same argument as~\Cref{prop:monomial_continuity} holds.
\end{proof}

\section{Universal and Characteristic Properties} \label{sec:univ_char}
We now have sufficient analytic and algebraic properties of the mapping space signature to prove that it is universal and characteristic, mimicking the case of the path signature. In this section we consider the space $\Jlipmap$ equipped with the Jacobian Lipschitz metric $\met_\infty$. Our approach in this section follows the arguments used for the path signature in~\cite{chevyrev_signature_2018}. We begin by recalling the formal definition of universal and characteristic maps for Hilbert spaces.

\begin{definition}
\label{def:univ_char}
    Suppose $\cX$ is a topological space, $\cH$ is a Hilbert space, $\cF \subset \R^\cX$ is a function class, and $\cF'$ the topological dual. Consider a feature map
    \begin{align*}
        \Phi: \cX \rightarrow \cH.
    \end{align*}
    Suppose that $\langle \ell, \Phi(\cdot)\rangle \in \cF$ for all $ \ell \in \cH$. We say that $\Phi$ is
    \begin{enumerate}
        \item \emph{universal} to $\cF$ if the map
        \begin{align*}
            \iota: \cH \rightarrow \R^\cX, \quad \ell \mapsto \langle \ell, \Phi(\cdot) \rangle
        \end{align*}
        has dense image in $\cF$; and
        \item \emph{characteristic} to $\cG \subset \cF'$, assuming that $\cG$ is a space of measures on $\cX$, if the map
        \begin{align*}
            \kappa : \cG \rightarrow \cH, \quad \mu \mapsto \int_X \Phi(x) d\mu(x)
        \end{align*}
        is injective.
    \end{enumerate}
\end{definition}

The universal and characteristic properties of a feature map are intricately linked through the following duality theorem, which is a reformulation of the corresponding result for kernels in~\cite{simon-gabriel_kernel_2018}.

\begin{theorem}[\cite{chevyrev_signature_2018}]
\label{thm:duality}
    Suppose that $\cF$ is a locally convex topological vector space. A feature map $\Phi$ is universal to $\cF$ if and only if $\Phi$ is characteristic to $\cF'$. 
\end{theorem}

This result allows us to show that a feature map is characteristic by showing that it is universal, which is generally easier to prove. In the case of ordinary monomials on $V$, where the domain is a compact subset $\cX \subset V$ as discussed in~\Cref{ssec:moment_map}, we can use to use the classical Stone-Weierstrass theorem to show universality with respect to $C(\cX, \R)$. Such an approach holds for the path signature, where the domain is restricted to a compact subset $\cX \subset \Jlippath$. One of the main difficulties when we move to the entire path space or mapping space is the fact that these spaces are not locally compact, and thus we cannot apply the standard Stone-Weierstrass theorem. The approach proposed in~\cite{chevyrev_signature_2018} is to consider the \emph{strict topology}~\cite{giles_generalization_1971} on the space of continuous bounded functions, for which a variant of the Stone-Weierstrass theorem holds. 

\begin{theorem}[\cite{giles_generalization_1971}]
\label{thm:strict_topology}
Let $\cX$ be a metrizable topological space, and $C_b(\cX, \R)$ denote the space of bounded continuous functions on $\cX$. 
\begin{enumerate}
    \item The strict topology\footnote{Let $\cX$ be a topological space. A function $\psi: \cX \rightarrow \R$ \emph{vanishes at infinity} if for all $\epsilon > 0$, there exists a compact set $K \subset \cX$ such that $\sup_{x \in \cX\backslash K} |\psi(x)| < \epsilon$. Let $B_0(\cX,\R)$ be the set of functions that vanish at infinity. The \emph{strict topology}~\cite{giles_generalization_1971, chevyrev_signature_2018} on $C_b(\cX,\R)$ is the topology generated by seminorms
    \begin{align*}
        p_\psi(f) = \sup_{x \in \cX} |f(x) \psi(x)|, \quad \psi \in B_0(\cX, \R).
    \end{align*}}
    on $C_b(\cX, \R)$ is weaker than the uniform topology and stronger than the topology of uniform convergence on compact sets.
    \item If $\cF_0$ is a subalgebra of $C_b(\cX, \R)$ such that for all $x, y\in \cX$, there exists some $f \in \cF_0$ such that $f(x) \neq f(y)$ ($\cF_0$ separates points), and for all $x \in \cX$, there exists some $f \in \cF_0$ such that $f(x) \neq 0$, then $\cF_0$ is dense in $C_b(\cX)$ under the strict topology.
    \item The topological dual of $C_b(\cX, \R)$ equipped with the strict topology is the space of finite regular Borel measures on $\cX$. 
\end{enumerate}
\end{theorem}

In order to use the strict topology, we must first normalize the mapping space signature such that its constituent monomials, and the finite linear functionals $\ip{\ell, \Phi(\cdot)}: \Jlipmap \rightarrow \R$ for some $\ell \in \phspace{\domdim}{}{V}$, are \emph{bounded} continuous functions. 

\begin{definition}
    Let $\cH$ be a graded Hilbert space. Given $\dgnorm > 0$, the \emph{graded scaling by $\dgnorm$} is a linear map $\delta_\dgnorm : \cH \rightarrow \cH$ defined by $\delta_\dgnorm(\br_m) \coloneqq \dgnorm^m \br_m\label{eq:graded_scaling_def}$ on each degree $m$ vector $\br_m \in \cH$. 
\end{definition}

Suppose $\bx \in \lipmap$ and $\dgnorm \in \R$. We denote by $\dgnorm \bx \in \lipmap$ the pointwise scaling of $\bx$ by $\dgnorm$. Because the Jacobian minor operator scales according to
\begin{align}
\label{eq:scaling_jacobian_minor}
    \hat{d}(\dgnorm \bx) = \dgnorm^\domdim \hat{d}\bx,
\end{align}
scaling of maps is well-defined on Jacobian equiavlence classes.

\begin{proposition}
\label{prop:scaling}
    Let $\bx \in \lipmap$ and $\dgnorm \in \R$. The mapping space signature of a scaled map is
    \begin{align}
        \Phi(\dgnorm\bx) = \delta_{\dgnorm^\domdim} \Phi(\bx).
    \end{align}
\end{proposition}
\begin{proof}
    It suffices to show that this is the appropriate scaling for each $m \in \N$ and $\bpi \in \Sigma_m^\domdim$. Indeed, by~\Cref{eq:scaling_jacobian_minor}, we have
    \begin{align*}
        \Phi_m^\bpi(\dgnorm \bx) &= \int_{\intdom{m}{\domdim}{\bpi}} \hat{d}(\dgnorm\bx)(\bt_1) \swedge \ldots \swedge \hat{d}(\dgnorm\bx)(\bt_m) = \dgnorm^{\domdim m}\Phi_m^\bpi(\bx).
    \end{align*}
\end{proof}

\begin{definition}
    Suppose $\cH$ is a graded Hilbert space. A \emph{graded normalization} is a continuous injective map of the form
    \begin{align}
    \label{eq:graded_normalization_def}
        \gnorm : \cH &\rightarrow \{ \br \in \cH \, : \, \|\br\| \leq C\} \\
        \br &\mapsto \delta_{\dgnorm(\br)}(\br)\nonumber
    \end{align}    
    where $C>0$ is a constant and $\dgnorm : \cH \rightarrow (0, \infty)$ is a function called the \emph{scaling function}.
\end{definition}

Such graded normalizations exist, as shown in~\cite[Appendix A]{chevyrev_signature_2018}. In the following theorem, we will consider the \emph{normalized signature}, defined by $\gnorm \circ \Phi$, where $\gnorm$ is a graded normalization. By~\Cref{prop:scaling}, for any $\bx \in \Jlipmap$, this is equivalent to
\begin{align}
\label{eq:graded_normalization_rescaling}
    \gnorm\circ \Phi(\bx) = \Phi(\dgnorm(\bx)^{1/\domdim} \bx),
\end{align}
where $\dgnorm: \psphspace{\domdim}{}{V} \rightarrow (0,\infty)$ is the corresponding scaling function. We can now prove our main theorem.

\begin{theorem}
\label{thm:parametrized_univ_char}
    Let $\gnorm: \psphspace{\domdim}{}{\oV_\domdim} \rightarrow \psphspace{\domdim}{}{\oV_\domdim}$ be a graded normalization. The normalized parametrized mapping space signature
    \begin{align*}
        \gnorm \circ \oPhi : \cL \rightarrow \psphspace{\domdim}{}{\oV_\domdim}
    \end{align*}
    \begin{enumerate}
        \item [\bf (AN1)] is continuous, injective and has factorial decay: $\|(\gnorm \circ \oPhi)_m(\bx)\| \leq C^m/(m!)^\domdim$ for a constant $C>0$ depending on $\bx$;
        \item [\bf (AN2)] is universal to $\cF \coloneqq C_b(\cL, \R)$ equipped with the strict topology; and
        \item [\bf (AN3)] is characteristic to the space of finite regular Borel measures on $\cL$.
    \end{enumerate}
\end{theorem}
\begin{proof}
    Continuity and injectivity in the first point is a consequence of~\Cref{prop:parametrized_signature_continuity},~\Cref{thm:parametrized_injectivity_first}, and the fact that graded normalizations are continuous and injective. Furthermore, factorial decay follows from the proof of~\Cref{prop:factorial_decay} and the definition of a graded normalization.

    Next, for any $\ell \in \psphspace{\domdim}{}{\oV_\domdim}$, the function $\langle \ell, \gnorm \circ \oPhi(\cdot) \rangle : \cL \rightarrow \R$ is an element of $\cF$ since the parametrized signature is continuous (\Cref{prop:parametrized_signature_continuity}), and we have applied a graded normalization. Let
    \begin{align*}
        \cF_0 \coloneqq \left\{\langle \ell, \gnorm \circ \oPhi(\cdot) \rangle : \cL \rightarrow \R \, : \, \ell \in \phspace{\domdim}{}{\oV_\domdim}\right\}
    \end{align*}
    be the functionals given by a finite linear combination of basis elements $e^{\cP, \bpi}$. The function space $\cF_0$ forms a subalgebra of $\cF$ through the shuffle product (~\Cref{cor:shuffle_subalgebra}). Next, since the normalized parametrized signature $\gnorm \circ \oPhi$ is injective, $\cF_0$ separates points. Furthermore, for all $\bx \in \cL$, we have $\langle 1, \gnorm \circ \oPhi(\bx)\rangle = 1 \neq 0$. Thus, $\cF_0$ satisfies all the conditions of (2) in~\Cref{thm:strict_topology}, so $\cF_0$ is dense in $\cF$ and $\gnorm \circ \oPhi$ is universal to $\cF$.
    
    Then, by~\Cref{thm:duality}, $\gnorm \circ \oPhi$ is also characteristic to $\cF'$, but according to point (3) of~\Cref{thm:strict_topology}, the dual $\cF'$ is the space of finite regular Borel measures on $\cL$, which completes the proof. 
\end{proof}

\begin{remark}
    In the setting where we choose $\cX \subset \Jlipmap$ to be compact, we can show that the unnormalized parametrized signature is universal to $C(\cX, \R)$ with the uniform topology and characteristic to its dual (regular Borel measures on $\cX$) using the same method as the $\domdim=0$ and $\domdim=1$ case. In particular, we can apply the classical Stone-Weierstrass theorem in this setting to prove the universal property, and the characteristic property follows by duality (\Cref{thm:duality}).
\end{remark}

\section{Conclusions and Future Work}

In this article, we return to the topological origins of the path space monomials as $0$-cochains in Chen's iterated integral cochain model in order to motivate an extension to mapping spaces from higher dimensional domains. Using this perspective, we established the novel notion of a mapping space signature, derived from the $0$-cochains of a cubical generalization of Chen's cochain construction for mapping spaces. The mapping space signature carries rich analytic and algebraic properties which largely coincide with the properties of the moment map and path signature. These are summarized in the following theorem and~\Cref{thm:parametrized_univ_char}.

\begin{theorem}
\label{thm:formal_ms_properties}
    The mapping space signature $\Phi: \Jlipmap \rightarrow \psphspace{\domdim}{}{V}$, where $\psphspace{\domdim}{}{V} \subset \pspvspace{\domdim}{}{V}$ is the Hilbert space of finite norm elements in $\pspvspace{\domdim}{}{V}$, is/has:
    \begin{enumerate}
        \item [\bf (AN1)] \emph{(\Cref{prop:factorial_decay}, \Cref{cor:signature_continuty})} factorial decay at each level $m \in \N$:
        \begin{align*}
            \|\Phi_m(\bx)\|^2 \leq \binom{\coddim}{\domdim}^m \frac{\|J[\bx]\|_{\infty}^{\domdim m}}{(m!)^\domdim}
        \end{align*}
        and is continuous with the Jacobian Lipschitz metric, $\met_\infty$, on $\Jlipmap$;
        \item [\bf (AL0')] \emph{(\Cref{prop:ms_reparametrization_invariance})} coordinate-wise reparametrization invariant: if $\psi_i: \square^1 \rightarrow \square^1$ is a Lipschitz, orientation-preserving bijection, and 
        \[
        \psi(s_1, \ldots, s_\domdim) = (\psi_1(s_1), \ldots, \psi_\domdim(s_\domdim)),
        \]
        then $\Phi(\bx) = \Phi(\bx \circ \psi)$;
        \item [\bf (AL1')] \emph{(\Cref{thm:composition})} a modified Chen's identity given by~\Cref{eq:mapping_signature_composition_formula} when the permutation index is restricted to be the identity;
        \item [\bf (AL2)] \emph{(\Cref{thm:ms_shuffle})} generalized shuffle product structure: the finite linear functionals $\ell \in \phspace{\domdim}{}{V}$, viewed as functions 
        \[
        \langle \ell, \Phi(\cdot) \rangle : \Jlipmap \rightarrow \R,
        \]
        forms a subalgebra of $C(\Jlipmap, \R)$;
        \item [\bf (AL3)] \emph{(\Cref{prop:Bd_equivariance}, \Cref{prop:cod_equivariance})} equivariant with respect to both the hyperoctahedral group, $B_\domdim$, action on the domain and $\GLV$ action on the codomain.
    \end{enumerate}
\end{theorem}

In particular, the universal and characteristic properties on Jacobian equivalence classes of maps allow us to effectively study functions and measures on mapping spaces. However, the more complex structure of the mapping space structure leads to obstructions in fully generalizing some properties of the path signature, namely {\bf (AL0)} and {\bf (AL1)}.
\begin{itemize}
    \item [\bf(AL0)] It is well known that the path signature is invariant up to tree-like equivalence of paths~\cite{hambly_uniqueness_2010}. In~\Cref{prop:ms_reparametrization_invariance}, we showed that the mapping space signature is invariant under independent reparametrizations of each dimension in the domain, and in~\Cref{cor:trivial_ms_signature}, we constructed a class of maps with trivial mapping space signature. Is it possible to completely characterize the invariance of the mapping space signature? In other words, what is the higher dimensional analogue of tree-like equivalence?
    \item [\bf(AL1)] The path signature preserves the concatenation of paths through Chen's identity, given in~\Cref{eq:chens_identity_paths}. In~\Cref{sec:ms_composition}, we showed that there is an obstruction to a direct generalization of Chen's identity to higher dimensions due to the inability to factor the integrand into components which depend only on a single simplex in the domain. In the case of the identity permutation, we obtain a variation of Chen's identity involving an additional integration step. Is it possible to obtain a general composition formula for all monomials?
\end{itemize}

The path signature has recently become a powerful tool in both pure mathematics, primarily through the development of rough paths~\cite{lyons_differential_2007}, and applied mathematics. The wide applicability of the path signature suggests further study into how the mapping space signature can provide meaningful generalizations in these two domains.
\begin{itemize}
    \item Can we use the mapping space signature to define a higher dimensional analogue of rough paths?
    \item For computational purposes, the classical path signature enjoys the factorial decay property but suffers from the combinatorial explosion of coordinates of tensors. This extends to the mapping space signature but the factorial explosion in coordinates is even worse. 
    Techniques that have been developed to handle this for the path signature (such as kernelization, low-rank techniques, and randomization) could be extended to the case of $d \ge 2$. 
\end{itemize}

Throughout this article, we restricted our focus to Lipschitz maps valued in a finite-dimensional inner product space. However, there are three natural directions of generalization with interesting implications to explore.
\begin{itemize}
    \item The recent generalization of Young integration to higher dimensional differential forms~\cite{zust_integration_2011, stepanov_towards_2020, alberti2019integration} suggests an extension of the mapping space signature to lower regularity $1/p$-H\"older maps for some $p > 1$, analogous to the case of the path signature with $1 \leq p < 2$. While the 1-dimensional Young integral utilizes path increments to sidestep differentiation, the multi-dimensional extension~\cite{zust_integration_2011} uses volume increments to avoid the computation of Jacobians. 
    \item The path signature is defined for paths valued in a Banach space~\cite{lyons_system_2007}, and our coordinate-free definitions here could be extended to the setting of Banach spaces.

    \item Many properties of the path signature can be extended to paths valued in Lie groups \cite{lee_path_2020} and thus it is natural to ask if one can replace the co-domain $V$ by a non-linear space $X$. 
    Indeed, the mapping space monomials derived from the cubical mapping space construction in~\Cref{ssec:ms_0cochains_Rn} can also be defined on an arbitrary manifold $X$, where we must choose a set $\omega_1, \ldots, \omega_m \in \Omega^\domdim(X)$ of differential forms, rather than the standard $\domdim$-forms used throughout this article. If we let $\cP = (\omega_1, \ldots, \omega_m)$, choose $\bpi \in \Sigma_m^\domdim$, and let $\bx \in \smmapman$, then we can define the mapping space monomial of $\bx$ with respect to $(\cP, \bpi)$ as
    \begin{align*}
        \Phi_m^{\cP, \bpi}(\bx) \coloneqq \int_{\intdom{m}{\domdim}{\bpi}} \bigwedge_{i=1}^m \bx^*\omega_i(\bt_i).
    \end{align*}
\end{itemize}

\appendix

\section{A Cubical Variant of Chen's Mapping Space Construction}
\label{apx:cubical_chen_construction}

Chen's original cochain model was developed for path spaces and loop spaces~\cite{chen_iterated_1977}, and was more recently generalized to the setting of mapping spaces~\cite{patras_cochain_2003, ginot_chen_2010} in the language of simplicial sets~\cite{friedman_survey_2012, may_simplicial_1992}. Because we focus on the case of cubical domains, due to its ubiquity in analytic and data applications, we reformulate the mapping space cochain construction using cubical sets~\cite{brown_nonabelian_2011}. This appendix provides a more detailed exposition of the material in~\Cref{sec:chen_construction}. 

Chen's construction is a method to associate a differential form on smooth mapping spaces $C^\infty(|K|, X)$ by using a collection of differential forms on the codomain $X$. Here, $|K|$ is a space constructed from a combinatorial structure called a \emph{cubical set}. However, in order to discuss differential forms, we must place a smooth structure on $|K|, X$, and the mapping space $C^\infty(|K|, X)$. Chen introduced \emph{differentiable spaces}~\cite{chen_iterated_1977}, which generalizes the notion of smooth manifolds, allows one to specify a smooth structure on any set, and enjoys strong category-theoretic properties~\cite{baez_convenient_2011}. We refer the reader to~\cite{chen_iterated_1977, giusti_iterated_2020, baez_convenient_2011} for formal definitions, and we keep our exposition in this appendix independent from such details.

We begin with preliminary background on cubical sets, and then introduce the general reformulation of Chen's construction for mapping spaces in the setting of cubical sets. Finally, we restrict our focus to $0$-cochains on $\smmapR$ and demonstrate how the mapping space monomials arise from this construction.

\subsection{Cubical Sets}

Cubical sets provide a combinatorial and hierarchical description of topological spaces. They come with face, degeneracy, and connection maps that specify how combinatorial cubes can be ``glued together.'' A complete reference is found in~\cite{brown_nonabelian_2011}.

\begin{definition}
\label{def:cubical_set}
    A \emph{cubical set}\footnote{More precisely, this definition is for a \emph{cubical set with connection}~\cite{brown_nonabelian_2011}. Because we exclusively work with cubical sets with connection, we call them cubical sets for simplicity.} $K_\bullet$ is a collection of sets $\{K_p\}_{p=0}^\infty$, equipped with a collection of maps:
    \begin{itemize}
        \item \textbf{face maps}: $d_i^\epsilon : K_{p} \rightarrow K_{p-1}\label{eq:face}$, for $i \in [p]$ and $\epsilon \in \{0,1\}$,
        \item \textbf{degeneracy maps}: $s_i : K_{p} \rightarrow K_{p+1}\label{eq:degeneracy}$, for $i \in [p+1]$, and
        \item \textbf{connection maps}: $g_i: K_{p} \rightarrow K_{p+1}\label{eq:connection}$ for $i \in [p-1]$ and $p \geq 1$,
    \end{itemize}
    which satisfy \emph{cubical compatibility conditions}~\cite{brown_nonabelian_2011}. We say that a cubical set $K_\bullet$ is \emph{finite} if $K_p$ is a finite set for each $p \in \N$.
\end{definition}

An element $a \in K_p$ is called a \emph{$p$-cube}, where the face maps $d^\epsilon_i(a) \in K_{p-1}$ determine the $(p-1)$-cubes that make up the boundary. The degeneracy and connection maps constitute additional combinatorial structure of the cubical set, leading to higher dimensional depiction of lower dimensional cubes. A cube $a \in K_p$ is called \emph{nondegenerate} if there does not exist any $b \in K_{p-1}$ such that $s_i(b) = a$ or $g_i(b) = a$ for any compatible $i$; otherwise, the $p$-cube $a$ is called \emph{degenerate}. Note that any degenerate cube $a \in K_p$ can be written a canonical form as
\begin{align}
\label{eq:canonical_degenerate_form}
    a = s_{i_k} s_{i_{k-1}} \ldots s_{i_1} g_{j_m} g_{j_{m-1}} \ldots g_{j_1} b
\end{align}
where 
\begin{align}
\label{eq:canonical_form_indices}
    1 \leq i_1 < \ldots < i_k \leq p, \quad 1 \leq j_1 < \ldots < j_m \leq p-k-1
\end{align}
and $b \in K_{p-k-m}$ is a nondegenerate cube. In this case, we say that $a$ is a \emph{degeneracy of $b$}.

Associated to each cubical set $K_\bullet$ is its \emph{geometric realization} $|K|$, which is a topological space constructed as the disjoint union of all cubes in $K_\bullet$ glued along the face, degeneracy, and connection maps.
\begin{definition}
\label{def:cocubical_maps}
Define the \emph{coface}, \emph{codegeneracy} and \emph{coconnection} maps as
\begin{itemize}
    \item \textbf{coface maps:} $\delta_i^\epsilon: \square^{p-1} \rightarrow \square^p$ for $i \in [p]$ and $\epsilon \in \{0,1\}$
    \begin{align}
    \label{eq:coface}
        \delta_i^\epsilon(x_1, \ldots, x_{p-1}) \coloneqq (x_1, \ldots, x_{i-1}, \epsilon, x_i, \ldots, x_{p-1});
    \end{align}
    \item \textbf{codegeneracy maps:} $\sigma_i: \square^{p+1} \rightarrow \square^p$ for $i \in [p+1]$
    \begin{align}
    \label{eq:codegeneracy}
        \sigma_i(x_1, \ldots, x_{p+1}) \coloneqq (x_1, \ldots, x_{i-1}, x_{i+1}, \ldots, x_{p+1});
    \end{align}
    \item \textbf{coconnection maps:} $\bx_i: \square^{p+1} \rightarrow \square^p$, for $i \in [p]$ and $p \geq 1$
    \begin{align}
    \label{eq:coconnection}
        \gamma_i(x_1, \ldots, x_{p+1}) \coloneqq (x_1, \ldots, x_{i-1}, \max\{x_i, x_{i+1}\}, x_{i+2}, \ldots, x_{p+1}).
    \end{align}
\end{itemize}
These maps satisfy \emph{cocubical compatibility conditions}~\cite{brown_nonabelian_2011}, which are dual to the cubical compatibility conditions.
\end{definition}

\begin{definition}
    Let $K_\bullet$ be a cubical set. The \emph{geometric realization of $K_\bullet$}, denoted $|K|$, is a topological space defined by
    \begin{align}
    \label{eq:geometric_realization}
        |K| \coloneqq \left.\left( \coprod_{p=0}^\infty K_p \times \square^p \right)\right/\sim,
    \end{align}
    where 
    \begin{align}
    \label{eq:geometric_realization_equivalence_relation}
        (d_i^\epsilon a, \bx) \sim (a, \delta_i^\epsilon \bx), \quad (s_i a, \by) \sim (a, \sigma_i \by), \quad (g_i a, \by) \sim (a, \gamma_i \by),
    \end{align}
    for any $a \in K_p$, $\bx \in \square^{p-1}$, and $\by \in \square^{p+1}$.
\end{definition}

In~\Cref{eq:geometric_realization}, the term $K_p \times \square^p$ associates a topological $p$-cube $\square^p$ to each element $a \in K_p$, which we view as a combinatorial $p$-cube. The equivalence relation then identifies the various disjoint topological cubes according to the face, degeneracy, and connection maps. In particular, the topological cube $\square^p$ associated to each degeneracy $a \in K_p$ of a nondegenerate combinatorial cube $b \in K_{p-q}$ is identified with the topological cube $\square^q$ associated with $b$. Each combinatorial $p$-cube in $K_\bullet$ is equipped with a corresponding \emph{evaluation map}.

\begin{definition}
\label{def:evaluation_map_a}
    Suppose $K_\bullet$ is a cubical set, and $a \in K_p$. The \emph{evaluation map with respect to $a$} is defined by the composition
    \begin{align}
    \label{eq:evaluation_map_a}
        \eta_a: \square^p \hookrightarrow \coprod_{p=0}^\infty K_p \times \square^p \twoheadrightarrow |K|,
    \end{align}
    where the first map is given by the injection $\square^p \mapsto \{a\} \times \square^p$, and the second map is the quotient by the equivalence relation in~\Cref{eq:geometric_realization_equivalence_relation}.
\end{definition}

\begin{example}
\label{ex:2cube}
    Here, we provide describe the standard cubical model, $Z^2_\bullet$ for the $2$-dimensional cube $\square^2$. We begin by defining the non-degenerate cubes, which are labelled as elements of $E^2 \coloneqq \{0, 1, e\}^2$, where the dimension of a non-degenerate cube $a_1 a_2 \in E^2$ is the number appearances of $e$ in the word. The $0$-cubes in $Z_0$ are all non-degenerate, and represent the four vertices, which we represent visually as follows.
    \begin{figure}[hbt!]
    \centering
    		\includegraphics[scale=0.5]{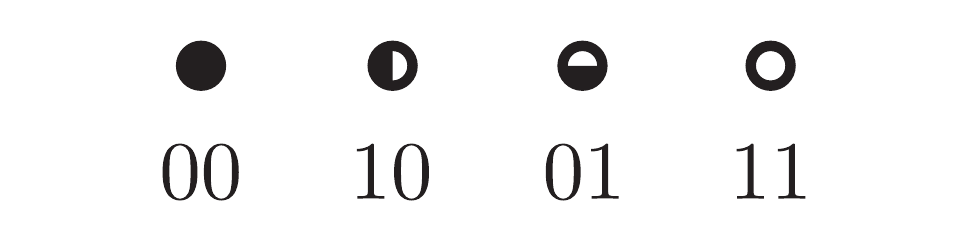}
    \end{figure}
    
    Next, the non-degenerate $1$-cubes in $Z^2_1$ represent the four edges of the square, visualized below.
    \begin{figure}[hbt!]
    \centering
    		\includegraphics[scale=0.5]{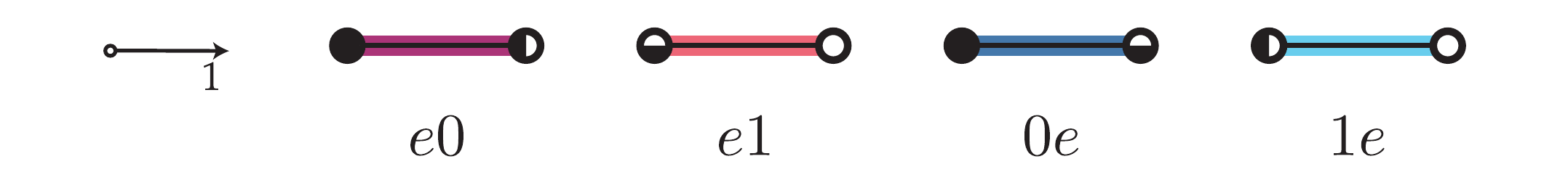}
    \end{figure}
    
    This visual representation allows us to easily read off the faces of the cube; for instance
    \begin{align*}
        \delta_1^0(e0) = 00, \quad \quad \delta_1^1(e0) = 10.
    \end{align*}
    Symbolically, the face map $\delta_1^\epsilon$ changes the unique $e$ into the $\epsilon$. Finally, we have the unique non-degenerate $2$-cube in $Z_2$.
    \begin{figure}[hbt!]
    \centering
    		\includegraphics[scale=0.5]{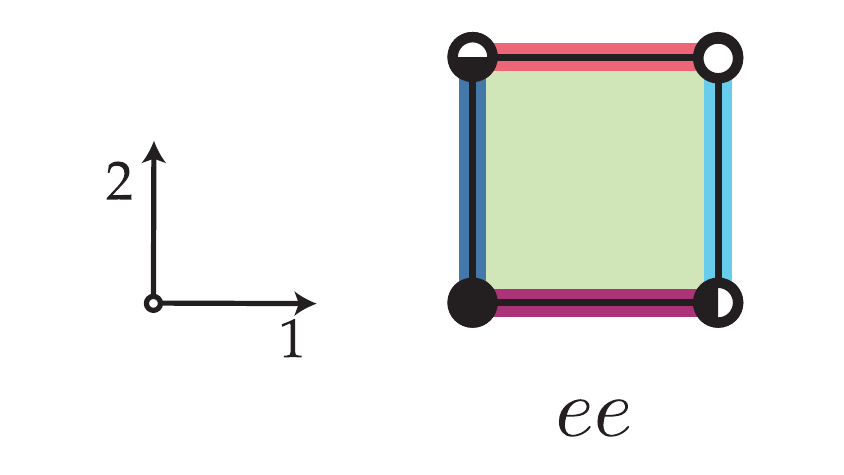}
    \end{figure}
    
    Once again, we can read off the faces of from the visualization as
    \begin{align*}
        \delta_1^0(ee) = 0e, \quad \delta_1^1(ee) = 1e, \quad \delta_2^0(ee) = e0, \quad \delta_2^1(ee) = e1.
    \end{align*}
    Symbolically, $\delta_i^\epsilon$ changes the $\text{i}^{th}$ $e$ into $\epsilon$. In addition to the non-degenerate cubes, $Z^2_1$ and $Z^2_2$ also contain degenerate cubes. In $Z^2_1$, we have one degeneracy for each $0$-cube, where we omit colors for degenerate cubes.
    \begin{figure}[hbt!]
    \centering
    		\includegraphics[scale=0.5]{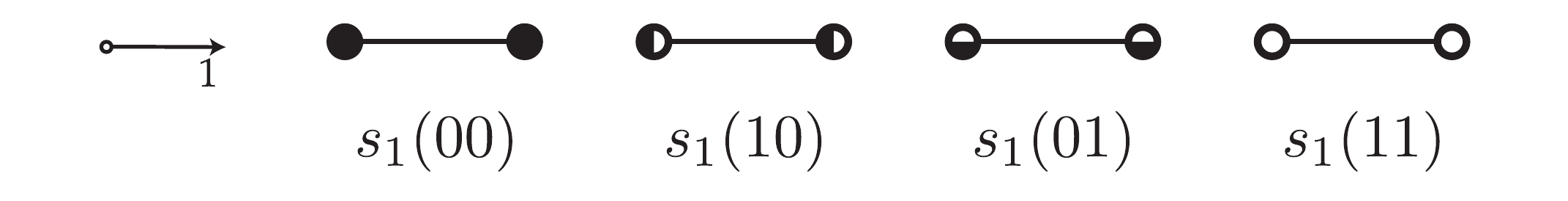}
    \end{figure}
    
    In $K^2_2$, we have one degeneracy for each $0$-cube, and three degeneracies for each non-degenerate $1$-cube; one example of each is shown below.
    \begin{figure}[hbt!]
    \centering
    		\includegraphics[scale=0.5]{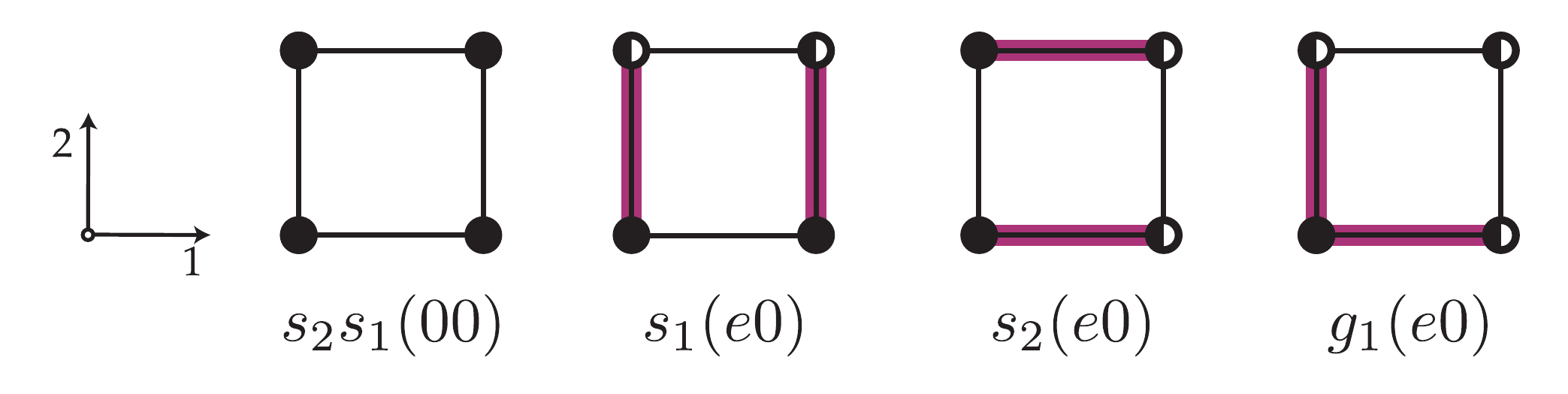}
    \end{figure}
    Note that the face maps of these degeneracies are determined by the cubical identities. There are only non-degenerate cubes in dimension $p \leq 2$, and thus, all cubes in $Z^2_p$, where $p > 2$, are degenerate and can be formed in a similar manner. 
    
    The geometric realization this cubical set can be visualized using the figure of the non-degenerate $2$-cube above: the disjoint union of geometric $p$-cubes is identified along the face maps. Furthermore, all degeneracies will be identified with a unique non-degenerate cube. 
\end{example}

\subsection{Cubical Mapping Space Construction}
Our version of Chen's construction for mapping spaces is based on reformulating the simplicial construction in~\cite{patras_cochain_2003, ginot_chen_2010} to the setting of cubical sets. In order to discuss differential forms on the smooth mapping space, we must first equip it with a smooth structure.

\begin{definition}
\label{def:simplicial_smooth_structure}
    Let $p \geq 0$. The $p$-cube $\square^p$ equipped with the \emph{simplicial smooth structure} is a differentiable space\footnote{This is proved to be a differentiable space by equipping each subsimplex $\Delta^p_\sigma$ with the subspace smooth structure and using the fact that unions and quotients of differentiable spaces are differentiable~\cite{baez_convenient_2011}.} characterized as follows: a map $f: \square^p \rightarrow X$, where $X$ is a manifold, is smooth (in the differentiable space sense) if the restriction to permuted $m$-simplex $f|_{\Delta^p_\sigma}: \Delta^p_\sigma \rightarrow X$ is smooth (in the manifold sense) for all $\sigma \in \Sigma_p$. 
\end{definition}

Throughout this appendix, we will always assume that the unit cubes $\square^p$ are differentiable spaces equipped with the simplicial smooth structure. In particular, the coface, codegeneracy and coconnection maps in~\Cref{def:cocubical_maps} are smooth. With this smooth structure, the geometric realization defined in~\Cref{eq:geometric_realization} is a differentiable space, since the category is complete and cocomplete~\cite{baez_convenient_2011}. Furthermore, smooth mapping spaces between differentiable spaces, such as $C^\infty(|K|,X)$, is also a differentiable space.

\begin{definition}
Let $K_\bullet$ be a finite cubical set, $\#K_p$ be the cardinality of $K_p$, and $A$ be a differential graded algebra. The \emph{cubical Hochschild complex of $A$ over $K_\bullet$} is a differential graded algebra $CH_\bullet^{K_\bullet}(A)$, where the degree $q$ component is given by
\begin{align*}
\label{eq:cubical_hochschild_def}
    CH_q^{K_\bullet}(A) \coloneqq \bigoplus_{p\geq 0} \left(A^{\otimes \#K_p}\right)_{p+q},
\end{align*}
where the differential maps are defined using the face maps of $K_\bullet$.
\end{definition}

The cubical Chen mapping space construction is defined as a map
\begin{align}
    \int: CH_\bullet^{K_\bullet}(\Omega^\bullet(X)) \rightarrow \Omega^\bullet(\smmapKman).
\end{align}
For fixed $p \in \N$, an explicit definition can be given as the composition of the following two maps
\begin{align}
\label{eq:general_cubical_chen_construction}
    \left(\Omega^\bullet(X)\right)^{\otimes \#K_p} \xrightarrow{\ev_{K_\bullet, p}^*} \Omega^\bullet(\square^p \times \smmapKman) \xrightarrow{\int_{\square^p}} \Omega^{\bullet-p}(\smmapKman).
\end{align}
In particular, the degree $p$ construction is performed in three steps:
\begin{enumerate}
    \item For each $a \in K_p$, select differential forms $\omega_a \in \Omega^{q_a}(X)$ and set
    \begin{align}
        \bomega = \bigotimes_{a \in K_p} \omega_a \in \left(\Omega^\bullet(X)\right)^{\otimes \#K_p}.
    \end{align}
    We let $q = \sum_{a \in K_p} q_a$.
    
    \item Define the \emph{degree $p$ evaluation map with respect to $K_\bullet$} by
    \begin{align}
    \label{eq:ev_Kbullet_def}
        \ev_{K_\bullet, p} : \square^p \times \smmapKman \rightarrow X^{\#K_p}\\
        (\bs, \bx) \mapsto (\bx \circ \eta_a(\bs))_{a \in K_p},\nonumber
    \end{align}
    and compute the pullback of $\bomega$ along this map to get
    \begin{align*}
        \ev_{K_\bullet,p}^* \bomega(\bs, \bx) = \bigwedge_{a\in K_p} (\bx \circ \eta_a)^*\omega_a(\bs).
    \end{align*}
    \item Integrate the resulting differential form along $\square^p$ to obtain a $q-p$ form on $\smmapKman$ to obtain a \emph{Chen $(q-p)$-form},
    \begin{align*}
        \int \bomega \coloneqq \int_{\square^p} \bigwedge_{a \in K_p} (\bx \circ \eta_a)^* \omega_a(\bs).
    \end{align*}
\end{enumerate}

\subsection{Mapping Space \texorpdfstring{$0$}{0}-Cochains on \texorpdfstring{$\smmapR$}{smmapR}}
\label{ssec:ms_0cochains_Rn}

We now focus our attention on the $0$-cochains in the case where $X = \R^\coddim$ in order to recover the Chen $0$-cochains discussed in~\Cref{sec:chen_construction} which is then used as the definition of mapping space monomials in~\Cref{def:mapping_space_monomial}. We begin by describing a cubical set $Z^\domdim_\bullet$ which models the unit $\domdim$-cube $\square^\domdim$. The mapping space monomials are derived from forms which only use the degeneracies of the top dimensional non-degenerate cube in $Z^\domdim_\bullet$, and we provide an equivalence between such $p$-cubes and \emph{$\domdim$-ordered subsets of $[p]$}. \medskip

The \emph{standard cubical model for the $\domdim$-cube}, which we denote by $Z_\bullet^\domdim\label{eq:Zbullet_def}$ (or simply by $Z_\bullet$), is the direct generalization of the cubical set from~\Cref{ex:2cube}. The collection of non-degenerate cubes is indexed by $E^\domdim$, where $E \coloneqq \{0, 1, e\}$. Given a word $a= a_1 \ldots a_\domdim \in E^\domdim$, the \emph{dimension} of the corresponding cube is the number of instances of $e$. Given a non-degenerate $p$-cube $a= a_1 \ldots a_p \in E^\domdim$, and $i \in [p]$, suppose $j \in [\domdim]$ be the $\mathrm{i}^{th}$ instance of $e$ in $a$. Then, the face maps for the non-degenerate cube are defined by
\begin{align*}
    \delta_i^\epsilon = a_1 \ldots a_{j-1} \epsilon a_{j+1} \ldots a_\domdim.
\end{align*}
The remaining cubes are degeneracies of the non-degenerate cubes, and can be written using the form given in~\Cref{eq:canonical_degenerate_form}. Suppose $e^\domdim \in Z_\domdim$ is the unique nondegenerate $\domdim$-cube, and let $\widetilde{Z}_p$ be the set of degeneracies of $e^\domdim$ of degree $p$; in other words, any $a \in \widetilde{Z}_p$ is of the form given in~\Cref{eq:canonical_degenerate_form}, with $b = e^\domdim$. We can more easily describe elements of $\widetilde{Z}_p$ using \emph{$\domdim$-ordered subsets}.

\begin{definition}
\label{def:d_ordered_subset}
    Let $\domdim, p \in \N$ with $\domdim \leq p$. A \emph{$\domdim$-ordered subset of $[p]$}, denoted by $\bI = (I^{(1)}, \ldots, I^{(\domdim)})$, is a collection of $\domdim$ nonempty disjoint subsets $I^{(r)} \subset [p]$, for $r \in [\domdim]$, such that if $r < s$, $x \in I^{(r)}$, and $y \in I^{(s)}$, then $x < y$. Let $\ordsubset{\domdim}{p}$ be the set of $\domdim$-ordered subsets of $[p]$.
\end{definition}

\begin{lemma}
\label{lem:degeneracy_ordered_subsets}
    Let $Z_\bullet$ be the standard cubical model for $\square^\domdim$, and $\widetilde{Z}_\bullet$ be the degeneracies of the unique nondegenerate $\domdim$-cube $e^\domdim \in Z_\domdim$. There is a bijection between $\widetilde{Z}_p$ and $\ordsubset{\domdim}{p}$. Thus, each degeneracy $a \in \widetilde{Z}_p$ will be denoted by $a_\bI$ for some $\bI \in \ordsubset{\domdim}{p}$. Furthermore, the evaluation map corresponding to $a_{\bI}$, denoted $\eta_\bI \coloneqq \eta_{a_\bI}: \square^p \rightarrow \square^\domdim$, is given by 
    \begin{align}
    \label{eq:evaluation_map_I}
        \eta_{\bI}(s_1, \ldots, s_p) = \left( \max_{i \in I^{(1)}} \{s_i\}, \ldots, \max_{i \in I^{(d)}}\{s_i\}\right).
    \end{align}
\end{lemma}
\begin{proof}
Suppose $p \geq \domdim$ and $a \in \widetilde{Z}_p$, which we represent in canonical form as in~\Cref{eq:canonical_degenerate_form}, where $b = e^\domdim$. Based on the conditions of the degeneracy and connection indices in~\Cref{eq:canonical_form_indices}, we can represent $a$ using subsets
\begin{align}
    \oI &= (i_1 < \ldots < i_k) \subset [p]. \label{eq:oI_oJ}\\
    \oJ &= (j_1 < \ldots < j_{p-k-\domdim}) \subset [p-k-1]. \nonumber
\end{align}
In particular, we write $a = (\oI, \oJ)$. Note that the total number of degeneracies is fixed at $\#\oI + \#\oJ = p-\domdim$ since the dimension of the nondegenerate cube $e^\domdim$ is fixed.

We will prove this lemma by finding some $\bI \in \ordsubset{\domdim}{p}$ such that the evaluation map of $a\in \widetilde{Z}_p$ coincides with~\Cref{eq:evaluation_map_I}. We use the definition of the evaluation map (\Cref{def:evaluation_map_a}) as a quotient with respect to the cubical equivalence relations~\Cref{eq:geometric_realization_equivalence_relation},  and the definitions of the codegeneracy and coconnection maps in~\Cref{eq:codegeneracy} and~\Cref{eq:coconnection} to find such a $\bI$.

We begin by considering the case where $\oI = \emptyset$ so that the complement is $I = [n]$. In this case, we have $\oJ = (j_1, \ldots, j_m) \subset [p-1]$, where $m = p - d$, and we need to consider the iterated coconnection map
\begin{align*}
	\eta_{(\emptyset, \oJ)} = \gamma_{j_1} \ldots \gamma_{j_m}.
\end{align*}
A consecutive subsequence in $\oJ$ corresponds to consecutive coconnection maps, which can be written as
\begin{align*}
	\gamma_{j-l} \gamma_{j-l+1} \ldots \gamma_j(s_1, \ldots, s_p) &= \gamma_{j-l} \ldots \gamma_{j-1}(s_1, \ldots, \max\{s_j, s_{j+1}\}, \ldots, s_p)\\
	& =  \gamma_{j-l}\ldots \gamma_{j-2}( s_1, \ldots, \max\{s_{j-1}, s_j, s_{j+1}), \ldots, s_p\} \\
	& = (s_1, \ldots, \max\{s_{j-l}, \ldots, s_{j+1}\}, \ldots, s_p).
\end{align*}

We now begin our construction of $\bI$. A maximal consecutive sequence $(j-l, \ldots, j) \subset \oJ$ corresponds to a subset $I^{(r)} = (j-l, \ldots, j, j+1)$. For each maximal consecutive sequence in $\oJ$ (including sequences of length 1), we add the corresponding subset of $I$ to $\bI$ in ascending order. Once every element of $\oJ$ has been accounted for, we add the remaining elements of $I$ into $\bI$ as singletons. Note that this procedure will provide a partition of $I$ made up of exactly $\domdim$ subsets.

Next, we consider the case where $\oI =(i_1, \ldots, i_p)$, so we must consider the coconnections as well as the codegeneracies. However, note that codegeneracies simply omit coordinates, so that
\begin{align*}
	\sigma_{i_1} \ldots \sigma_{i_p}(s_1, \ldots, s_p) = (s_j)_{j \in I},
\end{align*}
where $I$ is the complement of $\oI$ in $[p]$. Now, we can simply repeat the previous construction of the partition with a smaller set $I$.

We have shown that each $\oI$ and $\oJ$, as defined in~\Cref{eq:oI_oJ} corresponds to a distinct  $\bI \in \ordsubset{\domdim}{p}$. Furthermore, we can reverse the construction here to obtain a pair $(\oI, \oJ)$ from some $\bI \in \ordsubset{\domdim}{p}$. Therefore, this provides the desired bijection.
\end{proof}

By restricting the cubical Chen mapping space construction (\Cref{eq:general_cubical_chen_construction}) to these degeneracies, we obtain the construction discussed in~\Cref{sec:chen_construction}. In particular, the construction in~\Cref{eq:restricted_cubical_chen_construction} builds $0$-cochains using the degenerate cubes in $\widetilde{Z}_\bullet$.

\begin{example}
\label{ex:0_cochain}
    Here, we provide an explicit example of a Chen $0$-cochain from the construction in~\Cref{eq:restricted_cubical_chen_construction} for $\domdim=2$ and $p=4$. We define $\bI_1, \bI_2 \in \ordsubset{2}{4}$ by
    \begin{align*}
        I_1^{(1)} = \{1\}, &\quad I_1^{(2)} = \{3\} \\
        I_2^{(1)} = \{1,2\}, &\quad I_2^{(2)} = \{3,4\},
    \end{align*}
    and we let
    \begin{align*}
        \omega_1 &= dx_{P_1} = dx_{P_1(1)} \swedge dx_{P_1(2)}, \quad \quad \omega_2 = dx_{P_2} = dx_{P_2(1)} \swedge dx_{P_2(2)},
    \end{align*}
    be standard $2$-forms in $\R^n$, where $P_1, P_2 \in \cO_{2,\coddim}$. Given $\bx = (x_1, \ldots, x_n) \in \smmapR$, we have
    \begin{align*}
        \bx^*\omega_1 = J[\bx_{P_1}](s_1, s_2) ds_1 \swedge ds_2, \quad \quad \bx^*\omega_2 = J[\bx_{P_2}](s_1, s_2) ds_1 \swedge ds_2,
    \end{align*}
    where $\bx_{P_i} = (x_{P_i(1)}, x_{P_i(2)}) : \square^2 \rightarrow \R^2$, and $J[\bx_{P_i}](s_1, s_2)$ is the determinant of the Jacobian of $\bx_{P_i}$ at $(s_1, s_2) \in \square^2$. Next, the evaluation maps $\eta_{\bI_i}: \square^4 \rightarrow \square^2$ are
    \begin{align*}
        \eta_{\bI_1}(\bt) &= (t_1, t_3) \\
        \eta_{\bI_2}(\bt) & = \left( \max\{t_1, t_2\}, \max\{t_3, t_4\}\right).
    \end{align*}
    The corresponding pullbacks are
    \begin{align*}
        \left(\bx\circ \eta_{\bI_1}\right)^* \omega_1(\bt) & = J[\bx_{P_1}](t_1, t_3) dt_1 \swedge dt_3 \\
        \left(\bx\circ \eta_{\bI_2}\right)^* \omega_2(\bt) & = \left \{
        \begin{array}{ll}
        J[\bx_{P_2}](t_1, t_3) dt_1 \swedge dt_3 &: t_1 > t_2, \, t_3 > t_4 \\
        J[\bx_{P_2}](t_2, t_3) dt_2 \swedge dt_3 &: t_2 > t_1, \, t_3 > t_4 \\
        J[\bx_{P_2}](t_1, t_4) dt_1 \swedge dt_4 &: t_1 > t_2, \, t_4 > t_3 \\
        J[\bx_{P_2}](t_2, t_4) dt_2 \swedge dt_4 &: t_2 > t_1, \, t_4 > t_3
        \end{array}
        \right.
    \end{align*}
    Due to the alternating condition that $dt_i \swedge dt_i = 0$, the only region with a nontrivial wedge product of these two pullbacks is 
    \begin{align*}
        \{0 \leq t_1 < t_2 \leq 1\} \times \{0 \leq t_3 < t_4\leq 1\} = \Delta^2 \times \Delta^2 \subset \square^4.
    \end{align*}
    Finally, the explicit form of the Chen $0$-cochain evaluated on $\bx$ is
    \begin{align*}
        \int_{\Delta^2 \times \Delta^2} J[\bx_{P_1}](t_1, t_3) \cdot  J[\bx_{P_2}](t_2, t_4) \, dt_1 \swedge dt_3 \swedge dt_2 \swedge dt_4. 
    \end{align*}
\end{example}

There is a much larger variety of Chen $0$-cochains for the mapping space $\smmapR$ than the path space $C^\infty(\square^1, \R^\coddim)$ due to the existence of many more degeneracies of the top dimensional cube (since $1$-ordered subsets of $[p]$ are equivalent to elements of $[p]$), as well as the flexibility to choose differential forms up to degree $\domdim$. A detailed study of the variety of Chen $0$-cochains for mapping spaces is provided in~\cite{lee_mapping_2021}. In this article, we have restricted our attention to a certain class of $0$-cochains, which generalizes the path space $0$-cochains which make up the path signature. In particular, all nontrivial differential forms $\omega_i$ will be $\domdim$-forms, thus restricting the total degree to be $p = \domdim m$. Additionally,
\begin{enumerate}
    \item the differential forms $\omega_1, \ldots, \omega_m$ are standard $\domdim$-forms on $\R^\coddim$, given by
    \begin{align*}
        \omega_i = dx_{P_i} = dx_{P_i(1)} \swedge \ldots \swedge dx_{P_i(\domdim)},
    \end{align*}
    where $P_i : [\domdim] \rightarrow [\coddim]$ is an order-preserving injection, and
    \item the collection $(\bI_1, \ldots, \bI_m) \subset \ordsubset{\domdim}{\domdim m}$ of $\domdim$-ordered subsets of $[\domdim m]$ is represented by a permutation $\bpi = (\pi_1, \ldots, \pi_\domdim) \in \Sigma_m^\domdim$, where
    \begin{align}
    \label{eq:restricted_d_ord_subsets}
        I_i^{(j)} = \{(j-1)m + \pi_j(1), (j-1)m + \pi_j(2), \ldots, (j-1)m + \pi_j(i)\}.
    \end{align}
\end{enumerate}
By following the same reasoning as~\Cref{ex:0_cochain}, we can obtain an explicit form for the Chen $0$-cochain. Let $\cP = (P_1, \ldots, P_m)$ denote the collection of standard $\domdim$-forms. By renaming the coordinates of $\square^{\domdim m}$ by
\begin{align*}
    t_{i,j} \coloneqq t_{(j-1)m+i},
\end{align*}
where $i \in [m]$ and $j \in [\domdim]$, the resulting $0$-cochain, which is the \emph{mapping space monomial of $\bx$ with respect to $(\cP, \bpi)$} (\Cref{def:mapping_space_monomial}), is
\begin{align*}
    \Phi^{\cP, \bpi}_m(\bx) \coloneqq \int_{\intdom{m}{\domdim}{\bpi}} \prod_{i=1}^m J[\bx_{P_i}](\bt_i) d\bt.
\end{align*}.

\section{Recursive Definition of the Identity Signature}
\label{apx:recursive_definition}
An important property of the classical path signature is that can be defined recursively, that is the $(m+1)$-th iterated integral is given by integrating the $m$-th iterated integral.
For the general case $d\ge 2$ more care is needed since the integration domain is not just a simplex, but here we show that a similar recursive definition holds where the permutation index is fixed to be the identity.

\begin{lemma}
\label{lem:recursive_definition}
    Let $m \in \N$, $(\cP, \bid) \in \cO_{\domdim,\coddim}^m \times \Sigma_m^\domdim$ be a level $m$ index with the identity permutation index, and $(\ba, \bb) \in (\Delta^2)^d$. For $\cP = (P_1, \ldots, P_m)$, define $\overline{\cP} = (P_1, \ldots, P_{m-1})$. Then,
    \begin{align}
        \Phi^{\cP, \bid}_m(\bx)_{\ba, \bb} = \int_{\square^{\domdim}(\ba, \bb)} \Phi^{\overline{\cP}, \id}_{m-1}(\bx)_{\ba, \bs} \cdot J[\bx_{P_m}](\bs) d\bs.
    \end{align}
\end{lemma}
\begin{proof}
    This can be proved by expanding the definition of $\Phi_{m-1}^{\overline{\cP},\id}(\bx)_{\ba,\bs}$ and an application of Fubini's theorem.
\end{proof}

\begin{corollary}
	For $(\ba, \bb) \in (\Delta^2)^d$ and $\bx \in \lipmap$, let
	\begin{align}
	\label{eq:identity_mapping_space_signature}
		\Phi^{\bid}(\bx)_{\ba, \bb} \coloneqq 1 + \sum_{m=1}^\infty \Phi_{m}^{\bid}(\bx)_{\ba, \bb} \in \prod_{m=0}^\infty \left(\Lambda^\domdim V\right)^{\otimes m} \eqqcolon \idpspvspace{\domdim}{}{V}.
	\end{align}
	Then,
	\begin{align}
		\Phi^{\bid}(\bx) = \int_{\square^{\domdim}} \Phi^{\bid}(\bx)_{\bzero,\bs} \otimes J[\bx](\bs) d\bs.
	\end{align}
\end{corollary}
\begin{proof}
	This is immediate from~\Cref{lem:recursive_definition}.
\end{proof}

This recursive definition when the permutation index is restricted to the identity shows that the restricted mapping space signature given in~\Cref{eq:identity_mapping_space_signature} can be written as the solution to the differential equation
\begin{align}
    \frac{\partial^d}{\partial s_1 \ldots \partial s_{\domdim}}\Phi^{\bid}(\bx)_{\bzero, \bs} = \Phi^{\id}(\bx)_{\bzero, \bs} \otimes J[\bx](\bs),
\end{align}
with the boundary condition $\Phi^{\bid}(\bx)_{\bzero, \bs} = 1$ whenever $s_j = 0$ for some $j \in [\domdim]$, resembling the differential equation formulation of the path signature. \medskip

    While the recursive definition of the mapping space signature holds for the identity permutation index, it is natural to ask whether we can achieve a similar formulation for an arbitrary permutation index. When the permutation index is repeated copies of the same permutation such as $\bpi = (\pi, \pi, \ldots, \pi) \in \Sigma_m^{\domdim}$ for some $\pi \in \Sigma_m$, permutation invariance in ~\Cref{prop:permutation_invariance} shows that this is equivalent to the setting of the identity permutation. Thus, we consider permutation indices where there are at least two permutations which are distinct. In this setting, we find that a direct generalization is not possible.

\begin{example}
Consider the case where $d = 2$ and $m= 4$, and let $(\cP, \bpi) \in \cO_{2,\coddim}^4 \times \Sigma_4^2$, where we set the permutation index to be $\bpi = (\pi_1, \pi_2)$ with
\begin{align*}
    \pi_1 = (1~2~3~4), \quad \pi_2 = (3~1~4~2).
\end{align*}
The corresponding level $4$ monomial is
\begin{align*}
    \Phi_4^{\cP, \bpi}(\bx) = \int_{\intdom{m}{\domdim}{\bpi}} J[\bx_{P_1}](t_{1,1}, t_{1,2})\cdot J[\bx_{P_2}](t_{2,1}, t_{2,2})\cdot J[\bx_{P_3}](t_{3,1}, t_{3,2})\cdot J[\bx_{P_4}](t_{4,1}, t_{4,2})\, d\bt,
\end{align*}
where the coordinates have the relations
\begin{align*}      
    0 &\leq t_{1,1} < t_{2,1} < t_{3,1} < t_{4,1} \leq 1 \\
    0 &\leq t_{3,2} < t_{1,2} < t_{4,2} < t_{2,2} \leq 1.
\end{align*}
The goal is to isolate the $J[\bx_{P_4}](\bt_4)$ term and integrate out all variables other than the $\bt_4$ term. However, this integral would be
\begin{align*}
    \int_{\Delta^3(0, t_{4,1}) \times \left(\Delta_{(3 \, 1)}^2(0, t_{4,2}) \times \Delta^1_{( 2)}(t_{4,2},1)\right)}J[\bx_{P_1}](t_{1,1}, t_{1,2})\cdot J[\bx_{P_2}](t_{2,1}, t_{2,2})\cdot J[\bx_{P_3}](t_{3,1}, t_{3,2}) \, d\bt,
\end{align*}
which is not of the form of a mapping space monomial due to the domain of integration. 
\end{example}

\bibliographystyle{plain}
\bibliography{mappingspace}

\end{document}